\documentclass[a4paper,12pt]{amsart}
\usepackage[utf8]{inputenc}
\usepackage{geometry}
\usepackage{amsmath}
\usepackage{amsfonts,amssymb}
\usepackage{appendix}
\usepackage{tikz}
\usepackage{epsfig}
\usepackage{amssymb}
\usepackage{mathrsfs}
\usepackage{amscd} 
\usepackage[all]{xy}
\usetikzlibrary{cd}

\newtheorem{theorem}{Theorem}[section]
\newtheorem{lemma}[theorem]{Lemma}
\newtheorem{proposition}[theorem]{Proposition}

\theoremstyle{definition}
\newtheorem{definition}[theorem]{Definition}

\newtheorem{remark}[theorem]{Remark}

\newcommand{\vol}{\mathrm{Vol}}
\newcommand{\sign}{\mathrm{sign}}
\newcommand{\PD}{\mathrm{PD}}
\newcommand{\Ind}{\mathrm{Ind}}
\newcommand{\sgn}{\mathrm{sgn}}
\newcommand{\GL}{\mathrm{GL}}
\newcommand{\End}{\mathrm{End}}
\newcommand{\Hom}{\mathrm{Hom}}
\newcommand{\rank}{\mathrm{rank}}

\newcommand{\R}{\mathbb{R}}

\newcommand{\Ker}{\text{Ker}}

\newcommand{\Id}{\text{Id}}

\newcommand{\ch}{\mathrm{ch}}
\newcommand{\diff}{\mathrm{diff}}
\newcommand{\op}{\mathrm{Op}}
\newcommand{\im}{\mathrm{Im}}
\newcommand{\Str}{\mathrm{Str}}
\newcommand{\Pf}{\mathrm{Pf}}
\newcommand{\chg}{\mathrm{ch_{G}}}
\newcommand{\supp}{\mathrm{supp}}
\newcommand{\Tr}{\mathrm{Tr}}
\newcommand{\spin}{\mathrm{Spin}}
\newcommand{\tr}{\mathrm{tr}}
\newcommand{\ev}{\mathrm{ev}}
\begin{document}

\title{Index theorem for homogeneous spaces of Lie groups}
\date{}

\author[Hang Wang]{Hang Wang} 
\address{\normalfont{Research Center for Operator Algebras, School of Mathematical Sciences\\
East China Normal University, 3663 North Zhongshan Rd,\\
Shanghai 200062, P.R.China}}
\email{wanghang@math.ecnu.edu.cn}

\author[Zijing Wang]{Zijing Wang}
\address{\normalfont{Research Center for Operator Algebras, School of Mathematical Sciences\\
East China Normal University, 3663 North Zhongshan Rd,\\
Shanghai 200062, P.R.China}}
\email{52205500014@stu.ecnu.edu.cn}

\begin{abstract}
We study index theory on homogeneous spaces associated to an almost connected Lie group in terms of the topological aspect and the analytic aspect. On the topological aspect, we obtain a topological formula as a result of the Riemann-Roch formula for proper cocompact actions of the Lie group, inspired by the work of Paradan and Vergne. On the analytic aspect, we apply heat kernel methods to obtain a local index formula representing the higher indices of equivariant elliptic operators with respect to a proper cocompact actions of the Lie group.  
\end{abstract}
\dedicatory{Dedicated to  Professor Mich\`ele Vergne}
\thanks{H. Wang gratefully acknowledge the support from the National Natural Science Foundation of China (No. 12271165) and from Science and Technology Commission of Shanghai Municipality (No.
22DZ2229014)}

\subjclass{Primary 19K56; Secondary: 58B34 58J20, 46L80}
\keywords{Riemann-Roch formula, local index formula, heat kernel method, higher index theory, homogeneous space, Dirac operator}
\maketitle

\tableofcontents
\section{Introduction}

In this paper, we investigate the higher index theory of an equivariant elliptic operator on a complete Riemannian manifold equipped with a proper, cocompact action of an almost connected Lie group. 
Motivated by the classic Atiyah-Singer index theorem~\cite{ref1}, our study consists both the topological and analytic perspectives.  
The two perspectives are believed to be encoded in the Baum-Connes assembly map, formulated in \cite{Baum-Connes, Baum-Connes-Higson}. 
For example, in the case of a \emph{discrete} 
group, the left-hand side of the Baum-Connes assembly map, known as the topological $K$-theory of the group, can be identified with its Chern character whose formulation involves a version of Riemann-Roch formula and the formula of which gives rise to a local formula for all equivariant elliptic operators with respect to a proper cocompact action~\cite{CRWW}. On the other hand, one can formulate suitable pairings with the higher index in right hand side of the Baum-Connes assembly map and obtain local index formula using geometric and analytic approaches, see~\cite{ref10} for example. 
The aim of the paper is to understand index theory for almost connected Lie groups, both from the topological perspective via a Riemann-Roch formula and from the analytic perspective via heat kernel methods . 

The Riemann-Roch theorem was born in complex analysis and algebraic geometry, in the process of computing the dimension of the space of meromorphic functions with prescribed zeros and allowed poles. In 1954, Hirzebruch generalized the classical Riemann-Roch theorem from Riemann surfaces to all complex algebraic varieties of higher dimensions.  
Three years later, Grothendieck generalized the Hirzebruch-Riemann-Roch theorem with respect to a map $f: Y\rightarrow X$ of algebraic varieties. 
This so-called Grothendieck-Riemann-Roch theorem reduces to the previous Riemann-Roch Theorem by taking $X$ to be a point.
Grothendieck's theorem requires many conditions on characteristic classes, and so it is natural to ask if these conditions hold for almost complex or, more generally, for differentiable manifolds. 
This question has been solved by Atiyah and Hirzebruch in \cite{ah}. 
As an application of the Riemann-Roch theorem for differentiable manifolds, Atiyah and Singer established their Riemann-Roch formula in \cite{ref2}. 
At that time, this formula was at the level of cohomology. Subsequently, Mathai and Quillen gave a beautiful proof of the Riemann-Roch formula in \cite{ref3} using the powerful tools of superconnections. This work transforms the formula to the level of differential forms. Based on that, Paradan and Vergne defined Quillen's relative Chern character and obtained the Riemann-Roch formula on the level of relative cohomology in \cite{ref5} and generalized to the equivariant case for compact group actions in \cite{ref6}.

In this paper, we formulate the equivariant Chern character and the Riemann-Roch formula in the context of proper actions by an almost-connected Lie group. This is the first main result of the paper. 

\begin{theorem}[Theorem \ref{err}]
Let $G$ be an almost-connected Lie group, $M$ be a proper $G$-compact space satisfying conditions in Theorem~\ref{sk}, and $V$ be a $G$-vector bundle (even rank) over $M$ with an equivariant spin structure. The following diagram commutes:
    \begin{equation*}
\xymatrix{   K^{0}_{G}(V)\ar[r]^{\pi!}\ar[d]_{\chg} &   K^{0}_{G}(M)  \ar[d]^{(-2\pi i)^{n/2}\cdot\chg\wedge\Hat{A}_{g}^{-1}(V) }  \\
           H^{\infty}_{cv}(\mathfrak{g},V) \ar[r]^-{\pi!} & H^{\infty}(\mathfrak{g},M),
                      }
\end{equation*}
where $\Hat{A}_{g}(V)^{-1}$ is the inverse of the equivariant $\hat{A}$-genus in Definition~\ref{Agenus}, $\chg$ on the left is defined in (\ref{ch2}), $\chg$ on the right is defined in (\ref{ch1}) and $\pi!:H_{cv}^{\infty}(\mathfrak{g},V)\rightarrow H^{\infty}(\mathfrak{g},M)$ is the homomorphism induced by integration over the fiber.
\end{theorem}

The highlight of this theorem is the introduction of the Chern character map associated to a smooth version of the equivariant $K$-theory in the sense of Phillips~\cite{ref7}, and the formulation of the Riemann-Roch formula using techniques from the Mathai-Quillen formalism. Both constructions are inspired by the work of Paradan and Vergne in the case of a compact group action~\cite{ref6}.

Moving to the analytic perspective, heat kernel methods play a role in obtaining the cohomology formula. Another goal of this paper is to use heat kernel methods to obtain local index formulas for the $C^{*}$-higher indices of equivariant Dirac-type operators.

Consider a Lie group $G$, acting properly on a manifold $M$, with compact quotient. Let $D$ be a $G$-equivariant Dirac operator on a $\mathbb{Z}_{2}$-graded, $G$-equivariant vector bundle over $M$. 
In the special case when $G$ is discrete and acts freely and properly, the quotient $N$ of $M$ by $G$ is a closed manifold. Thus, $G$ is the fundamental group of $N$ with Deck transformation on the universal cover $M$. In their seminal paper \cite{ref4}, Connes and Moscovici extended the $L^2$-index theorem of Atiyah \cite{at} from the $G$-trace case to higher $G$-cocycles. Using  Getzler’s symbolic calculus, they found the topological expression of the localized analytic indices of elliptic operators.  
In 2015,  Pflaum,  Posthuma, and Tang unified several famous equivariant index theorems in the framework of Lie groupoids in \cite{ppt2}, including the equivariant index formula in \cite{at} and \cite{ref4}. In the notation as above, if $G$ is almost connected and satisfies the rapid decay property, Piazza and Posthuma establish an index formula for the $C^{*}$-higher indices of a G-equivariant Dirac-type operator on $M$, building on the algebraic index theorem established in \cite{ppt2}.

The second main result of this paper is a local formula for the higher index of a Dirac type operator on a complete Riemannian manifold carrying a proper cocompact action of an almost connected Lie group, via heat kernel methods:
\begin{theorem}[Theorem \ref{mthm2}]
\label{thm:1.2}
    Let the dimension $n$ of $M$ be even. For every $f\in C^{2q}_{v,anti}(M)^{G}$ (c.f. Definition~\ref{def1}) and $\Ind_{t}(D)\in S^{u}_{G}(M,E)$ (c.f. Definition~\ref{def2}), we have 
    $$\lim_{t\rightarrow0}\tau(f)(\Ind_{t}(D))=\frac{(-1)^{n/2-q}}{(2\pi i)^{2q-n/2}}\frac{q!}{(2q)!}\int_{M}c_{0}\Hat{{A}}_{G}(M)\wedge\ch_{G}^{0}(\pi!([\pi^{*}(E^{+}),\pi^{*}(E^{-}),\sigma(D)]))\wedge \omega_{f}$$
    where $$\omega_{f}:=(d_{1}...d_{2q}f)|_{\Delta}.
$$
Here $d_{i}$ stands for the differential concerning the $i$-th variable of the function $f$ and $\Delta: M\rightarrow M^{\times (2q+1)}$ is the diagonal embedding.
\end{theorem}
This theorem implies the formula for $C^{*}$-higher indices of a $G$-equivariant Dirac-type operator on $M$. See Theorem~\ref{thm7.6}. 
Piazza and Posthuma obtained a similar formula in \cite{pp1}. See more details in Section~6.

The heat kernel proof of Theorem~\ref{thm:1.2} follows closely that of Connes and Moscovici~\cite{ref4} but with significant analytical complications due to the more general setting being considered. 
In \cite{ref4}, because of the assumption of free actions by a discrete group, it is sufficient to consider compactly supported functions $f$ in the index pairing, and hence the index pairing problem can be reduced to the case of a compact manifold. But in our setting, we need to consider equivariant $f$ with suitable growth conditions, so that in order to have a well-defined index pairing, a careful estimate on the Schwartz kernel of the Wasserman class representing the higher index $\Ind_{t}(D)$ is needed.   
We achieved this aim by making use of the spirit of the fantastic work \cite{gromv} by Cheeger, Gromov, and Taylor on the Schwartz kernel estimate regarding scalar Laplacian on complete manifolds.

In the last section of the paper, we combine the equivariant Riemann-Roch formula and the local formula for the higher index, and identify the topological index and analytic index in the special case of homogeneous spaces, that is:
\begin{theorem}[Theorem \ref{mth}]
The following diagram commutes:
\begin{equation*}
    \xymatrix{   K_{G}^{0}(T^{*}(G/H))\ar[r]^{\Ind_{A}}\ar[d]_{\Ind_{T}} &  K_{0}(C_{r}^{*}(G)) \ar[d]^{\tilde{\ch} }\\
         H^{even}_{DR}(G/H)^{G}\ar[r]^{PD}&H_{DR,even}(G/H)^{G}
        }
\end{equation*}
where $\tilde{\ch}:K_{0}(C^{*}_{r}(G))\rightarrow H_{DR,even}(G/H)^{G}$ is determined by
$$\tilde{\ch}([p]-[q])(\alpha):=\langle\alpha,[p]-[q]\rangle_{DR,G}$$
for every $\alpha\in  H^{even}_{DR}(G/H)^{G}$ and $[p]-[q]\in K_{0}(C^{*}_{r}(G))$. Here $G$ has finitely many connected components and satisfies the RD condition, $H$ is a maximal compact subgroup of $G$ and $G/H$ is of non-positive sectional curvature.

\end{theorem}
Although homogeneous spaces are examples of spaces with proper actions, they include the Riemannian symmetric space of noncompact type $G/K$, where $K$ is a maximal compact subgroup. The homogeneous space $G/K$ is known as a model for the universal example of proper actions by an almost connected group. 

This paper is organized as follows. In Section $2$, we recall the definitions of equivariant $K$-theory in \cite{ref7} and equivariant cohomology in \cite{ref7} and \cite{ref1}. Then, we introduce the smooth equivariant $K$ theory. In Section $3$, we define the equivariant Chern character and prove some basic properties of this homomorphism. In Section $4$, we deduce the equivariant Riemann-Roch formula. In Section $5$, we work on the special case of homogeneous spaces and reduce the equivariant Riemann-Roch formula to the $G$-invariant case. In Section $6$, we use heat kernel methods to prove the local index formula. In Section $7$, we show that following the local index formula proved in the previous section, we obtain an index formula for $C^{*}$-higher indices of a $G$-equivariant Dirac-type operator on $M$. In Section $8$, we show the coincidence between the topological index and analytic index when $M$ is a homogeneous space.
 
\section{Preliminaries}
\subsection{Equivariant $K$-theory for proper actions}
In this section, we recall some basic definitions of equivariant $K$ theory for proper action which can be found in \cite{ref7} and \cite{p1}, and define the smooth equivariant $K$ theory.

First, we recall the definition of Hilbert bundle and $G$-Hilbert bundle. Let $M$ be a $C^{\infty}$-manifold with a smooth action of a Lie group $G$. 

\begin{definition}[\cite{ref7} $G$-Hilbert bundle] A Hilbert bundle over $M$ is a locally trivial fiber bundle $E$ whose fiber is a Hilbert space.
A $G$-Hilbert bundle is a Hilbert bundle $E$ with a continuous linear isometric action of $G$ such that the projection $\pi: E\rightarrow{M}$ is equivariant.
\end{definition}

Next, we turn to the definition of the morphisms.
\begin{definition}[\cite{ref7}]
    Let $\pi_{E}:E\rightarrow M$ and $\pi_{F}:F\rightarrow M$ be $G$-Hilbert bundles over $M$. A morphism $t$ is a continuous function $t:E\rightarrow{F}$ which maps $E_{m}$ linearly to $F_{m}$ for every $m\in M$ where $E_{m}:=\pi_{E}^{-1}(m)\in E$ and $F_{m}:=\pi_{F}^{-1}(m)\in F$. We say that $t$ is equivariant if $t(g\xi)=g\cdot t(\xi)$, for $g\in G$ and $\xi \in E$.
\end{definition}

Then, we recall the definition of the $K$-cocycles.
\begin{definition}[\cite{p1}]
    A $K$-cocycle for $(G,M)$ is a triple $(E,F,t)$ where $E$ and $F$ are $G$-Hilbert bundles over $M$ and $t$ is an equivariant morphism from $E$ to $F$ which is inverible outside a $G$-compact subset of $M$. 
    Here $G$-compact means that $M/G$ is compact.
\end{definition}

 After that, we recall the definition of a finite dimensional $K$-cocycle which plays an important role in the definition of the equivariant $K$-theory $K^{0}_{G}(M)$ when $G$ is almost-connected and $M$ is a proper $G$-space.

\begin{definition}[Finite dimensional (smooth) $K$-cocycle]
A finite dimensional (smooth) $K$-cocycle for $(G, M)$ is a triple $( E, F, t)$, consisting of two finite dimensional (smooth) $G$-vector bundles $E$ and $F$ and an (a smooth) equivariant 
morphism $t:E\rightarrow F$  whose restriction to the complement of some $G$-compact subset of $M$ is an (a smooth) isomorphism. 
\end{definition}
Two finite-dimensional (smooth) $K$-cocycles $(E_{1},F_{1},\sigma_{1})$, $(E_{2},F_{2},\sigma_{2})$, are said to 
be equivalent, if there exist finite-dimensional (smooth) G-vector bundles $H_{1}$ and $H_{2}$ and (smooth) $G$-equivariant isomorphisms:
\begin{align*}
    a:E_{1}\oplus H_{1}\rightarrow E_{2}\oplus H_{2}\\
    b:F_{1}\oplus H_{1}\rightarrow F_{2}\oplus H_{2}
\end{align*}
such that $b^{-1}\circ ({\sigma_{2}\oplus \Id_{H_{2}}})\circ{a}=\sigma_{1}\oplus \Id_{H_{1}}$ outside a $G$-compact subset $B$ of $M$. Denote 
by $[ E, F, t]$  the equivalence class of a finite-dimensional (smooth) $K$-cocycle $(E, F,t)$ and define an addition on the set of equivalence classes of  finite-dimensional (smooth) $K$-cocycles  by:
$$[E_{0},F_{0},t_{0}]+[E_{1},F_{1},t_{1}]=[E_{0}\oplus{E_{1}},F_{0}\oplus{F_{1}},t_{0}\oplus{t_{1}}]$$
which makes the set of equivalence classes of finite-dimensional (smooth) $G$-vector bundles over $M$ a  semigroup.
\begin{definition}[equivariant (smooth) $K$ theory]
     Let $M$ be a $C^{\infty}$- manifold with an action of an almost-connected Lie group $G$ such that the action is proper and smooth. Then the equivariant (smooth) $K$ theory ($K^{0}_{G, sm}(M)$ ) $K^{0}_{G}(M)$ is the Grothendieck group of the equivalence class of finite dimensional (smooth) $K$-cocycle for $(G, M)$.
\end{definition}
\begin{remark}
    In \cite{ref7}, Phillips defined the equivariant $K$-theory for proper actions of second countable locally 
    compact groups on second countable locally compact spaces. This definition involves infinite-dimensional Hilbert bundles and Fredholm morphisms between them. In \cite{p1}, Phillips showed that finite-dimensional vector 
    bundles would suffice for proper actions of almost-connected groups. For this reason, we used the Grothendieck group of the equivalence classes of finite dimensional (smooth) $K$-cocycle for $(G, M)$ to define the equivariant (smooth) $K$-theory here.
\end{remark}

If $M$ is G-compact, then the morphism of the $K$-cocycle plays no role, so that $K_{G, sm}^{0}(M)$ is just the Grothendieck group of the semigroup of finite-dimensional smooth $G$-vector bundles over $M$. Moreover, if $\pi:V\rightarrow M$ is a smooth  $G$-vector bundles over $M$ such that $M$ is G-compact, $K_{G,sm}^{0}(V)$ consists of the  equivalence class of finite dimensional smooth $K$-cocycle for $(G, V)$. To prove the latter point, we need the following:
\begin{lemma}
    Let $\pi: V\rightarrow M$ be a smooth  $G$-vector bundle over $M$ such that $M$ is G-compact. Then for any $[E,F,\sigma]\in K_{G,sm}^{0}(V)$, there exist $[F,E,s]\in K_{G,sm}^{0}(V)$ such that $[E,F,\sigma]=-[F,E,s]$.
\end{lemma}
\begin{proof}
      When $M$ is a $G$-compact manifold and the group action is proper, every $G$-vector bundle $V$ over $M$ has a $G$-invariant Hermitian metrics $\langle, \rangle$.  For any $[E,F,\sigma]\in K_{G,sm}^{0}(V)$, since $\supp(\sigma):=\{x\in V|\sigma_{x}~is~not~invertible\}$ is $G$-compact, there exists $a>0$ such that $\supp(\sigma)\subset \{v|v\in V, \langle v,v \rangle \leq a\}$. 
      
      Thus, let $f$ be a smooth function over  $\mathbb{R}$ valued in $[0,1]$ satisfying $f(x)=1$ when $|x|\geq a+2$ and $f(x)=0$ when $|x|\leq a+1$. Set $\chi(v):=f(\langle v,v \rangle)$, and then $\chi \in C^{\infty}(V)^{G}$.
      
      Let $\tilde{s}: E\rightarrow F$ be a $G$-equivariant bundle homomorphism such that $\tilde{s}$ is equal to the inverse of $\sigma$ restricted to the complement of $\{v|v\in V, \langle v,v \rangle \leq a\}$. Here $\tilde{s}$ may not be smooth but $s=\chi\tilde{s}$ is a  smooth equivariant bundle homomorphism from $E$ to $F$.

      We claim that $[ E, F, \sigma] = -[ F, E, s ]$. This is equivalent to claim that:
      $$[E\oplus F, F\oplus E,\sigma\oplus s]=[E\oplus F,E\oplus F, \Id_{E\oplus F}]$$
      since the right-hand of the above equation is the identity element of $K_{G, sm}^{0}(V)$. To prove the claim above, it is sufficient to find smooth $G$-equivariant isomorphisms $a: E\oplus F\rightarrow E\oplus F$ and $b: E\oplus F\rightarrow F\oplus E$ such that $b^{-1}(\sigma\oplus s)a=\Id_{E\oplus F}$ outside a $G$-compact subset of $V$. Take $a=\Id_{E\oplus F}$ and 
      $$ b_{v}=\begin{pmatrix}
\cos{(\frac{\pi}{2}(1-\chi(v))}\sigma_{v}&\sin(\frac{\pi}{2}(1-\chi(v))\Id_{F_{v}}\\
-\sin(\frac{\pi}{2}(1-\chi(v)))\Id_{E_{v}}&\cos{(\frac{\pi}{2}(1-\chi(v))}\sigma_{v}^{-1}
\end{pmatrix}
      $$
      for any $v\in V$. Therefore, We have $[ E, F, \sigma] = -[ F, E, s ]$ which completes the proof.
\end{proof}

Now we discuss the relationship between the smooth and ordinarily equivariant $K$-theory. There exists an obvious homomorphism $i_{*}$ from $K^{0}_{G, sm}(M)$ to $K^{0}_{G}(M)$:
\begin{align*}
    i_{*}:&K^{0}_{G,sm}(M)\rightarrow K^{0}_{G}(M)\\
    &[E,F,\sigma]\mapsto [E,F,\sigma].
\end{align*}

A natural problem is: when is $i_{*}$ an isomorphism? When $G$ and $M$ are compact, Proposition A.4 in \cite{bo} shows that $ i_{*}$ is an isomorphism. When $G$ and $M$ are not compact but the action is proper and cocompact, we have the following:
\begin{theorem}
    \label{sk}
    Let $M$ be a smooth manifold with a smooth proper action of a Lie group $G$. If there exists $H$, a compact subgroup of $G$, and $S\subset M$, a compact $H$-invariant submanifold without boundary, such that the action $[g,s]\mapsto gs$, for $g\in G$ and $s\in S$ defines a $G$-equivariant diffeomorphism from $G\times_{H}S$ to $M$, then we have the isomorphism:
    \begin{align*}
    i_{*}:&K^{0}_{G,sm}(M)\rightarrow K^{0}_{G}(M)\\
    &[E,F,\sigma]\mapsto [E,F,\sigma].
\end{align*}
\end{theorem}
For the proof, we need:
\begin{lemma}
\label{k2}
    Let $M$ and $S$ satisfy the condition in Theorem \ref{sk}. Let $[E]\in K_{G,sm}^{0}(M)$, where $E$ is a $G$-equivariant smooth vector bundle over $M$. There is an isomorphism from $K_{G,sm}^{0}(M)$ to $K_{H,sm}^{0}(S)$:
\begin{align*}
    r_{*}:&K_{G,sm}^{0}(M)\longrightarrow K_{H,sm}^{0}(S)\\
    &[E]\longmapsto [E|_{S}].
    \end{align*}

\end{lemma}
\begin{proof}
    It is clear that $r_{*}$ is a homomorphism from $K_{G,sm}^{0}(M)$ to $K_{H,sm}^{0}(S)$. For any $H$-equivariant bundle $F$ over $S$, the balanced product of $G$ with $F$ over $H$ is the space $G\times_{H} F=(G\times F)/H$, where $H$ acts on $G\times F$ via $(g,f)\cdot h=(gh,h^{-1}f)$. The space $G\times_{H} F$ is a smooth vector bundle over $G\times_{H}S$. Let $G$ act on $G\times F$ via $g_{0}\cdot(g,f)=(g_{0}g,f)$ for $g_{0},g\in G$ and $f\in F$ and this action descends to an action of  $G$ on $G\times_{H} F$ since it commutes with the action of $H$ on $G\times F$, which makes $G\times_{H} F$ a $G$-equivariant vector bundle over $G\times_{H}S$. Thus, we have:
    \begin{align*}
    \epsilon_{*}:&K_{H,sm}^{0}(S)\rightarrow K_{G,sm}^{0}(G\times_{H}S) \\
    &[F]\longmapsto [G\times_{H} F].
\end{align*}

Let $\phi:G\times_{H}S\rightarrow M$ be the action map. Since it is a $G$-equivariant diffeomorphism, it induces an isomorphism:
$$\phi^{*}: K_{G,sm}^{0}(M)\rightarrow K_{G,sm}^{0}(G\times_{H}S).$$

Similarly, $\phi:H\times_{H}S\rightarrow S$ is an $H$-equivariant diffeomorphism. Thus, it induces an isomorphism:
$$\phi^{*}: K_{H,sm}^{0}(S)\rightarrow K_{H,sm}^{0}(H\times_{H}S).$$
Since the following diagram commutes:
$$\begin{tikzcd}
 K_{G,sm}^{0}(M)\arrow[r, "r_{*}"] \arrow[d,"\phi^{*}" ] & K_{H,sm}^{0}(S) \arrow[ld, "\epsilon_{*}"]  \arrow[d, "\phi^{*}"]              \\
K_{G,sm}^{0}(G\times_{H}S)\arrow[r, "r_{*}"]                                & K_{H,sm}^{0}(H\times_{H}S),
\end{tikzcd}$$
the left vertical arrow being isomorphic implies that $\epsilon_{*}$ is surjective. Since the right vertical arrow is also isomorphic, $\epsilon_{*}$ is injective. Thus, $\epsilon_{*}$ is a isomorphism. Then, $r_{*}$ is an isomorphism, which completes the proof.
\end{proof}

Then, we will prove Theorem \ref{sk}.
\begin{proof}[Proof of Theorem \ref{sk}]
The following diagram commutes:
\begin{equation*}
\xymatrix{   K^{0}_{G,sm}(M)\ar[r]^{r_{*}}\ar[d]_{i_{*}} &   K^{0}_{H,sm}(S)  \ar[d]^{i_{*}}   \\
           K^{0}_{G}(M) \ar[r]^{r_{*}} &K_{H}^{0}(S) .
                      }
\end{equation*}
    The right-hand vertical arrow is isomorphic because of Proposition A.4 in~\cite{bo} and horizontal arrows are isomorphisms by Lemma~\ref{k2} and Corollary~8.5 in \cite{ref8}. Thus, $i_{*}:K^{0}_{G,sm}(M)\rightarrow K^{0}_{G}(M)$ is an isomorphism.
    \end{proof}

\subsection{Equivariant cohomology} 
In this subsection, we recall some basic definitions and theorems of equivariant cohomology and equivariant characteristic classes which can be found in \cite{ref6} and \cite{bgv}.

First, we recall the definition of equivariant cohomology. Let $M$ be a $C^{\infty}$-manifold with an action of a Lie group $G$, and let $\mathfrak{g}$ be the Lie algebra of $G$. For any $g\in G$, the conjugation map $C_{g}:G\rightarrow G$ given by $C_{g}(h)=ghg^{-1}$ is a Lie group homomorphism and denote $g_{*}={(C_{g})}_{*}:\mathfrak{g}\rightarrow\mathfrak{g}$ the induced Lie algebra homomorphism.
\begin{definition}
 Let $\Lambda T^{*}M$ be the bundle of exterior differentials over $M$. The space of sections $\Gamma(M,\Lambda T^{*}M)$ is called the space of differential forms, denoted by $\mathcal{A}(M)$. Let $\mathcal{A}^{\infty}(\mathfrak{g},M)$ be the algebra of $G$-invariant smooth maps $\alpha:\mathfrak{g}\rightarrow \mathcal{A}(M)$. For any $\alpha,\beta\in \mathcal{A}^{\infty}(\mathfrak{g},M)$ and $X\in \mathfrak{g}$, the multiplication is given by $\alpha\cdot\beta(X)=\alpha(X)\wedge\beta(X)$. The group $G$ acts on an element $\alpha\in \mathcal{A}^{\infty}(\mathfrak{g},M)$ by the formula 
 $$(g\cdot{\alpha})(X)=g\cdot(\alpha({(g^{-1})}_{*}(X))$$
 for any $g\in G$ and $X\in \mathfrak{g}$. The map $\alpha$ is $G$-invariant if it satisfies: $({\alpha})(g_{*}(X))=g\cdot\alpha(X)$ for any $X\in\mathfrak{g}$. The $\mathbb{Z}_{2}$-grading of $\mathcal{A}^{\infty}(\mathfrak{g},M)$ is the grading induced by the $\mathbb{Z}_{2}$-grading of $\mathcal{A}(M)$. 
\end{definition}

Thus, the equivariant differential can be defined as follows:
\begin{definition}[\cite{ref6}]
The group $G$ acts on $C^{\infty}(M)$ by the formula $(g\cdot\phi)(x)=\phi(g^{-1}x)$. For $X\in \mathfrak{g}$, we denote by $X_{M}$ the vector field on $M$ given by:
$$(X_{M}\cdot\phi)(x)=\frac{d}{dt}\phi(\exp(-tX)x)|_{t=0}.$$
Let the equivariant differential $d_{\mathfrak{g}}$ be given by $d_{\mathfrak{g}}(\alpha)(X)=d(\alpha(X))-\iota(X)(\alpha(X))$ for any $\alpha\in \mathcal{A}^{\infty}(\mathfrak{g},M)$ and $X\in \mathfrak{g}$, where $\iota(X)$ denotes contraction by the vector field $X_{M}$.
\end{definition}
Since $d_{\mathfrak{g}}^{2}(\alpha)=0$ for all $\alpha\in \mathcal{A}^{\infty}(\mathfrak{g},M)$ which makes $(A^{\infty}(\mathfrak{g},M),d_{\mathfrak{g}})$ a cochain complex, we obtain the following:
\begin{definition}[\cite{ref6}]
Let $H^{\infty}(\mathfrak{g}, M):=\Ker(d_{\mathfrak{g}})/\im(d_{\mathfrak{g}})$ be the cohomology of the complex $(\mathcal{A}^{\infty}(\mathfrak{g},M),d_{\mathfrak{g}})$, called the equivariant cohomology of the $G$-manifold $M$.
\end{definition}

Next, we recall the notion of equivariant superconnections which is useful for describing equivariant characteristic classes. A $G$-equviariant smooth vector bundle $E=E^{+}\bigoplus E^{-}$ is a $G$-equivariant smooth superbundle if $E^{+}$ and $E^{-}$ are $G$-equivariant vector bundles. Let $\mathcal{A}(M, E)$ denote the space of differential
forms on $M$ with values in $E$. In other words, $\mathcal{A}(M,E)=\Gamma(M,\Lambda T^{*}M\otimes E)$.
 Similarly, we denote $\mathcal{A}^{\infty}(\mathfrak{g},M,E)$ to be  the $\mathbb{Z}_{2}$-graded algebra of $G$-invariant smooth maps $\alpha:\mathfrak{g}\rightarrow \mathcal{A}(M,E)$.
\begin{definition}
 A superconnection on $E$ is an odd-parity first-order differential operator $\mathbb{A}:\mathcal{A}^{\pm}(M,E)\rightarrow \mathcal{A}^{\mp}(M,E)$
which satisfies Leibniz's rule in the $\mathbb{Z}_{2}$-graded sense.
\end{definition}
A superconnection $\mathbb{A}$ on the $G$-equivariant vector bundle $E$ is a $G$-invariant superconnection if it commutes with the action of $G$ on $\mathcal{A}(M, E)$. Such $G$-invariant superconnections induce equivariant superconnections:
\begin{definition}
The equivariant superconnection $\mathbb{A}_\mathfrak{g}$ associated to a $G$-invariant superconnection $\mathbb{A}$ is the operator on $\mathcal{A}^{\infty}(\mathfrak{g},M,E)$ defined by 
$$ (\mathbb{A}_\mathfrak{g}\alpha)(X):=(\mathbb{A}-\iota(X))(\alpha(X)), \qquad X\in \mathfrak{g}$$
where $\iota(X)$ denotes the contraction operator $\iota(X_M)$ on $\mathcal{A}(M,E)$.
\end{definition}
Then, an equivariant superconnection induces an equivariant (super)curvature:

\begin{definition}
Let $\mathbb{A}_\mathfrak{g}$ be an equivariant superconnection on a $G$-equivariant vector bundle $E$. We denote the equivariant curvature $\Omega_\mathfrak{g}$ by the formula
$$\Omega_\mathfrak{g}(X):=\mathbb{A}_\mathfrak{g}(X)^2+\mathcal{L}^{E}(X),\qquad X\in\mathfrak{g}$$
where $\mathbb{A}_\mathfrak{g}(X)=\mathbb{A}-\iota(X)$ and for any $s\in \Gamma(M,E)$, $\mathcal{L}^{E}(X)s=\frac{d}{dt}|_{t=0}\exp(t\cdot X)\cdot s$.
\end{definition}
\begin{remark}
    Every $G$-vector bundle $E$ is $\mathbb{Z}_{2}$-graded where $E^{+}\cong E$ and $E^{-}\cong\{0\}$. Because of this, any $G$-invariant connection is a $G$-invariant superconnection under this special $\mathbb{Z}_{2}$-grading. Thus, every $G$-invariant connection induces a $G$-equivariant curvature.
\end{remark}

In the special case when the invariant superconnection $\mathbb{A}$ on the equivariant bundle $E$ is just a connection $\nabla$, the inverse of an equivariant $\Hat{A}$-genus $\hat{A}_\mathfrak{g}(\nabla)$ can be defined as follows:
 
\begin{definition}[\cite{bgv}]
\label{Agenus}
Let $E$ be a $G$-equivariant vector bundle with a $G$-invariant metric and a $G$-invariant connection $\nabla$ compatible with the metric. Denote the inverse of an equivariant $\Hat{A}$-genus by
$$\Hat{A}_\mathfrak{g}^{-1}(E)(X)={\det}^{1/2}\left(\frac{\sinh(\Omega_\mathfrak{g}(X)/2)}{\Omega_\mathfrak{g}(X)/2}\right)$$
and the inverse of an invariant $\Hat{A}$-genus by:
$$\hat{A}_{G}^{-1}(E)=\Hat{A}_\mathfrak{g}^{-1}(E)(0).$$
\end{definition}
The equivariant cohomology will be used to define the equivariant Chern character from $K_{G, sm}^{0}(M)$ to $H^{\infty}(\mathfrak{g}, M)$, where $M$ is a proper $G$-compact manifold. Next, we recall the relative equivariant cohomology, which will be used to define the equivariant Chern character from $K_{G, sm}^{0}(V)$ to $H^{\infty}_{cv}(\mathfrak{g}, V)$, where $V$ is a $G$-equivariant vector bundle over a proper $G$-compact manifold $M$, and $H^{\infty}_{cv}(\mathfrak{g}, V)$ is the equivariant cohomology with compact support in the fiber direction.

Let $M$ be a $G$-manifold and $K$  be a closed  $G$-invariant subset of $M$. The relative equivariant complex can be given as follows:
\begin{definition}[\cite{ref6}]
Consider the complex $(\mathcal{A}^{\infty}(\mathfrak{g},M,M\backslash K), D_{rel})$ where 
\[
\mathcal{A}^{\infty}(\mathfrak{g},M,M\backslash K)=\mathcal{A}^{\infty}(\mathfrak{g},M)\oplus{\mathcal{A}^{\infty}(\mathfrak{g},M\backslash K)}
\]
and differential is given by $D_{rel}(\alpha,\beta)=(d_{\mathfrak{g}}\alpha,\alpha|_{M\backslash K}-d_{\mathfrak{g}}\beta)$.
\end{definition}
As before, $(\mathcal{A}^{\infty}(\mathfrak{g}, M, M\backslash K), D_{rel})$ is a cochain complex so that we can define the relative equivariant cohomology:
\begin{definition}[\cite{ref6}]
The cohomology associated to the complex \\ $(\mathcal{A}^{\infty}(\mathfrak{g},M,M\backslash K),D_{rel})$ is the relative equivariant cohomology space $H^{\infty}(\mathfrak{g},M,M\backslash K)$.
\end{definition}
In order to define an equivariant Chern character later, we give a homomorphism from $H^{\infty}(\mathfrak{g}, V, V\backslash F)$ to $H^{\infty}_{cv}(\mathfrak{g}, V)$ as in \cite{ref6}.
As before, let $M$ be a $G$-compact manifold by a proper action of a Lie group $G$ and $V$ is a $G$-vector bundle over $M$.
Let $F$ be a $G$-invariant closed subset of $V$ such that $F$ is a subset of some $G$-compact subset of $V$. We say that $\chi$ is an $F$-cutoff function if:

\begin{enumerate}
\item $\chi\in C^{\infty}(V)^{G}$ where $ C^{\infty}(V)^{G}$ is the space of $G$ equivariant smooth function on $V$,
\item ${1}\geq{\chi}\geq{0},{\chi}\mid_{F}=1$,
\item $\supp(\chi)$ is G-compact.
\end{enumerate}
\begin{remark}
    When $M$ is a $G$-compact manifold and the group action is proper, every $G$-vector bundle $V$ over $M$ has a $G$-invariant Hermitian metrics $\langle, \rangle$. For any smooth function $f$ on $\mathbb{R}$ with compact support satisfying ${1}\geq{f}\geq{0} ,f(0)=0$, $\chi (v)=f(\langle v,v\rangle)$ for any $v\in V$ is an $M$-cutoff function.
    \end{remark}

\begin{definition}
For an $F$-cutoff function $\chi$, define the map 
\begin{align*}
    P_{\chi}:&A^{\infty}(\mathfrak{g},V,V\backslash F)\longrightarrow A^{\infty}_{cv}(\mathfrak{g},V)\\
    &(\alpha,\beta)\longmapsto \chi\alpha+d\chi\wedge{\beta},
\end{align*}
called a cutoff map.
\end{definition}

\begin{proposition}
The cutoff map induces a homomorphism from $H^{\infty}(\mathfrak{g},V,V\backslash F)$ to $H^{\infty}_{cv}(\mathfrak{g},V)$.
\end{proposition}

\begin{proof}
First, we show that if $(\alpha,\beta)$ is closed, then $P_{\chi}(\alpha,\beta)$ is closed.

We have
\begin{align*}
   d_{\mathfrak{g}}(\chi\alpha+d\chi\wedge{\beta})&=d_{\mathfrak{g}}(\chi\alpha)-d\chi\wedge d_{\mathfrak{g}}\beta \\
&=(d\chi)\alpha+\chi d_{\mathfrak{g}}\alpha-d\chi\wedge{d_{\mathfrak{g}}\beta}\\
&=(d\chi)(\alpha-d_{\mathfrak{g}}\beta)+\chi d_{\mathfrak{g}}\alpha\\
&=0.   
\end{align*}

The last equality follows from $(\alpha,\beta)$ being closed, which means $$(0,0)=D_{rel}(\alpha,\beta)=(d_{\mathfrak{g}}\alpha,\alpha\mid_{V-F}-d_{\mathfrak{g}}\beta).$$

Next we show that if $(\alpha_{1},\beta_{1}),(\alpha_{2},\beta_{2}) \in H^{*}(\mathfrak{g},V,V\backslash F)$, then $P_{\chi}(\alpha_{1},\beta_{1})=P_{\chi}(\alpha_{2},\beta_{2}).$

If $[(\alpha_{1},\beta_{1})]=[(\alpha_{2},\beta_{2})]$, then setting $\alpha=\alpha_{1}-\alpha_{2}$ and $\beta=\beta_{1}-\beta_{2}$, we have: $$(0,0)=D_{rel}(\alpha,\beta)=(d_{\mathfrak{g}}\alpha,\alpha\mid_{V-F}-d_{\mathfrak{g}}\beta).$$
Then,
\begin{align*}
  P_{\chi}(\alpha_{1},\beta_{1})-P_{\chi}(\alpha_{2},\beta_{2})&=\chi(\alpha_{1}-\alpha_{2})+d\chi\wedge{(\beta_{1}-\beta_{2})}\\
&=\chi(d_{\mathfrak{g}}\alpha)+d\chi\wedge{(\alpha-d_{\mathfrak{g}}\beta)}\\
&=d_{\mathfrak{g}}(\chi\alpha)-d\chi\wedge{\alpha}+d\chi\wedge{\alpha}-d_{\mathfrak{g}}(\chi d_{\mathfrak{g}}\beta)\\
&=d_{\mathfrak{g}}(\chi(\alpha-d_{\mathfrak{g}}\beta)).   
\end{align*}

The cutoff map is additive. Hence the cutoff map induces a homomorphism from $H^{*}(\mathfrak{g}, V,V\backslash F)$ to $H^{*}_{cv}(\mathfrak{g}, V)$.
\end{proof}

\begin{remark}
\label{rem:indofchi}
The cohomology class of $P_{\chi}(\alpha,\beta)$ does not depend on the choice of $\chi$. 
That is because for two different choices $\chi_{1}$ and $\chi_{2}$,
we have 
\[
P_{\chi_{1}}(\alpha,\beta)-P_{\chi_{2}}(\alpha,\beta)=d_{\mathfrak{g}}((\chi_{1}-\chi_{2})\beta).
\]
\end{remark}
Because of Remark~\ref{rem:indofchi}, we denote the cutoff map from $H^{*}(\mathfrak{g},V,V\backslash F)$ to $H_{cv}^{*}(\mathfrak{g},V)$ as $P_{F}$.
\section{Equivariant Chern character}
In this section, we introduce equivariant Chern characters. Let $G$ be an almost connected Lie group, i.e., its component group is compact. Let $M$ be a $G$-compact proper $G$-space. For any $G$-equivariant superbundle $E=E^{+}\oplus{E^{-}}$ and an equivariant supercurvature $\Omega_\mathfrak{g}$ on $E$, the supertrace on $\End(E)$ extends to a map $\Str:\mathcal{A}(M,E)\rightarrow\mathcal{A}(M)$. Then we can define a map from $K^{0}_{G,sm}(M)$ to the equivariant cohomology $H^{\infty}(\mathfrak{g},M)$ as follows:
\begin{definition}There is a homomorphism from $K^{0}_{G,sm}(M)$ to $H^{\infty}(\mathfrak{g},M)$:
$$\chg: K^{0}_{G,sm}(M) \longrightarrow H^{\infty}(\mathfrak{g},M)$$
$$[E^{+},E^{-},\sigma] \longmapsto \Str(\exp{\Omega_\mathfrak{g}}),$$
where $\sigma$ is a $G$-equivariant isomorphism from $E^{+}$ to $E^{-}$ outside a $G$-compact subset of $M$ and for any $X\in \mathfrak{g}$, $\Str(\exp{\Omega_\mathfrak{g}})(X)=\Str(\exp{\Omega_\mathfrak{g}(X)})$.
\end{definition}
\begin{remark}
In the case that $M$ is $G$-compact, $\sigma$ can be chosen to be a $G$-equivariant isomorphism with empty support.  Any pair of $G$-equivariant vector bundles $E$, $F$ gives rise to a $K$-theory element $[E,F,\sigma]$. Therefore, $\chg[E,F,\sigma]$ is independent of the choice of $\sigma$. The cohomology class of $\Str(\exp{\Omega_\mathfrak{g}})$ is independent of the choice of the $G$-invariant superconnection $\mathbb{A}$, which is proved in \cite{bgv} (Theorem $7.7$).
\end{remark}
\begin{proof}
    We will show that the equivariant Chern character is well-defined. Suppose  $[E_{1},F_{1},\sigma_{1}]$ and $[E_{2},F_{2},\sigma_{2}]$ define the same class in $K^{0}_{G,sm}(M)$. There exist smooth G-vector bundles $H_{1}$, $H_{2}$ and smooth $G$-equivariant isomorphisms $a$, $b$ such that:
\begin{align*}
    a:E_{1}\oplus H_{1}\rightarrow E_{2}\oplus H_{2}\\
    b:F_{1}\oplus H_{1}\rightarrow F_{2}\oplus H_{2}.
\end{align*}
Let $\mathbb{A}=\begin{pmatrix}
\nabla^{E_{1}}&0\\
0&\nabla^{F_{1}}
\end{pmatrix}$, where $\nabla^{E_{1}}$ and $\nabla^{F_{1}}$ are $G$-invariant connections on vector bundles $E_{1}$ and $F_{1}$ respectively. Thus, $\mathbb{A}$ becomes a $G$-invariant superconnection. Let $\Omega_{\mathfrak{g}}$ denote the equivariant curvature induced by $\mathbb{A}$ and $\Omega^{E_{1}}_{\mathfrak{g}}$, $\Omega^{F_{1}}_{{\mathfrak{g}}}$ induced by  $\nabla^{E_{1}}$ and $\nabla^{F_{1}}$ respectively. Thus, for any $X\in \mathfrak{g}$ we have:
\begin{align*}
    \chg[E_{1},F_{1},\sigma_{1}](X)&=\Str(\exp(\Omega_{\mathfrak{g}}(X)))\\
    &=\Str\left(\exp\left(\begin{pmatrix}
\Omega^{E_{1}}_{\mathfrak{g}}(X)&0\\
0&\Omega^{F_{1}}_{\mathfrak{g}}(X)
\end{pmatrix}\right)\right).
\end{align*}
Let $\nabla^{H_{1}}$ be the $G$-invariant connection of $H_{1}$ and $\Omega^{H_{1}}_{\mathfrak{g}}$ be the $G$-equivariant curvature induced by $\nabla^{H_{1}}$. Similarly, we have:
\begin{align*}
    &\chg[E_{1}\oplus H_{1},F_{1}\oplus H_{1},\sigma_{1}\oplus \Id_{H_{1}}](X)\\
&=\Str\left(\exp\left(\begin{pmatrix}
\Omega^{E_{1}}_{\mathfrak{g}}(X)&0&0&0\\
0&\Omega^{H_{1}}_{\mathfrak{g}}(X)&0&0\\
0&0&\Omega^{F_{1}}_{\mathfrak{g}}(X)&0\\
0&0&0&\Omega^{H_{1}}_{\mathfrak{g}}(X)
\end{pmatrix}\right)\right)\\
&=\Str\left(\exp\left(\begin{pmatrix}
\Omega^{E_{1}}_{\mathfrak{g}}(X)&0\\
0&\Omega^{F_{1}}_{\mathfrak{g}}(X)
\end{pmatrix}\right)\right)
\end{align*}
for any $X\in \mathfrak{g}$.
Thus, we obtain $$\chg[E_{1},F_{1},\sigma_{1}]=\chg[E_{1}\oplus H_{1},F_{1}\oplus H_{1},\sigma_{1}\oplus \Id_{H_{1}}].$$ 
Similarly, we have:
$$\chg[E_{2},F_{2},\sigma_{2}]=\chg[E_{2}\oplus H_{2},F_{2}\oplus H_{2},\sigma_{2}\oplus \Id_{H_{2}}].$$
The smooth $G$-equivariant isomorphisms $a$ and $b$ induce a $G$-equivariant isomorphism:
$$\begin{pmatrix}
a&0\\
0&b
\end{pmatrix}:E_{1}\oplus H_{1}\oplus F_{1}\oplus H_{1}\rightarrow E_{2}\oplus H_{2}\oplus F_{2}\oplus H_{2}.
$$
Because $M$ is $G$-compact, $\sigma_{1}\oplus \Id_{H_{1}}$ and $\sigma_{2}\oplus \Id_{H_{2}}$ play no role, we have:
$$\chg[E_{1}\oplus H_{1},F_{1}\oplus H_{1},\sigma_{1}\oplus \Id_{H_{1}}]=\chg[E_{2}\oplus H_{2},F_{2}\oplus H_{2},\sigma_{2}\oplus \Id_{H_{2}}].$$
Therefore, $\ch_{G}[E_{1},F_{1},\sigma_{1}]=\ch_{G}[E_{2},F_{2},\sigma_{2}]$.
\end{proof}
Then, we define an equivariant Chern character from $K_{G,sm}^{0}(V)$ to $H^{\infty}_{cv}(\mathfrak{g},V)$ where $V$ is a $G$-equivariant vector bundle over a $G$-compact manifold $M$ such that the action of $G$ on $M$ is proper. It is a little different from the above case. First, we recall Quillen's relative equivariant Chern character defined by Paradan and Vergne in \cite{ref6}.
\begin{definition}[\cite{ref6}
Quillen's equivariant relative Chern character]
Let $(E,F,\sigma)$ be a smooth $K$-cocycle for $(G,V)$, $\mathbb{A}$ be a $G$-invariant  superconnection of $E\oplus F$ and $\Omega_{\mathfrak{g}}$ be  the equivariant curvature induced by $\mathbb{A}$. There exists a cohomology class $\ch_{Q}(E,F,\sigma):=[(\ch(\mathbb{A}),\beta)]\in H^{*}(\mathfrak{g},V,V\backslash \supp(\sigma))$ which is independent of the choice of $\mathbb{A}$ and the $G$-invariant Hermite metric of $E\oplus F$ such that:
\begin{align*}
&\ch(\mathbb{A})=\Str(\exp(\Omega_{\mathfrak{g}}))\\  &\beta=\int_{0}^{+\infty}\Str\left(i\begin{pmatrix}
0&\sigma^{*}\\
\sigma&0
\end{pmatrix}\exp\left(-t^{2}+it[\mathbb{A},\begin{pmatrix}
0&\sigma^{*}\\
\sigma&0
\end{pmatrix}]+\Omega_{\mathfrak{g}}\right)\right)dt.
\end{align*}
\end{definition}
Then, we use Quillen's relative equivariant Chern character to define the equivariant Chern character from $K_{G, sm}^{0}(V)$ to $H^{\infty}_{cv}(\mathfrak{g}, V)$:
\begin{definition}[equivariant Chern character]
\label{def.eq.Chern}
The equivariant Chern character is given by:
\begin{align*}
    \chg: & K_{G,sm}^{0}(V)\longrightarrow H^{\infty}_{cv}(\mathfrak{g},V) \\
    & [E,F,\sigma]\longmapsto P_{\supp(\sigma)}\circ \ch_{Q}(E,F,\sigma).
\end{align*}
\end{definition}

To show that the equivariant Chern character is well-defined, we need to prove that if
$[E_{1},F_{1},\sigma_{1}]=[E_{2},F_{2},\sigma_{2}]\in K_{G}^{0}(V)$, then:
$$P_{\supp(\sigma_{1})}\circ{\ch_{Q}(E_{1}, F_{1},\sigma_{1})}=P_{\supp(\sigma_{2})}\circ{\ch_{Q}(E_{2}, F_{2},\sigma_{2})}.$$ 

For the proof, and in later arguments, we need two lemmas: 
\begin{lemma}
    \label{l0}
    For any $[E,F,\sigma]\in K^{0}_{G,sm}(V)$, let the restriction of $\sigma:E\rightarrow F$  to $V\backslash\supp(\sigma)$ be unitary. Then for any $\alpha,\gamma\in \mathcal{A}(V)$, $A,C\in \End(E)$ and $n\in \mathbb{N}^{+}$, we have:
    $$\Str\left(i\begin{pmatrix}
0&\sigma^{*}\\
\sigma &0
\end{pmatrix}\left(\alpha\otimes\begin{pmatrix}
0&A\sigma^{*}\\
\sigma A&0
\end{pmatrix}+\gamma\otimes\begin{pmatrix}
C&0\\
0&\sigma C\sigma^{*}
\end{pmatrix}\right)^{n}\right)=0
    $$
 outside $\supp(\sigma)$.
\end{lemma}
\begin{proof}
By induction, it can be concluded that $$i\begin{pmatrix}
0&\sigma^{*}\\
\sigma &0
\end{pmatrix}\left(\alpha\otimes\begin{pmatrix}
0&A\sigma^{*}\\
\sigma A&0
\end{pmatrix}+\gamma\otimes\begin{pmatrix}
C&0\\
0&\sigma C\sigma^{*}
\end{pmatrix}\right)^{n}$$  can be represented as a linear combination of $\alpha_{i}\otimes{\begin{pmatrix}
0&A_{i}\sigma^{*}\\
\sigma A_{i}&0
\end{pmatrix}}$ and $\gamma_{i}\otimes\begin{pmatrix}
C_{i}&0\\
0&\sigma C_{i}\sigma^{*}
\end{pmatrix}$ where $\alpha_{i},\gamma_{i}\in\mathcal{A}(M)$, $A_{i},C_{i}\in \End(E)$ for  some $i\in \mathbb{N}^{+}$.

The lemma is then proved because$$\Str\left(\alpha_{i}\otimes{\begin{pmatrix}
0&A_{i}\sigma^{*}\\
\sigma A_{i}&0
\end{pmatrix}}\right)=\Str\left(\gamma_{i}\otimes\begin{pmatrix}
C_{i}&0\\
0&\sigma C_{i}\sigma^{*}
\end{pmatrix}\right)=0$$
holds outside $\supp(\sigma).$
\end{proof}

\begin{lemma}
\label{ch is classical}
For any $K$-cocycle $[E,F,\sigma]\in K_{G,sm}^{0}(V)$, there exist  a suitable superconnection $\mathbb{A}$ and a suitable metric such that
$$\chg[E,F,\sigma]=\chi \ch(\mathbb{A}).$$
\end{lemma}

\begin{proof}
First we choose a $G$-invariant metric on $E\oplus{F}$, which makes
$\sigma:E\rightarrow F$ unitary outside $\supp(\sigma)$, and a $G$-invariant superconnection $\mathbb{A}=\begin{pmatrix}
\nabla^{E}&0\\
0&\nabla^{F}
\end{pmatrix}$, where $\nabla^{E}$ and $\nabla^{F}$ are $G$-invariant connections on vector bundles $E$ and $F$ respectively and satisfy that $\nabla^{E}=\sigma \circ{\nabla^{F}}\circ{\sigma^{*}}$ outside $\supp(\sigma)$.

Then,we claim that $\chg[E,F,\sigma]=\chi \ch(\mathbb{A})$. 
Indeed, we have 
\[
\chg[E,F,\sigma]=\chi \ch(\mathbb{A})+d(\chi)\wedge{\beta},
\]
where $$\beta(X)=\int_{0}^{+\infty}\Str\left(i\begin{pmatrix}
0&\sigma^{*}\\
\sigma&0
\end{pmatrix}\exp\left(-t^{2}+it\left[\mathbb{A},\begin{pmatrix}
0&\sigma^{*}\\
\sigma&0
\end{pmatrix}\right]+\mathbb{A}^{2}+\mu^{\mathbb{A}}(X)\right)\right)$$
for any $X\in \mathfrak{g}$.

To prove our claim, we only need to show $\beta(X)=0$ outside $\supp(\sigma)$.

Without loss of generality, let  $\nabla^{E}=d+\omega^{E}$, where  
 $\omega^{E}\in \mathcal{A}^{1}(V,\End(E))$. Thus $\mathbb{A}=\begin{pmatrix}
d+\omega^{E}&0\\
\sigma&\sigma(d+\omega^{E})\sigma^{*}
\end{pmatrix}$.
For any $\alpha\otimes s\in \mathcal{A}(V,E)$ where $\alpha$ is homogeneous (which means there exists $i\in \mathbb{N}^{+}$ such that $\alpha\in \mathcal{A}^{i}(V)$), we have:
\begin{align*}
    &\left[\mathbb{A},\begin{pmatrix}
0&\sigma^{*}\\
\sigma&0
\end{pmatrix}\right](\alpha\otimes s)=\mathbb{A}\circ\begin{pmatrix}
0&\sigma^{*}\\
\sigma&0
\end{pmatrix}(\alpha\otimes s)+\begin{pmatrix}
0&\sigma^{*}\\
\sigma&0
\end{pmatrix}\circ\mathbb{A}(\alpha\otimes s)\\
=&\mathbb{A}((-1)^{\deg(\alpha)}\alpha\otimes\begin{pmatrix}
0&\sigma^{*}\\
\sigma&0
\end{pmatrix}s)+\begin{pmatrix}
0&\sigma^{*}\\
\sigma&0
\end{pmatrix}(d\alpha\otimes s+(-1)^{\deg(\alpha)}\alpha\wedge \mathbb{A}s)\\
=&0.
\end{align*}
Here $$\left[\mathbb{A},\begin{pmatrix}
0&\sigma^{*}\\
\sigma&0
\end{pmatrix}\right]=\mathbb{A}\circ\begin{pmatrix}
0&\sigma^{*}\\
\sigma&0
\end{pmatrix}+\begin{pmatrix}
0&\sigma^{*}\\
\sigma&0
\end{pmatrix}\circ\mathbb{A},$$
and for any $\alpha\otimes s\in \mathcal{A}(V,E)$ where $\alpha$ is homogenous, 
$$\begin{pmatrix}
0&\sigma^{*}\\
\sigma&0
\end{pmatrix}(\alpha\otimes s)=(-1)^{\deg(\alpha)}\alpha\times\begin{pmatrix}
0&\sigma^{*}\\
\sigma&0
\end{pmatrix}s.$$
Now, restrict $\beta$ outside $\supp(\sigma)$ then for any $X\in\mathfrak{g}$ ,we have $$\beta(X)=\int_{0}^{\infty}\Str\left(i\begin{pmatrix}
0&\sigma^{*}\\
\sigma&0
\end{pmatrix}\exp\begin{pmatrix}
\Omega^{E}(X)-t^{2}&0\\
0&\sigma(\Omega^{E}(X)-t^{2})\sigma^{*}
\end{pmatrix}\right)dt$$
 where $\Omega^E(X)=(\nabla^{E})^{2}+\mathcal{L}^{E}(X)-\nabla_{X_{M}}$. Here
$$\begin{pmatrix}
\Omega^{E}(X)-t^{2}&0\\
0&\sigma(\Omega^{E}(X)-t^{2})\sigma^{*}
\end{pmatrix}
$$
can be written in the form of $\gamma\otimes\begin{pmatrix}
C&0\\
0&\sigma C\sigma^{*}
\end{pmatrix}$
where $\gamma\in \mathcal{A}(M)$ and $C\in \End(E)$.

Because of Lemma $\ref{l0}$, we have:
\begin{align*}
&\beta(X)\\
&=\int_{0}^{\infty}\Str\left(i\begin{pmatrix}
0&\sigma^{*}\\
\sigma&0
\end{pmatrix}\exp\begin{pmatrix}
\Omega^{E}(X)-t^{2}&0\\
0&\sigma(\Omega^{E}(X)-t^{2})\sigma^{*}
\end{pmatrix}\right)dt\\
&=\int_{0}^{\infty}\Str\left(i\begin{pmatrix}
0&\sigma^{*}\\
\sigma&0
\end{pmatrix}\exp\left(\gamma\otimes\begin{pmatrix}
C&0\\
0&\sigma C\sigma^{*}
\end{pmatrix}\right)\right)dt\\
&=\sum_{k=0}^{+\infty}\int_{0}^{\infty}\Str\left(i\begin{pmatrix}
0&\sigma^{*}\\
\sigma&0
\end{pmatrix}\frac{1}{k!}\left(\gamma\otimes\begin{pmatrix}
C&0\\
0&\sigma C\sigma^{*}
\end{pmatrix}\right)^{k}\right)dt\\
&=0.
\end{align*}
\end{proof}

To show the equivariant Chern character is well-defined, we need to prove that if
$[E_{1},F_{1},\sigma_{1}]=[E_{2},F_{2},\sigma_{2}]\in K_{G,sm}^{0}(V)$, then:
$$P_{\supp(\sigma_{1})}\circ{\ch_{Q}(E_{1},F_{1},\sigma_{1})}=P_{\supp(\sigma_{2})}\circ{\ch_{Q}(E_{2},F_{2},\sigma_{2})}.$$

First, we claim that:
\begin{proposition}
\label{equivariant Chern character}
For two finite dimensional $K$-cocycles $[E_{1},F_{1},\sigma_{1}]$, $[E_{2},F_{2},\sigma_{2}]$, 
if there exist two $G$-equivariant isomorphisms:
\begin{align*}
    a:E_{1}\rightarrow E_{2}\\
    b:F_{1}\rightarrow F_{2}
\end{align*}
such that $b^{-1}\circ {\sigma_{2}}\circ{a}=\sigma_{1}$ outside a $G$-compact subset $B$ of $V$, then:
$$P_{\supp(\sigma_{1})}\circ{\ch_{Q}(E_{1},F_{1},\sigma_{1})}=P_{\supp(\sigma_{2})}\circ{\ch_{Q}(E_{2},F_{2},\sigma_{2})}.$$
\end{proposition}
\begin{proof}
Let $F$ denote the union of $\supp(\sigma_{1}),\supp(\sigma_{2})$ and $B$.

Using the $G$-compact set $F$, we construct an $F$-cutoff function which we denote as $\chi$.
By choosing an appropriate metric, one can assume $a$ and $b$ to be unitary. In fact, fix a $G$-invariant metric on $E_{1}$ we can construct a $G$-invariant metric on $E_{2}$ by pulling back the one on $E_1$ through $a^{-1}$. It is easy to check that $a$ is unitary with respect to these two metrics.

Let $\mathbb{A}_{1}=\begin{pmatrix}
\nabla^{E_{1}}&0\\
0&\nabla^{F_{1}}
\end{pmatrix}, \mathbb{A}_{2}=\begin{pmatrix}
a\nabla^{E_{1}}a^{*}&0\\
0&b\nabla^{F_{1}}b^{*}
\end{pmatrix}$ 
be superconnections on $E_{1}\oplus{F_{1}}$ and $E_{2}\oplus{F_{2}}$ respectively. We have:
\begin{align*}
\chg[E_{1},F_{1},\sigma_{1}]-\chg[E_{2},F_{2},\sigma_{2}]&=\chi(\ch(\mathbb{A}_{1})-\ch(\mathbb{A}_{2}))+d\chi(\beta_{1}-\beta_{2})\\
(\text{by conjugation})&=d\chi(\beta_{1}-\beta_{2}).
\end{align*}

Since $\beta_{1}=\beta_{2}$ outside of $F$ up to conjugation, 
we have $d\chi(\beta_{1}-\beta_{2})=0$. Therefore, $\chg[E_{1},F_{1},\sigma_{1}]=\chg[E_{2},F_{2},\sigma_{2}]$, which completes the proof.
\end{proof}

Therefore, we have:
\begin{theorem}
The equivariant Chern character in Definition~\ref{def.eq.Chern} is well-defined.
\end{theorem}

\begin{proof}
Because of Lemma~\ref{ch is classical}, for every $X\in\mathfrak{g}$, we have:
\begin{multline*}
\chg[E\oplus{H},F\oplus{H},\sigma\oplus{\Id_{H}}](X)\\
=\chi \Str\left(\exp\begin{pmatrix}
\Omega^{E}(X)&0&0&0\\
0&\Omega^{H}(X)&0&0\\
0&0&\Omega^{F}(X)&0\\
0&0&0&\Omega^{H}(X)
\end{pmatrix}\right)\\
=\chi \Str\left(\exp\begin{pmatrix}
\Omega^{E}(X)&0\\
0&\Omega^{F}(X)
\end{pmatrix}\right)=\chg[E,F,\sigma](X).
\end{multline*}
  Suppose $[E_{1},F_{1},\sigma_{1}]=[E_{2},F_{2},\sigma_{2}]\in K_{G}^{0}(V)$. This means that there exist two $G$-equivariant vector bundles $H_{1}$,$H_{2}$, and two $G$-equivariant isomorphisms:
\begin{equation*}
      a:E_{1}\oplus{H_{1}}\rightarrow E_{2}\oplus{H_{2}} \quad\quad\quad
   b:F_{1}\oplus{H_{1}}\rightarrow F_{2}\oplus{{H_{2}}},
\end{equation*}
such that $b^{-1}\circ{(\sigma_{2}\oplus{\Id_{H_{2}}})}\circ{a}=\sigma_{1}\oplus{\Id_{H_{1}}}$ outside a $G$-compact set of $V$. Because of Lemma~\ref{ch is classical} and Proposition~\ref{equivariant Chern character}, we have:
\begin{align*}
\chg[E_{1},F_{1},\sigma_{1}]&=\chg[E_{1}\oplus{H_{1}},F_{1}\oplus{H_{1}},\sigma_{1}\oplus{\Id_{H_{1}}}]\\
&=\chg[E_{2}\oplus{H_{2}},F_{2}\oplus{H_{2}},\sigma_{2}\oplus{\Id_{H_{2}}}]\\
&=\chg[E_{2},F_{2},\sigma_{2}].
\end{align*}
The theorem is then proved.
\end{proof}

Because of Lemma~\ref{ch is classical}, we claim that the equivariant Chern character is additive.
\begin{proposition}
The equivariant Chern character in Definition~\ref{def.eq.Chern} is additive.
\begin{proof}
For $[E_{0},F_{0},\sigma_{0}]$, $[E_{1},F_{1},\sigma_{1}]\in K_{G}^{0}(V)$, we need to prove that $\chg([E_{0},F_{0},\sigma_{0}]\oplus{[E_{1},F_{1},\sigma_{1}]})=\chg[E_{0},F_{0},\sigma_{0}]+\chg[E_{1},F_{1},\sigma_{1}]$. We choose $\chi$ to be a $\supp({\sigma_{0}})\cup\supp({\sigma_{1}})$-cutoff function. Because of Lemma~\ref{ch is classical}, we have suitable superconnections $\mathbb{A}_{0}=\begin{pmatrix}
\nabla_{E_{0}}&0\\
0&\nabla_{F_{0}}
\end{pmatrix}$ and $\mathbb{A}_{1}=\begin{pmatrix}
\nabla_{E_{1}}&0\\
0&\nabla_{F_{1}}
\end{pmatrix}$ where $\nabla_{E_{0}}=\sigma_{0}\circ\nabla_{F_{0}}\circ{\sigma_{0}^{*}}$ outside $\supp({\sigma_{0}})$ and $\nabla_{E_{1}}=\sigma_{1}\circ\nabla_{F_{1}}\circ{\sigma_{1}^{*}}$ outside $\supp({\sigma_{1}})$ such that $\chg[E_{0},F_{0},\sigma_{0}]=\chi\ch(
\mathbb{A}_{0})$ and $\chg[E_{1},F_{1},\sigma_{1}]=\chi\ch(
\mathbb{A}_{1})$.

Therefore, for any $X\in\mathfrak{g}$, we have:
\begin{align*}
&\chg([E_{0},F_{0},\sigma_{0}]\oplus{[E_{1},F_{1},\sigma_{1}]})(X)\\
=&\chg[E_{0}\oplus E_{1},F_{0}\oplus{F_{1}},\sigma_{0}\oplus{\sigma_{1}}](X)\\
=&\chi\ch\begin{pmatrix}
\Omega_{E_{0}}&0&0&0\\
0&\Omega_{E_{1}}&0&0\\
0&0&\Omega_{F_{0}}&0\\
0&0&0&\Omega_{F_{1}}
\end{pmatrix}(X)\\
=&\chi\Str\left(\exp\begin{pmatrix}
\Omega_{E_{0}}(X)&0&0&0\\
0&\Omega_{E_{1}(X)}&0&0\\
0&0&\Omega_{F_{0}}(X)&0\\
0&0&0&\Omega_{F_{1}}(X)
\end{pmatrix}\right)\\
=&\chi\ch(\mathbb{A}_{0})(X)+\chi\ch(\mathbb{A}_{1})(X)\\
=&\chg[E_{0},F_{0},\sigma_{0}](X)+\chg[E_{1},F_{1},\sigma_{1}](X).
\end{align*}
The proposition is then proved.
\end{proof}
\end{proposition}

We can also deduce that the equivariant Chern character is multiplicative:
\begin{proposition}
The equivariant Chern character in Definition~\ref{def.eq.Chern} is multiplicative.
\begin{proof}
For $[E_{0},F_{0},\sigma_{0}]$, $[E_{1},F_{1},\sigma_{1}]\in K_{G}^{0}(V)$, we need to prove that $\chg([E_{0},F_{0},\sigma_{0}]\cdot{[E_{1},F_{1},\sigma_{1}]})=\chg[E_{0},F_{0},\sigma_{0}]\wedge{\chg[E_{1},F_{1},\sigma_{1}]}$. We choose $\chi_{i}$ to be a $\supp(\sigma_{i})$-cutoff function, $i=0,1$. Thus we can see $\chi=\chi_{0}\cdot\chi_{1}$ is a  $\supp(\sigma_{0})\cap\supp(\sigma_{1})$-cutoff function.
Indeed, if we choose $\mathbb{A}_{0}=\begin{pmatrix}
\nabla_{E_{0}}&0\\
0&\nabla_{F_{0}}
\end{pmatrix}$ and $\mathbb{A}_{1}=\begin{pmatrix}
\nabla_{E_{1}}&0\\
0&\nabla_{F_{1}}
\end{pmatrix}$ where $\nabla_{E_{0}}=\sigma_{0}\circ\nabla_{F_{0}}\circ{\sigma_{0}^{*}}$ outside $\supp({\sigma_{0}})$ and $\nabla_{E_{1}}=\sigma_{1}\circ\nabla_{F_{1}}\circ{\sigma_{1}^{*}}$ outside $\supp({\sigma_{1}})$ such that $\chg[E_{0},F_{0},\sigma_{0}]=\chi_{0}\ch(
\mathbb{A}_{0})$ and $\chg[E_{1},F_{1},\sigma_{1}]=\chi_{1}\ch(
\mathbb{A}_{1})$, we can define a superconnection $\mathbb{A}_{0}\otimes{\Id_{E_{0}\oplus{F_{0}}}}+\Id_{E_{1}\oplus{{F_{1}}}}\otimes{\mathbb{A}_{1}}$ on the superbundle $(E_{0}\oplus{F_{0}})\otimes({E_{1}\oplus{{F_{1}}}})$. Therefore, we have 
\begin{align*}
&\chg([E_{0},F_{0},\sigma_{0}]\cdot{[E_{1},F_{1},\sigma_{1}]})\\
=&\chg[E_{0}\otimes{E_{1}}\oplus{F_{0}\otimes{F_{1}}},F_{0}\otimes{E_{1}}\oplus{E_{0}\otimes{F_{1}}},\sigma]\\
=&\chi\ch(\mathbb{A}_{0}\otimes{\Id_{E_{0}\oplus{F_{0}}}}+\Id_{E_{1}\oplus{{F_{1}}}}\otimes{\mathbb{A}_{1}})\\
=&\chi_{0}\ch(\mathbb{A}_{0})\wedge\chi_{1}\ch(\mathbb{A}_{1})\\
=&\chg[E_{0},F_{0},\sigma_{0}]\wedge{\chg[E_{1},F_{1},\sigma_{1}]}.
\end{align*}
The proposition is then proved.
\end{proof}
\end{proposition}
So far, we have verified the well-definedness of equivariant Chern characters and proved that they are additive and multiplicative. Next, we will explain that in special cases, these maps give rise to homomorphisms from equivariant $K$-theory to equivariant cohomology.

If $M$ satisfies the conditions in Theorem \ref{sk}, then $i_{*}$ is an isomorphism from $K_{G,sm}^{0}(M)$ to $K_{G}^{0}(M)$. Then we can define the equivariant Chern character as follows:
\begin{equation}
    \label{ch1}
    \ch_{G}\circ (i_{*})^{-1}:K_{G}^{0}(M)\rightarrow H^{\infty}(\mathfrak{g},M).
\end{equation}
We also denote this map by $\ch_{G}$.

Next, when the $G$-equivariant bundle $\pi:V\rightarrow M$  has a $G$-$Spin$ structure and even rank. (Here $M$ need to satisfy the conditions in Theorem \ref{sk}). According to Theorem~$8.11$ in \cite{ref7}, there is a Thom isomorphism:
\begin{align*}
    i!:&K_{G}^{0}(M)\rightarrow K_{G}^{0}(V)\\
    &[E]\rightarrow \pi^{*}(E)\cdot[\pi^{*}(S^{+}_{V}),\pi^{*}(S^{-}_{V}),\sigma_{\mathbb{C}}]
\end{align*}
where $S_{V}=S_{V}^{+}\oplus S_{V}^{-}$ is the equivariant spinor bundle over $M$. We denote by $\pi!$ the inverse of $i!$.

Then, we can define the Thom homomorphism $i!$ from $K_{G, sm}^{0}(M)$ to $K_{G, sm}^{0}(V)$ in the same way. Unfortunately, this map is not necessarily isomorphic. But it is injective when $M$ satisfies the conditions in Theorem~\ref{sk}. Note that the following diagram commutes:
\begin{equation*}
\xymatrix{   K^{0}_{G,sm}(M)\ar[r]^{i!}\ar[d]_{i_{*}} &   K^{0}_{G,sm}(V)  \ar[d]^{i_{*}} \\
           K^{0}_{G}(M) \ar[r]^-{i!} & K^{0}_{G}(V).
                      }
\end{equation*}
Since the left vertical arrow and the lower horizontal arrow are both isomorphic, the higher horizontal arrow is injective and the right vertical arrow is surjective.

Because of that, $i_{*}$  restricted to the image $\im(i!)$ of $i!:K_{G, sm}(M)\rightarrow K_{G, sm}(V)$ is an isomorphism. Then we can define the equivariant Chern character as follows:
\begin{equation}
    \label{ch2}
    \ch_{G}\circ (i_{*})^{-1}:K_{G}^{0}(V)\rightarrow H^{\infty}_{cv}(\mathfrak{g},V).
\end{equation}
Similarly, we denote it by $\ch_{G}$.

\section{Equivariant Riemann-Roch formula}

In this section, we aim to prove the equivariant Riemann-Roch formula, the first main result of the paper:
\begin{theorem}
\label{err}
Let $G$ be an almost-connected Lie group, $M$ be a proper $G$-compact space satisfying conditions in Theorem~\ref{sk}, and $V$ be a $G$-vector bundle over $M$ of even
rank with an equivariant spin structure. The following diagram commutes:
    \begin{equation*}
\xymatrix{   K^{0}_{G}(V)\ar[r]^{\pi!}\ar[d]_{\chg} &   K^{0}_{G}(M)  \ar[d]^{(-2\pi i)^{n/2}\cdot\chg\wedge\Hat{A}_{g}^{-1}(V)}   \\
           H^{\infty}_{cv}(\mathfrak{g},V) \ar[r]^-{\pi!} & H^{\infty}(\mathfrak{g},M),
                      }
\end{equation*}
where $\Hat{A}_{g}(V)^{-1}$ is the inverse of the equivariant $\hat{A}$-genus in Definition~\ref{Agenus}, $\rank(V)=n$, $\chg$ on the left is defined in (\ref{ch2}), $\chg$ on the right is defined in (\ref{ch1}) and $\pi!:H_{cv}^{\infty}(\mathfrak{g},V)\rightarrow H^{\infty}(\mathfrak{g},M)$ is the homomorphism induced by integration over the fiber.
\end{theorem}

In \cite{ref6}, Paradan and Vergne proved the equivariant Riemann-Roch formula when $G$ is compact. Here, we apply the spirit of their proof and combine methods established by Mathai and Quillen in \cite{ref3} to prove Theorem \ref{err}.

Recall that $V$ is an equivariant bundle over $M$ with an equivariant spin structure.
We denote by $S_{V}$ the equivariant $\spin$-bundle over $M$ associated to $V$. 

First, we recall basic facts about Clifford algebras, spinors, and supertraces which can be found in \cite{ref3}.

Let the natural number $n$ be even. The Clifford algebra $C_{n}$, may be defined as the superalgebra with odd generators 
$\gamma_{1},...,\gamma_{n}$ subject to the relations:
$$[\gamma_{j},\gamma_{k}]=2\delta_{jk}.
$$
Clearly $C_{n}$ has a basis 
consisting of $2^{n}$ monomials  
$$\gamma_{I}=\gamma_{i_{1}}...\gamma_{i_{p}},~I=\{i_{1},...i_{p}\},~i_{1}<...<i_{p}$$
where $I$ runs over the subsets of $\{ 1, . . . , n\}$ and  $|I|$ is the cardinality of $I$. For all subsets of $\{ 1, . . . , n\}$ that appear in this article, we assume that the elements of it are arranged in increasing order.

There is a supertrace on $C_{n}$ such that for any $I\subset\{ 1, . . . , n\}$, 
$$\Str(\gamma_{I})=\begin{cases} 
        ~0 &I< \{ 1, . . . , n\}\\
        (2i)^{n/2}&I=\{ 1, . . . , n\}.
\end{cases}$$
Let $\gamma=(\gamma_{i})_{i=1}^{n}$ be the column vector with entries in $C_{n}$. Then we have a map:
$$c:\mathbb{C}^{n}\rightarrow C_{n},~c(z)=iz^{t}\gamma.$$

According to Proposition 2.10 in \cite{ref3}: 
\begin{lemma}[\cite{ref3}]
For any skew-symmetric matrix $\omega$, we denote $\frac{1}{2}\gamma^{t}\omega\gamma=\sum_{i<j}w_{ij}\gamma_{i}\gamma_{j}$. Thus, we have:
\label{l1}
    $$\Str(\exp(\frac{1}{2}\gamma^{t}\omega\gamma))=(2i)^{\frac{n}{2}}{\det}^{1/2}\left(\frac{\sinh(\omega)}{\omega}\right)\Pf(\omega).$$
\end{lemma}

And for later arguments, we need:
\begin{lemma}
\label{l2}
    For any $z\in \mathbb{C}^{n}$, we have: 
    $$\Str(c(z)\exp(\frac{1}{2}\gamma^{t}\omega\gamma))=0.
    $$
\end{lemma}
\begin{proof}
    Note that $\exp(\frac{1}{2}\gamma^{t}\omega\gamma)=\sum_{m=1}^{\infty}\frac{1}{m!}(\sum_{i<j}\omega_{ij}\gamma_{i}\gamma_{j})^{m}$. Since $\frac{1}{m!}(\sum_{i<j}\omega_{ij}\gamma_{i}\gamma_{j})^{m}$ is even in the Clifford algebra, $c(z)\frac{1}{m!}(\sum_{i<j}\omega_{ij}\gamma_{i}\gamma_{j})^{m}$ is odd.
    
    Because of this, we have: $$\Str(c(z)\exp(\frac{1}{2}\gamma^{t}\omega\gamma))=\sum_{m=1}^{\infty}\Str(c(z)\frac{1}{m!}(\sum_{i<j}\omega_{ij}\gamma_{i}\gamma_{j})^{m})=0.$$
\end{proof}

Next, we will study more general superalgebras. Let $A=\wedge[J_{1},....J_{n}]$, denote the exterior algebra generated by $J_{1},...,J_{n}$. It is the free commutative superalgebra with the odd generators $J_{i}$. When $I$ runs over all subsets of $\{1,2,...,n\}$, $$J^{I}:=\gamma_{i_{1}}...\gamma_{i_{|I|}},~I=\{i_{1},...i_{|I|}\},~i_{1}<...<i_{|I|}$$ constitute a basis of $A$. If we work in the tensor product superalgebra $A\otimes C_{n}$, where $A$ is a commutative superalgebra, then the supertrace of $C_{n}$ can be extended to an $A$-linear map 
\begin{align*}
    \Str: &A\otimes C_{n}\rightarrow A\\
    & \omega\otimes\alpha\mapsto \omega\Str(\alpha).
\end{align*}

Because $A$ is commutative, this ‘relative’ supertrace satisfies the basic property that for any homogeneous elements $a,b\in A\otimes C_{n}$, $\Str(ab)=(-1)^{|a||b|}\Str(ba)$. 

Then, we can prove the following theorem:
\begin{theorem}
\label{l4} 
Let $\gamma=(\gamma_{i})_{i=1}^{n}$ and $J=(J_{i})_{i}^{n}$ be vectors with entries in $C_{n}$ and $A$ respectively. Then:
\begin{align*}
    &\Str\left((\sum_{i=1}^{n}\gamma_{i})\cdot \exp(\frac{1}{2}\gamma^{t}\omega\gamma+J^{t}\gamma)\right)\\
    &=(2i)^{\frac{n}{2}}\cdot\det\left(\frac{\sinh{\omega}}{\omega}\right)^{\frac{1}{2}}\cdot\left(\sum_{k=1}^{n}\sum_{I^{'}}\sum_{I}(-1)^{\frac{|I|+1}{2}}c_{k}\epsilon(I\cup\{k\},I^{'})\epsilon(\{k\},I)\Pf(\omega_{I^{'}})J^{I}\right)
\end{align*}
where $I$ and $I'$ run over all subset of $\{1,2,3...,n\}$. For any $A, B\subset \{1,2,3...,n\}$, if $A\cap B\neq\emptyset$ or $A\cup B\neq \{1,2,3...,n\}$, then $\epsilon(A,B)=0$. If $A\cup B=\{1,2,3...,n\}$, $\epsilon(A,B)$ is defined by:
$$J^{A}J^{B}=\epsilon(A,B)J^{\{1,2,3...,n\}}.$$
\end{theorem}
The proof of the above theorem splits into two lemmas. The first one comes from Lemma 2.21 in \cite{ref3}.
\begin{lemma}\cite{ref3}
\label{l3}
    If $K_{1}, K_{2},...,K_{n}$ are odd elements of $A$, then for any $s\in \mathbb{R}$ we have:
    $$\exp(sK^{t}\gamma)\gamma_{i}\exp(-sK^{t}\gamma)=\gamma_{i}+2sK_{i}.$$
\end{lemma}
The second lemma is a technical lemma to be proved in Appendix B.
\begin{lemma}
    For any $n\times n$ skew-symmetric matrix $\omega=\{\omega_{i}^{j}\}$, we have:
    $$(\sum_{k=1}^{n}c_{k}\omega(J)_{k})\exp(\frac{1}{2}J^{t}\omega{J})=\sum_{k=1}^{n}\sum_{I}c_{k}\epsilon(\{k\},I)\Pf(\omega_{I\cup\{k\}})J^{I}
    $$
    where $\omega(J)_{k}=\omega^{l}_{k}J_{l}$.
\end{lemma}

Now, we will prove Theorem \ref{l4}:

\begin{proof}[Proof of Theorem \ref{l4}]
Let $K=-\omega^{-1}(J)$, we have:
 {\small
    \begin{align*}
        &\Str\left((\sum_{i=1}^{n}\gamma_{i})\cdot \exp(\frac{1}{2}\gamma^{t}\omega\gamma+J^{t}\gamma)\right)\\
        &=\Str\left((\sum_{i=1}^{n}\gamma_{i})\cdot \exp(\frac{1}{2}(\gamma-\omega^{-1}(J))^{t}\omega(\gamma-\omega^{-1}(J)))\right)\cdot\exp(\frac{1}{2}J^{t}\omega^{-1}J)\\
        &=\Str\left((\sum_{i=1}^{n}\gamma_{i})\cdot \exp(\frac{1}{2}K^{t}\gamma)\cdot\exp(\frac{1}{2}\gamma^{t}\omega\gamma)\cdot\exp(-\frac{1}{2}K^{t}\gamma)\right)\cdot\exp(\frac{1}{2}J^{t}\omega^{-1}J)\\
        &=\Str\left(\exp(-\frac{1}{2}K^{t}\gamma)\cdot(\sum_{i=1}^{n}\gamma_{i})\cdot \exp(\frac{1}{2}K^{t}\gamma)\cdot\exp(\frac{1}{2}\gamma^{t}\omega\gamma)\right)\cdot\exp(\frac{1}{2}J^{t}\omega^{-1}J)\\
        &=\Str\left((\sum_{i=1}^{n}\gamma_{i})\exp(\frac{1}{2}\gamma^{t}\omega\gamma)\right)\cdot\exp(\frac{1}{2}J^{t}\omega^{-1}J)+\Str\left((\sum_{i=1}^{n})\omega^{-1}(J)_{i}\cdot \exp(\frac{1}{2}\gamma^{t}\omega\gamma)\right)\cdot\exp(\frac{1}{2}J^{t}\omega^{-1}J)\\
        &=\Str\left(\exp(\frac{1}{2}\gamma^{t}\omega\gamma)\right)\cdot(\sum_{i=1}^{n}\omega^{-1}(J)_{i})\cdot\exp(\frac{1}{2}J^{t}\omega^{-1}J)\\
        &=(2i)^{\frac{n}{2}}\cdot\det\left(\frac{\sinh{\omega}}{\omega}\right)^{\frac{1}{2}}\cdot\left(\sum_{k=1}^{n}\sum_{I^{'}}\sum_{I}(-1)^{\frac{|I|+1}{2}}c_{k}\epsilon(I\cup\{k\},I^{'})\epsilon(\{k\},I)\Pf(\omega_{I^{'}})J^{I}\right).
    \end{align*}
    }
    The first equality holds because $\frac{1}{2}\gamma^{t}\omega\gamma+J^{t}\gamma=\frac{1}{2}(\gamma-\omega^{-1}(J))^{t}\omega(\gamma-\omega^{-1}(J))+\frac{1}{2}J^{t}\omega^{-1}J$, where the second term on the right is an even element so it commutes (strictly) with the first term.
    By Lemma \ref{l3}, we have:
    $$\frac{1}{2}(\gamma-\omega^{-1}(J))^{t}\omega(\gamma-\omega^{-1}(J))=\exp(\frac{1}{2}K^{t}\gamma)\cdot\exp(\frac{1}{2}\gamma^{t}\omega\gamma)\cdot\exp(-\frac{1}{2}K^{t}\gamma),$$
    and then the second equality holds.

     Because $\exp(\frac{1}{2}K^{t}\gamma)$ has an even degree, the third equality holds by the property of the supertrace.

     The fourth equality holds because of Lemma \ref{l3}. Since Lemma \ref{l2} implies that $\Str\left((\sum_{i=1}^{n}\gamma_{i})\exp(\frac{1}{2}\gamma^{t}\omega\gamma)\right)\cdot\exp(\frac{1}{2}J^{t}\omega^{-1}J)=0$, the fifth equality holds.

     Finally, the last equality is obtained by Lemmas \ref{l1} and \ref{l4}.
\end{proof}
\begin{remark}
In \cite{ref3}, Mathai and Quillen have proved that:

$$\Str(\exp(\frac{1}{2}\gamma^{t}\omega\gamma+J^{t}\gamma))=(2i)^{n/2}\cdot\det\left(\frac{\sinh{\omega}}{\omega}\right)^{\frac{1}{2}}\cdot\left(\sum_{I~even}\epsilon(I,I')(-1)^{\frac{1}{2}|I'|}\Pf(\omega_{I})J^{I'}\right)$$
    still holds when $\omega$ is a skew-symmetric matrix with entries in the even part of $A$ and $J$ is a vector with entries in the odd part of $A$, where either the entries of $\omega$ are nilpotent, or $A$ is a Banach algebra. After a similar argument, Theorem \ref{l4} still holds when  $\omega$ is a skew-symmetric matrix with entries in the even part of $\wedge J$ since $\wedge J$ is a Banach algebra.
\end{remark}
Finally, we will prove Theorem \ref{err}.
\begin{proof}[Proof of Theorem \ref{err}]
We only need to show that for $[\pi^{*}S^{+}_{V},\pi^{*}S^{-}_{V},\pi^{*}\sigma]\in K^{0}_{G}(V)$, we have 
$$\chg([\pi^{*}S^{+}_{V},\pi^{*}S^{-}_{V},\pi^{*}\sigma])=(-2\pi i)^{n/2}\pi^{*}(\hat{A}_{g}(V)^{-1})U_{g}(V),$$ 
 where $U_{g}(V)$ is an equivariant Thom form.

Let $\nabla_{S_{V}}:=d+\omega$ be the equivariant connection on $S_{V}$ associated with a connection on $V$. Consider the equivariant morphism $\sigma:S^{+}_{V}\rightarrow S^{-}_{V}$ defined by: 
$$\sigma(x):=-ic(x):S^{+}\rightarrow S^{-} $$
for $x\in$ V.

Let $\mathbb{A}=\pi^{*}\nabla_{S_{V}}$ be the invariant superconnection on $\pi^{*}S_{V}$.
Then for every $X\in\mathfrak{g}$,we have
$$\chg[\pi^{*}S^{ +}_{V},\pi^{*}S^{-}_{V},\pi^{*}\sigma](X)=\chi \ch(\mathbb{A})(X)+d\chi\beta(X)$$
where 
\begin{equation}
\label{eq:beta}
\beta(X)=\int_{0}^{\infty}\Str(c(x)\exp(-t^{2}||x||^{2}+[\mathbb{A},c(x)]+\mathbb{A}^{2}+\mu^{\mathbb{A}}(X)).
\end{equation}

For every $m\in M$, choose a local frame $e_{1},e_{2}......e_{n}$ near $m$. To simplify the notations, we denote $\tilde{\Omega}(X)=\mathbb{A}^{2}+\mu^{\mathbb{A}}(X)$ and $\gamma=(c(e_{1}),...,c(e_{n}))^{t}$. Thus, let $\Omega(X)$ be the equivariant curvature of $V$, we have: $\tilde{\Omega}(X)=\frac{1}{4}\gamma^{t}\Omega(X)\gamma$.  

Because of Proposition 2.10 in \cite{ref3} for a skew-symmetric matrix $\omega$, $\Str(e^{\frac{1}{2}\gamma^{t}\omega\gamma})=(2i)^{\frac{n}{2}}\det^{\frac{1}{2}}(\frac{\sinh\omega}{\omega})\Pf(\omega)$, we have:
\begin{align*}
    \chi \ch(\mathbb{A})(X)&=\chi \Str(\exp(\tilde{\Omega}(X))\\
    &=\Str((\exp(\frac{1}{4}\gamma^{t}\Omega(X)\gamma)\\
    &=\chi(-2\pi i)^{\frac{n}{2}}\hat{A}(\Omega(X))^{-1}\Pf\left(-\frac{\Omega(X)}{2\pi}\right).
\end{align*}
Using the  local frame, $\beta(X)$ in (\ref{eq:beta}) can be rewritten as 
$$\int_{0}^{\infty}\Str(\Sigma(x_{i}c(e_{i})\exp(-t^{2}||x||^{2}+t\Sigma J_{i}c(e_{i})+\frac{1}{4}\gamma^{t}\Omega(X)\gamma))$$
where $x=\sum_{i=1}^{n}x_{i}e_{i}$ and $J_{i}=dx_{i}+\pi^{*}(\omega^{l}_{i})x_{l}$. Here $\omega=\{\omega_{i}^{l}\}$ and $\omega_{i}^{l}\in \mathcal{A}^{1}(M)$.

By Theorem~\ref{l4}, we have:
\begin{align*}
\beta(X)=&\int_{0}^{\infty}\Str(\Sigma(x_{i}c(e_{i})\exp(-t^{2}||x||^{2}+t\Sigma J_{i}c(e_{i})+\frac{1}{4}\gamma^{t}\Omega(X)\gamma))\\
=&(2i)^{\frac{n}{2}}\det\left(\frac{\sinh{\Omega(X)/2}}{\Omega(X)/2}\right)^{\frac{1}{2}} \Gamma(X)
\end{align*}
where
$$\Gamma(X):=\int_{0}^{\infty}e^{-t||x||^{2}}\left(\sum_{k=1}^{n}\sum_{I^{'}}\sum_{I}(-1)^{\frac{|I|+1}{2}}x_{k}\epsilon(I\cup\{k\},I^{'})\epsilon(\{k\},I)\Pf(\Omega_{I'}(X)/2)J^{I}t^{|I|+1}\right)dt.$$
Thus we have:
\begin{multline*}
\chg[\pi^{*}S^{ +}_{V},\pi^{*}S^{-}_{V},\pi^{*}\sigma](X)=\chi \ch(\mathbb{A})(X)+d\chi\beta(X)\\
  =(-2\pi i)^{\frac{n}{2}}\hat{A}(\Omega(X))^{-1}\cdot\left(\chi\Pf\left(-\frac{\Omega(X)}{2\pi}\right)+(-\frac{1}{\pi})^{\frac{n}{2}}d\chi\cdot\Gamma(X)\right).
\end{multline*}

Finally  we show that $\chi \Pf(-\frac{\Omega(X)}{2\pi})+({-\frac{1}{\pi })^{\frac{n}{2}}}d\chi\Gamma(X)$ is a Thom form. 

Since $\Omega(X)$ is pulled back from the base space $M$, it implies that $\int _{fiber}\chi\Pf(-\frac{\Omega(X)}{2\pi})=0$. Then we choose $\chi=f(||x||^{2})$, where $f\in C^{\infty}(\mathbb{R})$ has compact support and is equal to 1 in a neighborhood of 0. Thus, we have 
$$\int _{fiber}\chi \Pf(\frac{\Omega(X)}{2})+(-\frac{1}{\pi} )^{n/2}d\chi\Gamma(X))=\int _{fiber}(-\frac{1}{\pi })^{n/2}d\chi\Gamma(X).$$

Since we consider the integration over the fiber, we focus on the component of maximal degree along the fiber of the differential form $d\chi\Gamma(X)$. Indeed, we have 

\begin{align*}
 &\int_{fiber}(-\frac{1}{\pi})^{n/2}d\chi\Gamma(X)\\
 =&\int_{fiber}\int_{0}^{\infty}\frac{1}{(\pi )^{n/2}}(-2)t^{n-1}f'(||x||^{2})||x||^{2}e^{-t||x||^{2}}dt\wedge{dx_{1}}\wedge{dx_{2}}\wedge{...}\wedge{dx_{n}}\\
=&\int_{fiber}\int_{0}^{\infty}-2t^{n-1}f'(||x||^{2})||x||^{2}e^{-t||x||^{2}}\wedge{dx_{2}}\wedge{...}\wedge{dx_{n}}\\
=&\frac{1}{\pi^{n/2}}\int_{fiber}e^{-||x||^{2}}\\
=&1.   
\end{align*}
This means that $\chi \Pf(-\frac{\Omega(X)}{2\pi})+({-\frac{1}{\pi })^{\frac{n}{2}}}d\chi\Gamma(X)$  is a Thom form, so we complete the proof.
 \end{proof}
\section{$G$-invariant Riemann-Roch formula on homogeneous spaces}
In this section, we discuss an application of the equivariant Riemann-Roch formula for the special case of homogeneous spaces.
We will first recall the Chern character in this special case defined by Connes and Moscovici \cite{ref8} and then rewrite it in the language of the equivariant Chern character. 
Then, we will prove the main theorem of the section:
\begin{theorem}
\label{thm6.1}
Let $G$ be an almost-connected unimodular Lie group, $H$ a maximal compact subgroup of $G$, and $G/H$ the corresponding homogeneous space. We assume that $G/H$  is even dimensional with a $G$-equivariant spin structure. The following diagram commutes:
\begin{equation*}
\xymatrix{   K_{G}^{0}(T^{*}(G/H))\ar[r]^{\pi!}\ar[d]_{\ch_{G}^{0}\wedge\hat{A}_{G}(G/H)} &  K_{G}^{0}(G/H) \ar[r]^{r_{*}}\ar[d]_{(-2\pi i)^{n/2}\ch_{G}^{0}} &R(H)\ar[d]_{(-2\pi i)^{n/2}\ch} \\
         H^{*}_{DR,cv}(T^{*}(G/H))^{G}\ar[r]^{\pi!}& H^{*}_{DR}(G/H)^{G}\ar[r]^{Id}&H^{*}(\mathfrak{g},H).
                      }
\end{equation*}
\end{theorem}

Assuming the assumptions in Theorem~\ref{thm6.1}, we first introduce an isomorphism from $K_{G}^{0}(G/H)$ to $K_{H}^{0}(pt)$:

\begin{definition}
\label{rk}
Let $[E]\in K_{G}^{0}(G/H)$, where $E$ is a $G$-equivariant vector bundle over $G/H$. There is an isomorphism from $K_{G}^{0}(G/H)$ to $K_{H}^{0}(pt)$ given by the restriction map:
\begin{align*}
    r_{*}:&K_{G}^{0}(G/H)\longrightarrow K_{H}^{0}(pt)\\
    &[E]\longmapsto [E|_{eH}]
\end{align*}
where $e$ is the identity of $G$ and $eH$ the left coset of $e$ in the homogeneous space $G/H.$
\end{definition}

Analogously, there is an isomorphism on the cohomology level from $H^{\infty}(\mathfrak{g}, G/H)$ to ${H^{\infty}(\mathfrak{{h}}, pt)}$, still denoted as $r_{*}$ :
\begin{definition}
There is an isomorphism from $H^{\infty}(\mathfrak{g},G/H)$ to ${H^{\infty}(\mathfrak{{h}},pt)}$ given by:
\label{rh}
\begin{align*}
    r_{*}:H^{\infty}(\mathfrak{g},G/H&)\rightarrow {H^{\infty}(\mathfrak{{h}},pt)}\cong{C^{\infty}(\mathfrak{h})^{H}}\\
    &\alpha\mapsto \alpha_{[0]}|_{eH}
\end{align*}
where $\alpha_{[0]}$ is the $0$-degree component of $\alpha$.
\end{definition}
\begin{remark}
    The claim that restriction maps in Definitions \ref{rk} and~\ref{rh} are isomorphisms have been proved in {\cite[Corollary 8.5, p131]{ref7}} and {\cite[Theorem 24, p33]{eq}} respectively.
\end{remark}

Then, we recall the Chern character $\ch_{CM}: R(H)\rightarrow H^{*}(\mathfrak{g}, H)$ defined by Connes and Moscovici in \cite{ref8}. Let $\mathfrak{g}$ be the Lie algebra of the Lie group $G$. Fix an $Ad(H)$-invariant splitting $\mathfrak{g}=\mathfrak{h}\oplus{\mathfrak{m}}$ of $\mathfrak{g}$. 
Then the projection $\theta:\mathfrak{g}\rightarrow\mathfrak{h}$ induces a $G$-invariant connection on the principal bundle $H\rightarrow G\rightarrow G/H$ which we denote by $\Tilde{\theta}$. We denote the corresponding curvature by $\Tilde{\Theta}$. When $\Tilde{\Theta}$ restricts to $eH$, for $X,Y\in \mathfrak{m}$, we have $\Tilde{\Theta}|_{eH}(X,Y)=-\frac{1}{2}\theta([X,Y]) $. 
Composing $\Tilde{\Theta}$ with $\epsilon:\mathfrak{h}\rightarrow \mathfrak{gl}(E)$, the differential of a unitary representation of $H$ on some vector space $E$, we have :
\begin{align*}
    \ch_{CM}:&R(H)\rightarrow H^{*}(\mathfrak{g},H)\\
    &\epsilon\mapsto\Tr(\exp(\frac{1}{2\pi i}\epsilon(\Theta))).
\end{align*}

Based on this special Chern character, we introduce the Chern character without normalization as follows:
\begin{definition}
\label{chn}
    \begin{align*}
    \ch:&R(H)\rightarrow H^{*}(\mathfrak{g},H)\\
    &\epsilon\mapsto\Tr(\exp(\epsilon(\Theta))).
\end{align*}
\end{definition}
This Chern character is very similar to $\ch_{CM}$ and we will demonstrate in the following that it is compatible with the equivariant Chern character.

In order to explain this viewpoint, we recall the equivariant Chern-Weil homomorphism defined by Duflo
and Vergne in {\cite[p20]{eq}}.

\begin{definition}[\cite{eq}]
Let $\Omega$ be a $G$-invariant curvature on the principle $H$ bundle $G$, the Chern-Weil map is given by:
\begin{align*}
    W_{\Omega}:C^{\infty}(\mathfrak{h})&^{H}\longrightarrow A^{\infty}_{G}(\mathfrak{g},G/H)\\
    &\phi\longmapsto\phi(\Omega).
    \end{align*}
    For any $X\in\mathfrak{g}$, we have $\phi(\Omega)(X):=\phi(\Omega(X))$.
\end{definition}

If we choose a different $G$-invariant curvature $\Omega_{1}$, one can check that for any $\phi\in C^{\infty}(\mathfrak{h})^{H}$, $[W_{\Omega}(\phi)]=[W_{\Omega_{1}}(\phi)]\in H^{\infty}(\mathfrak{g},G/H)$. Because of this, the Chern-Weil homomorphism induces a homomorphism from $H^{\infty}(\mathfrak{h}, pt)$ to $H^{\infty}(\mathfrak{g}, G/H)$, since $H^{\infty}(\mathfrak{h},pt)\cong C^{\infty}(\mathfrak{h})^{H}$.

\begin{definition}
Let $\Omega$ be a $G$-invariant curvature on the principle $H$ bundle $G$. One has the Chern-Weil homomorphism:
\begin{align*}
    W:H^{\infty}(\mathfrak{h},pt)&\longrightarrow H^{\infty}(\mathfrak{g},G/H)\\
    &\phi\longmapsto[\phi(\Omega)].
    \end{align*}
\end{definition}
\begin{remark}
This Chern-Weil homomorphism is an isomorphism. We will prove it later.
\end{remark}

After reviewing the special Chern character defined by Connes and Moscovici,  we will describe the non-normalized version in the language of the equivariant Chern character. First, we show the relation between the equivariant cohomology and the invariant cohomology. Let $M$ be a proper $G$ manifold, and $\alpha\in H^{\infty}(\mathfrak{g},M)$ be described as a smooth function valued in differential forms, compatible with the group action. With such a perspective, $\alpha(0)$ is a $G$-invariant differential form. (For $X\in\mathfrak{g}$, $\alpha(X)$ may not be $G$-invariant. But when $X=0$, $\alpha(0)$ is $G$-invariant.) We define the evaluation map as follows:
\begin{definition}
\label{ev}
The evaluation map is given by:
\begin{align*}
   \ev:A^{\infty}_{G}(\mathfrak{g},M&)\rightarrow A(M)^{G}\\
   &\alpha\mapsto\alpha(0),
\end{align*}
where $A(M)^{G}$ is the complex of $G$-invariant  differential form.
\end{definition}

Indeed, $\ev$ induces a homomorphism from $H^{\infty}(\mathfrak{g},M)$ to $H^{*}_{DR}(M)^{G}$. 
\begin{proposition}
    The evaluation map induces the  homomorphism:
\begin{align*}
    \ev_{*}:H^{\infty}(\mathfrak{g},&M)\rightarrow H^{*}_{DR}(M)^{G}\\
    &[\alpha]\mapsto [\alpha(0)].
\end{align*}
\end{proposition}
\begin{proof}
Suppose $[\alpha]=[\beta]\in H^{\infty}(\mathfrak{g},M)$. This means that exists an equivariant form $\gamma\in \mathcal{A}^{\infty}(\mathfrak{g},M)$ satisfying $\alpha-\beta=d_{\mathfrak{g}}(\gamma)$. Then we have:
$\ev(\alpha-\beta)=\ev(d_{\mathfrak{g}}(\gamma))=d_{\mathfrak{g}}(\gamma)(0)=d(\gamma(0))-\iota(0)(\gamma(0))=d(\gamma(0))$. Therefore, $\ev_{*}([\alpha])=\ev_{*}([\beta])$.
\end{proof}

Since the $G$-invariant differential form on $G/H$ is determined by its value at $eH$, we can identify $H^{*}_{DR}(G/H)^{G}$ with $H^{*}(\mathfrak{g},H)$. 
(In fact, we have $H^{*}_{DR}(G/H)^{G}\cong{H^{*}(\mathfrak{g},H)}$.) For this reason, we can describe $\ch$ in Definition \ref{chn} as the composition of equivariant Chern character, Chern-Weil homomorphism, and evaluation map:
\begin{proposition}
The following diagram commutes:
\begin{equation*}
\label{prop 5.9}
\xymatrix{   K^{0}_{H}(pt)\cong{R(H)}\ar[r]^{\ch_{H}}\ar[d]_{\ch} &    C^{\infty}(\mathfrak{h})^{H} \ar[d]^{W} \\
         H^{*}(\mathfrak{g},H) &H^{\infty}(\mathfrak{g},G/H) \ar[l]^{\ev_{*}}.
                      }
\end{equation*}
\end{proposition}
\begin{proof}
Let $E$ be an $H$-equivariant vector bundle over a point $pt$ and $\Tilde{\epsilon}:H\rightarrow \GL(E)$ be the corresponding representation of $H$. Here $\GL(E)$ means the group of bundle isomorphisms from $E$ to $E$. (For $s\in \Gamma(E)$ $h\in H$, we have $h\cdot{s}(pt)=\Tilde{\epsilon}(h)(s(pt))$.) Since $pt$ is a $0$-dimensional manifold, any $H$-invariant curvature $\Omega$ vanishes. Therefore, $\Omega(X)=\Omega+\mu(X)=\mu(X)$ where $X\in \mathfrak{h}$ and $\mu$ is the moment associated to the $H$-invariant connection $\omega$. 
Since $\mu(X)=\epsilon(X)$ where $\epsilon:\mathfrak{h}\rightarrow \mathfrak{gl}(E)$ is the differential of $\Tilde{\epsilon}$, we have $\ch_{H}[E,\Tilde{\epsilon}](X)=\Tr(\exp(\epsilon(\Omega(X)))=\Tr(\exp(\epsilon(X)))$ where $[E,\Tilde{\epsilon}]\in K_{H}^{0}(pt)$ and $X\in \mathfrak{h}$. For this reason, choosing a $G$-invariant curvature $\Omega\in A^{2}(G)\otimes{\mathfrak{h}}$, we have 
\[
\ev_{*}\circ W\circ\ch_{H}[E,\Tilde{\epsilon}]=\Tr(\exp(\epsilon(\Omega(0)))=\Tr(\exp(\epsilon(\Omega))\in H^{*}(\mathfrak{g},H).
\] 
So we obtain $\ch=\ev_{*}\circ W\circ\ch_{H}$, which completes the proof.
\end{proof}

Next, we show a result which will imply the commutativity of part of the diagram in Theorem~\ref{thm6.1}:
\begin{proposition}
The following diagram commutes:
\begin{equation*}
\label{thm 5.10}
\xymatrix{   K_{G}^{0}(G/H)\ar[r]^{r_{*}}\ar[d]_{\ch_{G}} &  K_{H}(pt) \ar[d]^{\ch_{H}} \\
         H^{\infty}(\mathfrak{g},G/H)\ar[r]^{r_{*}}&C^{\infty}(\mathfrak{h})^{H}.
                      }
\end{equation*}
\end{proposition}
\begin{proof}
For $[E,\Tilde{\epsilon}]\in K_{G}^{0}(G/H)$ where $E$ is a $G$-equivariant vector bundle over $G/H$ and $\Tilde{\epsilon}:G\rightarrow\GL(E)$ is the representation associated to $E$, we have $r_{*}([E,\Tilde{\epsilon}])=[E|_{eH},\Tilde{\epsilon}|_{H}]$. Besides, for any $X\in \mathfrak{h}$ we have $\ch_{H}[E|_{eH},\Tilde{\epsilon}|_{H}](X)=\Tr(\exp(\epsilon|_{H}(X)))|_{eH}$ where $\epsilon|_{H}:\mathfrak{h}\rightarrow \mathfrak{gl}(E)$ is the differential of $\Tilde{\epsilon}|_{H}$. Thus, we have $\ch_{H}\circ r_{*}[E,\Tilde{\epsilon}](X)=\Tr(\exp(\epsilon|_{H}(X)))$ for any $X\in \mathfrak{h}$. On the other hand, we know that $\ch_{G}[E,\Tilde{\epsilon}](Y)=\Tr(\exp(\Omega(Y)))$ where $\Omega$ is a $G$-invariant curvature of $E$ and $Y\in \mathfrak{g}$. Since the $0$-degree component of $\Tr(\exp(\Omega(Y)))$ is $\Tr(\exp(\mu(Y)))$ where $\mu$ is the moment map, we have $r_{*}\circ\ch_{G}[E,\Tilde{\epsilon}](X)=\Tr(\exp(\mu(X)))$. Because $X\in \mathfrak{h}$, we have $\mu(X)=\epsilon|_{H}(X)$. Therefore, we have $r_{*}\circ\ch_{G}[E,\Tilde{\epsilon}](X)=\Tr(\exp(\epsilon|_{H}(X)))$. So far, we have completed the proof.
\end{proof}

Next, we compose the equivariant Chern character with the evaluation map:
\begin{definition}The $G$-invariant Chern character is given by:
\begin{align*}
    \ch_{G}^{0}:&K_{G}^{0}(M)\rightarrow H^{*}_{DR}(M)^{G}\\
    &[E,F,\sigma]\mapsto \ev_{*}\circ\ch_{G}[[E,F,\sigma]]
\end{align*}
where $M$ is a proper $G$-manifold and $ev_{*}$ is defined in Definition \ref{ev} and $\chg$ is defined in (\ref{ch1}).
\end{definition}
Then we get a special case of Theorem \ref{err} which is also the commutativity of the other part of the diagram in Theorem \ref{thm6.1}:
\begin{proposition}
The following diagram commutes:
\begin{equation*}
\label{thm 5.12}
\xymatrix{   K_{G}^{0}(V)\ar[r]^{\pi!}\ar[d]_{\ch_{G}^{0}\wedge{\hat{A}_{G}(V)}} &  K_{G}^{0}(M) \ar[d]^{(-2\pi i)^{n/2}\ch_{G}^{0}} \\
         H^{*}_{DR,cv}(V)^{G}\ar[r]^{\pi!}&H^{*}_{DR}(M)^{G}
                      }
\end{equation*}
where $V$ is a $G$-equivariant vector bundle over $M$ and $M$ satisfies conditions in Theorem~\ref{sk}.
\end{proposition}

At the end of this section, we summarize the main result of the section by combining Proposition~\ref{prop 5.9}, Proposition~\ref{thm 5.10}, and Proposition~\ref{thm 5.12}:

\begin{proof}[Proof of Theorem \ref{thm6.1}]
From Proposition~\ref{thm 5.10} we have the commutative diagram:
\begin{equation*}
    \xymatrix{   K_{G}^{0}(G/H)\ar[r]^{r_{*}}\ar[d]_{\ch_{G}} &  K_{H}(pt) \ar[d]^{\ch_{H}} \\
         H^{\infty}(\mathfrak{g},G/H)\ar[r]^{r_{*}}&C^{\infty}(\mathfrak{h})^{H}.
                      }
\end{equation*}

First, we show that $W:C^{\infty}(\mathfrak{h})^{H}\rightarrow  H^{\infty}(\mathfrak{g},G/H)$ is the inverse of $r_{*}:H^{\infty}(\mathfrak{g},G/H)\rightarrow C^{\infty}(\mathfrak{h})^{H}$. Since $r_{*}$ is an isomorphism, just verifying $r_{*}\circ{W}=\Id$ is enough to prove our claim. For $\phi\in C^{\infty}(\mathfrak{h})^{H}$ and $X\in \mathfrak{g}$, we have $W(\phi)(X)=\phi(\Omega)(X)=\Sigma_{I}\frac{1}{|I|!}\Omega^{I}\partial_{I}(\phi)|_{\tilde{\mu}(X)}$ where $\Omega$ is the $G$-invariant curvature over the principle bundle $G$ and $\tilde{\mu}$ is the moment map associated with $\Omega$. Therefore, we have $r_{*}\circ W(\phi)(X)=r_{*}(\Sigma_{I}\frac{1}{|I|!}\Omega^{I}\partial_{I}(\phi)|_{\tilde{\mu}(X)})=\frac{1}{0!}\Omega^{0}\partial_{0}(\phi)|_{\tilde{\mu}(X)}=\phi(X)$ where $X\in \mathfrak{h}$.(Because $\tilde{\mu}(X)=X$ when $X\in \mathfrak{h}$.) Then we have proved $r_{*}\circ{W}=\Id$.

Because of this, we have: 
\begin{equation*}
\xymatrix{   K_{G}^{0}(G/H)\ar[r]^{r_{*}}\ar[d]_{\ch_{G}} &  K_{H}^0(pt) \ar[d]^{\ch_{H}} \\
         H^{\infty}(\mathfrak{g},G/H)&C^{\infty}(\mathfrak{h})^{H} \ar[l]^{W}.
                      }
\end{equation*}
Then composing the  evaluation map with the equivariant Chern character map, we have: 
\begin{equation*}
\xymatrix{   K_{G}^{0}(G/H)\ar[r]^{r_{*}}\ar[d]_{\ch_{G}^{0}} &  K_{H}^0 (pt) \ar[d]^{\ch_{H}} \\
         H^{*}_{DR}(G/H)^{G} &C^{\infty}(\mathfrak{h})^{H} \ar[l]^{\ev_{*}\circ W}.
                      }
\end{equation*}

Since $\ch=\ev_{*}\circ W\circ\ch_{H}$ (Proposition \ref{prop 5.9}), we have:
\begin{equation*}
\xymatrix{    K_{G}^{0}(G/H) \ar[r]^{r_{*}}\ar[d]_{\ch_{G}^{0}} &R(H)\ar[d]_{\ch} \\
          H^{*}_{DR}(G/H)^{G}\ar[r]^{Id}&H^{*}(\mathfrak{g},H).
                      }
\end{equation*}

Because of this and Proposition \ref{thm 5.12}, we have:
\begin{equation*}
\xymatrix{   K_{G}^{0}(T^{*}(G/H))\ar[r]^{\pi!}\ar[d]_{\ch_{G}^{0}\wedge\hat{A}_{G}(G/H)} &  K_{G}^0(G/H) \ar[r]^{r_{*}}\ar[d]_{(-2\pi i)^{n/2}\ch_{G}^{0}} &R(H)\ar[d]_{(-2\pi i)^{n/2}\ch} \\
         H^{*}_{DR,cv}(T^{*}(G/H))^{G}\ar[r]^{\pi!}& H^{*}_{DR}(G/H)^{G}\ar[r]^{Id}&H^{*}(\mathfrak{g},H).
                      }
\end{equation*}
Therefore, the proof is complete.
\end{proof}

\section{Local index formula and Getzler’s symbolic calculus}

In this section, we shall use the heat kernel approach to prove the following local index formula:
\begin{theorem}
\label{mthm2}
    For every $f\in C^{2q}_{v,anti}(M)^{G}$ (c.f. Definition \ref{def1}) and $\Ind_{t}(D)\in S^{u}_{G}(M,E)$ (c.f. Definition \ref{def2}), we have 
    $$\lim_{t\rightarrow0}\tau(f)(\Ind_{t}(D))=\frac{(-1)^{n/2-q}}{(2\pi i)^{2q-n/2}}\frac{q!}{(2q)!}\int_{M}c_{0}\Hat{{A}}_{G}(M)\wedge\ch_{G}^{0}(\pi!([\pi^{*}(E^{+}),\pi^{*}(E^{-}),\sigma(D)]))\wedge \omega_{f}$$
    where $$\omega_{f}:=(d_{1}...d_{2q}f)|_{\Delta}.
$$
Here $d_{i}$ stands for the differential concerning the $i$-th variable of the function $f$, $\dim(M)=n$($n$ is even) and $\Delta: M\rightarrow M^{\times (2q+1)}$ is the diagonal embedding.
\end{theorem}

  The local index formula of this form was proved first by Connes and Moscovici in \cite{ref4}. In their case, the group is assumed to be discrete and the action is free. When the group is a Lie group having finitely many components and when the action is proper and cocompact, Piazza and Posthuma give a brief proof to a similar index formula in \cite{pp2} as Theorem \ref{mthm2}. Their proof is based on the results of Moscovici and Wu \cite{wu}. However, in this paper we use the index class represented by the Wassermann projector as in \cite{ref4}, which is not the idempotent of finite propagation speed representing the index class, from which Moscovici and Wu obtained the higher index formula \cite{wu}. Our proof of Theorem~\ref{mthm2} then follows by adapting the original proof of the localized index formula in \cite{ref4}.

  \subsection{Pairing between $C^{*}_{v}(M)^{G}$ and $S^{u}_{G}(M,E)$}In this subsection, we will define $C^{*}_{v}(M)^{G}$ and $S^{u}_{G}(M,E)$ and their pairing.

   Choosing an complete invariant Riemannian metric, we denote the associated distance function by $d(x,y)$ for $x, y \in  M$. Let the Riemannian volume form associated with the invariant complete Riemannian metric be denoted by $dx$ and let $d\mu(x)$ be the  Riemannian densities induced by $dx$ i.e. $d\mu(x)=|dx|$. Because the Riemannian metric is $G$-invariant, $d\mu(x)$ is also $G$-invariant.

\begin{definition}
\label{def1}
For $v\geq 0$ and $q\in \mathbb{N}^{+}$, let $C_{v}^{q}(M)^{G}$ be the vector space consisting of $G$-equivariant smoothing functions $\psi$ from $M^{q+1}$ to $\mathbb{C}$ satisfying $$\sup(|\psi(x^{0},...,x^{q})\cdot \exp(-v\cdot(\sum_{i=0}^{q}d(x_{i},z_{0})))|)< +\infty$$
and let a coboundary homomorphism $d:C_{v}^{q}(M)^{G}\rightarrow{C_{v}^{q+1}(M)^{G}}$ be given by 
$$d(\psi)(x^{0},...,x^{p+1}):=\sum_{i=0}^{q+1}(-1)^{i}\psi(x^{0},...,x^{i-1},x^{i+1},...,x^{q+1}),$$
where $\psi\in C_{v}^{q}(M)^{G}$. The corresponding cohomology is denoted by $H_{v}^{*}(M)^{G}$. 
\end{definition}
\begin{remark}
     Under the notation above, let 
 \begin{align*}
     C^{q}_{v,\lambda}(M)^{G}:=\{&\psi\in C^{q}_{v}(M)^{G}|\\   
     &\psi(x_{0},...,x_{q})=(-1)^{q}\psi(x_{q},x_{0},...,x_{q-1}),\text{ for all }x_{0},...,x_{q}\in M\}
 \end{align*}
 and
  \begin{align*}
     C^{q}_{\diff,anti}(M)^{G}:=\{&\psi\in C^{q}_{v}(M)^{G}|\\
     &\psi(x_{0},...,g_{k})=\sign(\sigma)\psi(x_{\sigma(0)},x_{\sigma(1)},...,x_{\sigma(q)}),\\
     &\text{ for all }x_{0},...,x_{q}\in M \text{ and } \sigma\in S_{q+1}\}.
     \end{align*}
\end{remark}
\begin{definition}
     Let $E$ be a $G$-bundle over $M$. For every $x,y\in M$, we have $k(x,y)\in \Hom(E_{y},E_{x})$, and it carries the matrix norm, i.e., 
     $\|k(x,y)\|=\|k(x,y)\|_{M_{n}(\mathbb{C})}$.
\end{definition}

\begin{definition}
\label{def2}
For $u\geq 0$, define $S^{u}_{G}(M;E)$ to be
    \begin{align*}
\{k\in C^{\infty}(M\times M,E\boxtimes E^{*})^{G}|&\sup_{x,y\in M}||\exp{(u\cdot d(x,y))}\nabla_{x}^{m}\nabla_{y}^{n}k(x,y)||<C_{m,n}\\
& \text{ for all }m,n\in \mathbb{N}\}.
\end{align*}
\end{definition}
Since  $M$ is a $C^{\infty}$-manifold acted properly and cocompactly by an almost-connected Lie group $G$, there exist $M_{1}$ and $r_{M}>0$, such that for any $x\in M$ and $r\leq 0$, we have $\vol (B_{r}(x))\leq M_{1}e^{r_{M}\cdot r}$,  where $\vol(B_{r}(x))$ is the volume of $B_{r}(x)$. Thus, we can define a pairing between $S^{u}_{G}(M;E)$ and $ C_{v}^{q}(M)^{G}$ by choosing suitable $u$ and $v$:
\begin{proposition}
Let $A_{0},...,A_{q}\in S^{u}_{G}(M;E)$.  For each $\psi\in C_{v}^{q}(M)^{G}$ we denote:
$$\tau(\psi)(A_{0},...,A_{q}):=\int_{M^{q+1}}c(x_{0})\psi(x_{0},...,x_{q})\tr(A_{0}(x_{0},x_{1})A_{1}(x_{1},x_{2})...A_{q}(x_{q},x_{0}))d\mu(x_{0})...d\mu(x_{q})$$
and when $A_{0}=A_{1}=...=A_{q}$, we denote $\tau(\psi)(A_{0},...,A_{q})$ by $\tau(\psi)(A_{0})$.
We claim that if $qv-u<-r_{m}$, then $\tau(\psi)(A_{0},...,A_{q})$ is well defined.
\end{proposition}
\begin{proof}
    We show that $\tau(\psi)$ is well defined. 
    Because $A_{i}\in S^{u}_{G}(M;E)$ and $\psi\in C_{v}^{q}(M)^{G}$, we have $||A_{i}(x,y)||\leq \exp(-u\cdot(d(x,y)))$ and $|\psi(x^{0},...,x^{q})|\leq \exp(v\cdot(\sum_{i=0}^{q}d(x_{i},z_{0})))$. Thus, 
\begin{align*}
    &|\tau(\psi)(A_{0},...,A_{q})|\\
    &=|\int_{M^{q+1}}c(x_{0})\psi(x_{0},...,x_{q})\tr(A_{0}(x_{0},x_{1})A_{1}(x_{1},x_{2})...A_{q}(x_{q},x_{0}))d\mu(x_{0})...d\mu(x_{q})|\\
    &\leq \int_{M^{q+1}}c(x_{0})\exp(v\cdot(\sum_{i=0}^{q}d(x_{i},z_{0}))))\cdot\exp(-u(\cdot d(x_{0},x_{1})+...+d(x_{q},x_{0})))d\mu(x_{0})...d\mu(x_{q})\\
&\leq\int_{M^{q+1}}c(x_{0})\exp(nv\cdot d(x_{0},z_{0})-r_{m}\cdot(d(x_{0},x_{1})+...+d(x_{q},x_{0})))d\mu(x_{0})...d\mu(x_{q})\\
    &<+\infty.
\end{align*} 
\end{proof}

\subsection{Schwartz kernel estimates for smoothing operators}
Let $D:C^{\infty}_{c}(M,E)\rightarrow C^{\infty}_{c}(M,E)$ be the $G$-equivariant Dirac operator. Following Proposition 2.10 in \cite{roe}, for any $f\in \mathcal{S}(\mathbb{R})$, $f(D): L^{2}(M, E)\rightarrow L^{2}(M, E)$ is a bounded operator with a smooth Schwartz kernel. In this subsection, we will estimate the Schwartz kernel for such $f(D)$.
\begin{definition}
    Let $E$ be a $G$-bundle over $M$. For any $s\in L^{2}(M,E)$, we denote:
    $$||s||_{0,B_{r}(x)}:=\int_{B_{r}(x)}|s(y)|^{2}d\mu(y)
    $$
    where $|s(y)|^{2}=(s(y),s(y))$.
\end{definition}

To estimate the Schwartz kernel for $f(D)$, we need the following two lemmas:
\begin{lemma}
\label{c1}
    Let $E$ be a $G$-bundle over $M$, and $D:C^{\infty}_{c}(M,E)\rightarrow C^{\infty}_{c}(M,E)$ be the $G$-equivariant Dirac operator . There exist $C>0$ and $N\in\mathbb{N}^{+}$ depending on $n=\dim E$, such that for any $0<r\leq 1$ and $x\in M$, we have:
    $$|s(x)|\leq \frac{C}{r^{N}}\left(\sum_{p=0}^{n}||D^{p}s||_{0,B_{r}(x)}\right).
    $$
    \end{lemma}
    \begin{lemma}
    \label{c2}
      For any $k\in C^{\infty}(M\times M,E\boxtimes E^{*})^{G}$, $y_{0}\in M$ and $r>0$, we have:
      $$\frac{1}{n}||k(x,y)||_{0,B_{r}(y_{0})}\leq \sup_{s=1,~\supp(s)\subset B_{r}(y_{0})}|\int_{B_{r}(y_{0})}k(x,y)s(y)dy|.$$
    \end{lemma}
They have both been proved in Appendix A.
\begin{theorem}
\label{g1}
Let $D$ be the equivariant Dirac operator. There exist $C>0$ and $N\in\mathbb{N}^{+}$, so that for any distinct $x,y\in M$ and  $f\in S(\mathbb{R})$, we have:
    $$||K_{f(D)}(x,y)||\leq \frac{C}{d(x,y)^{N}}\sum_{k=0}^{n+n^2}\int_{|s|\geq \frac{d(x,y)}{100}}|\hat{f}^{(k)}(s)|ds.$$
Here $n=\dim E.$    
\end{theorem}

\begin{proof}
    Let $x_{0},y_{0}\in M$ where $x_{0}\neq y_{0}$. First we consider the case when $d(x_{0},y_{0})\leq 1$. Let $r=\frac{d(x_0,y_0)}{100}$. 

    For any $x_{0}\in M$, let $\Hom(E,E_{x_0})=E_{x_{0}}\otimes{E^{*}}$ be a $G$-equivariant bundle over $M$ and $D_{y}:\Gamma(M,\Hom(E,E_{x_0}))\rightarrow \Gamma(M,\Hom(E,E_{x_0}))$ be the $G$-equivariant Dirac operator. Thus, $k(x_{0},y)$ can be viewed as a smooth section of $\Hom(E,E_{x_0})$ for any $k(x,y)\in \Gamma(M\times M, E\boxtimes E^{*})$.

    It follows from Lemma \ref{c1} that exist $C_{1}>0$ and $N_{1}\in \mathbb{N}^{+} $ such that 
    $$||K_{f(D)}(x_{0},y_{0})||\leq \frac{C_{1}}{r^{N_{1}}}\left(\sum_{p=0}^{n^2}||D_{y}^{p}K_{f(D)}(x_{0},y)||_{0,B_{r}(y_{0})}\right). $$
  The numbers $C_{1}$ and $N_{1}$ are independent of $x_{0}$ and $y_{0}$ because $\Hom(E, E_{x})\cong \Hom(E, E_{y})$ for any $x,y\in M$ and the bundle isomorphism can be chosen to be an isometry.
    It follows from Lemma \ref{c2} that
    \begin{align*}
        ||D_{y}^{p}K_{f(D)}(x_{0},y)||_{0,B_{r}(y_{0})}\leq n^{2}\sup_{||s||=1, \supp(s)\subset B_{r}(y_{0})}|D^{p}_{y}f(D)s(x_{0})|.\\
    \end{align*}
    Here $D^{p}_{y}f(D)$ is a linear operator defined as follows:
    \begin{align*}
        &D^{p}_{y}f(D): C^{\infty}_{c}(M,E)\rightarrow \Gamma(M,E)\\
        &s\mapsto \int_{M}D_{y}^{p}K_{f(D)}(x,y)s(y)d\mu(y).
    \end{align*}
     We need to show that $D^{p}_{y}f(D)s\in \Gamma(M,E)$. Indeed, we have $D^{p}_{y}f(D)=f(D)(D^{*})^{p}$. Since for  any $s\in C^{\infty}_{c}(M,E)$ and $x\in M$, without loss of generality, let $k_{f(D)}(x,y)=f(x,y)u(x)\otimes v(y)^{*}\in \Gamma(M\times M, E\boxtimes E^{*})$ where $u(x), v(y)\in \Gamma(M,E)$ and $f(x,y)\in C^{\infty}(M\times M)$, we have:
    \begin{align*}
        \int_{M}D_{y}(f(x,y)u(x)\otimes v(y)^{*})s(y)d\mu(y)&=\int_{M}u(x)\otimes (D(f(x,y)v)(y))^{*}s(y)d\mu(y)\\
        &=\int_{M}u(x)(D(f(x,y)v)(y),s(y))d\mu(y)\\
        &=\int_{M}u(x)(f(x,y)v(y),D^{*}s(y))d\mu(y)\\
        &=\int_{M}f(x,y)u(x)\otimes{v(y)^{*}}(D^{*}(s))(y)d\mu(y).
    \end{align*}

    Thus, we have $D^{p}_{y}K_{f(D)}(x,y)=K_{f(D)(D^{*})^{p}}(x,y)$ for any $x,y\in M$. Because of this, we have $D^{p}_{y}f(D)=f(D)(D^{*})^{p}$.
    
    Then applying Lemma \ref{c1} again, there exist $C_{2}>0$ and $N_{2}\in\mathbb{N}^{+}$ such that
    $$|D^{p}_{y}f(D)s(x_{0})|\leq \frac{C_{2}}{r^{N_{2}}}
\left(\sum_{k=0}^{n}||D^{k}D_{y}^{p}f(D)s||_{0,B_{r}(x_{0})}\right).
    $$
    Then, we have 
    $$||K_{f(D)}(x_{0},y_{0})||\leq \frac{C_{1}C_{2}}{r^{N_{1}+N_{2}}}\left(\sum_{0\le k\le n, 0\le p\le n^{2}}
    \sup_{||s||=1, \supp(s)\subset B_{r}(y_{0})}||D^{k}D_{y}^{p}f(D)s||_{0,B_{r}(x_{0})}\right).
    $$
    For a subset $A$ of $M$, let $\chi_{A}:M\rightarrow [0,1]$ be the indicator function of $A$. Thus, we have 
   $$\sup_{||s||=1, \supp(s)\subset B_{r}(y_{0})}||D^{k}D_{y}^{p}f(D)s||_{0,B_{r}(x_{0})}=||\chi_{B_{r}(x_{0})}D^{k}D_{y}^{p}f(D)\chi_{B_{r}(y_{0})}||_{0}.$$
   Note that 
   \begin{align*}
       &||\chi_{B_{r}(x_{0})}D^{k}D_{y}^{p}f(D)\chi_{B_{r}(y_{0})}||_{0}\\
       &=\sup_{||s_{1}||=1, ||s_{2}||=1}|\int_{M}(\chi_{B_{r}(x_{0})}D^{k}D_{y}^{p}f(D)\chi_{B_{r}(y_{0})}s_{1}(x),s_{2}(x))d\mu(x)|\\
       &=\sup_{||s_{1}||=1, ||s_{2}||=1}|\int_{M}(\chi_{B_{r}(x_{0})}D_{y}^{p}D^{k}f(D)\chi_{B_{r}(y_{0})}s_{1}(x),s_{2}(x))d\mu(x)|\\
       &=\sup_{||s_{1}||=1, ||s_{2}||=1}|\int_{M}(s_{1}(x),\chi_{B_{r}(y_{0})}D^{p}(D^{*})^{k}f(D^{*})\chi_{B_{r}(x_{0})}s_{2}(x))d\mu(x)|\\
       &=||\chi_{B_{r}(y_{0})}D^{p}(D^{*})^{k}f(D^{*})\chi_{B_{r}(x_{0})}||_{0}
   \end{align*}
and because $M$ is complete and $D$ is essentially self-adjoint, we have:
$$||\chi_{B_{r}(y_{0})}D^{p}(D^{*})^{k}f(D^{*})\chi_{B_{r}(x_{0})}||_{0}=||\chi_{B_{r}(y_{0})}D^{p}D^{k}f(D)\chi_{B_{r}(x_{0})}||_{0}.$$

 Then, we have:
   \begin{align*}
       ||K_{f(D)}(x_{0},y_{0})||&\leq \frac{C_{1}C_{2}}{r^{N_{1}+N_{2}}}(\sum_{0\leq k\leq n, 0\leq p\leq n^{2}}
    \sup_{||s||=1, \supp(s)\subset B_{r}(y_{0})}||D^{k}D_{y}^{p}f(D)s||_{0,B_{r}(x_{0})})\\
    &=\frac{C_{1}C_{2}}{r^{N_{1}+N_{2}}}(\sum_{0\leq k\leq n, 0\leq p\leq n^{2}}||\chi_{B_{r}(y_{0})}D^{p}(D^{*})^{k}f(D^{*})\chi_{B_{r}(x_{0})}||_{0})\\
    &=\frac{C_{1}C_{2}}{r^{N_{1}+N_{2}}}\left(\sum_{0\leq k\leq n, 0\leq p\leq n^{2}}||\chi_{B_{r}(y_{0})}D^{p+k}f(D)\chi_{B_{r}(x_{0})}||_{0}\right).
   \end{align*}
   Next, we will prove:
   $$||\chi_{B_{r}(y_{0})}D^{p+k}f(D)\chi_{B_{r}(x_{0})}||_{0}\leq \frac{1}{2\pi} \int_{|s|>r}|\hat{f}^{(p+k)}|(s)ds.$$

Since $D$ has unit propagation speed, for any $u\in L^{2}(M,E)$ with $\supp(u)\subset B_{r}(x_{0})$, we have $\supp(e^{isD}u)\subset B_{r+|s|}(x_{0})$. Thus, when $|s|<r$, for any $u\in L^{2}(M,E)$ with $\supp(u)\subset B_{r}(x_{0})$ and $y\in B_{r}(y_{0})$ we have $e^{isD}(u)(y)=0$.($r=\frac{d(x_{0},y_{0})}{100})$.

Then for any $u\in L^{2}(M,E)$, $\supp(u)\subset B_{r}(x_{0})$ and $y\in B_{r}(y_{0})$, we have:
\begin{align*}
    D^{p}D^{k}f(D)(u)(y)&=\left(\frac{(i)^{p+k}}{2\pi}\int_{\R}e^{isD}\hat{f}^{(p+k)}(s)ds\right)(u)(y)\\
    &=\left(\frac{(i)^{p+k}}{2\pi}\int_{|s|>r}e^{isD}\hat{f}^{(p+k)}(s)ds\right)(u)(y).
\end{align*}
Therefore, we have \[
\chi_{B_{r}(y_{0})}D^{p}D^{k}f(D)\chi_{B_{r}(x_{0})}=\chi_{B_{r}(y_{0})}\left(\frac{(i)^{p+k}}{2\pi}\int_{|s|>r}e^{isD}\hat{f}^{(p+k)}(s)ds\right)\chi(B_{r}(x_{0})).
\]

Then we have:
\begin{align*}
    &||\chi_{B_{r}(y_{0})}D^{p}D^{k}f(D)\chi_{B_{r}(x_{0})}||_{0}\\
    &=||\chi_{B_{r}(y_{0})}\frac{(i)^{p+k}}{2\pi}(\int_{|s|>r}e^{isD}\hat{f}^{(p+k)}(s)ds)\chi_{B_{r}(x_{0})}||_{0}\\
    &\leq ||\chi_{B_{r}(y_{0})}||_{0}||(\frac{1}{2\pi}\int_{|s|>r}e^{isD}\hat{f}^{(p+k)}(s)ds)||_{0}||\chi_{B_{r}(x_{0})}||_{0}\\
    &\leq \frac{1}{2\pi}\int_{|s|>r}|\hat{f}^{(p+k)}|(s)ds.
\end{align*}
Therefore, for any $x,y\in M$, $x\neq y$ satisfying $d(x,y)\leq 1$, we obtain:
    $$||K_{f(D)}(x,y)||\leq \frac{C_{1}C_{2}}{(\frac{d(x,y)}{100})^{N_{1}+N_{2}}}\sum_{k=0}^{n+n^2}\int_{|s|\geq \frac{d(x,y)}{100}}|\hat{f}^{(k)}(s)|ds.$$
    Next, consider the case when  $d(x,y)\geq 1$.

    When $d(x,y)\geq 1$, we fix $r=\frac{1}{100}$. Similarly, we have:
    $$||K_{f(D)}(x,y)||\leq \frac{C_{1}C_{2}}{(\frac{1}{100})^{N_{1}+N_{2}}}\left(\sum_{0\leq k\leq n, 0\leq k\leq n^{2}}||\chi_{B_{\frac{1}{100}}(y)}D^{p+k}f(D)\chi_{B_{\frac{1}{100}}(x)}||_{0}\right).
    $$
    Finally, we prove 
    $$||\chi_{B_{\frac{1}{100}}(y)}D^{p+k}f(D)\chi_{B_{\frac{1}{100}}(x)}||_{0}\leq\int_{|s|>\frac{d(x,y)}{100}}|\hat{f}^{(p+k)}(s)|ds.$$

    Since $d(x,y)\geq 1$,  we have $B_{\frac{1}{100}+\frac{d(x,y)}{100}}(x)\subseteq B_{\frac{2d(x,y)}{100}}(x)\cap  B_{\frac{1}{100}}(y)=\emptyset$.

    Then we have: when $|s|<\frac{d(x,y)}{100}$, for any $u\in L^{2}(M,E), \supp(u)\subset B_{\frac{1}{100}}(x)$ and $z\in B_{\frac{1}{100}}(y)$ we have $e^{isD}(u)(z)=0$. 

    Therefore, for any $u\in L^{2}(M,E), \supp(u)\subset  B_{\frac{1}{100}}(x)$ and $z\in  B_{\frac{1}{100}}(y)$, we have:
\begin{align*}
    D^{p}D^{k}f(D)(u)(z)&=\left(\frac{(i)^{p+k}}{2\pi}\int_{\R}e^{isD}\hat{f}^{(p+k)}(s)ds\right)(u)(z)\\
    &=\left(\frac{(i)^{p+k}}{2\pi}\int_{|s|>\frac{d(x,y)}{100}}e^{isD}\hat{f}^{(p+k)}(s)ds\right)(u)(z).
\end{align*}
    Similarly, we have  
    $$||\chi_{B_{\frac{1}{100}}(y)}D^{p+k}f(D)\chi_{B_{\frac{1}{100}}(x)}||_{0}\leq\frac{1}{2\pi}\int_{|s|>\frac{d(x,y)}{100}}|\hat{f}^{(p+k)}(s)|ds,$$ 
    which completes the proof.
\end{proof}

When $x=y$, Theorem \ref{g1} still holds because the right hand side is infinite. In the next theorem, we have a finer estimate when $x=y.$
 The proof is similar (we just need to fix $r=1$ in Theorem \ref{g1}). 
\begin{theorem}
\label{g2}
    Let $D$ be the equivariant Dirac operator. There exists $C>0$ , such that for any $x,y\in M$ and $f\in S(\R)$, we have:
    $$||K_{f(D)}(x,y)||\leq C\sum_{k=0}^{n+n^2}\int_{\mathbb{R}}|\hat{f}^{(k)}(s)|ds.$$
\end{theorem}
\begin{remark}
    Note that the integral is over $\mathbb{R}$. This is because we have fixed $r=1$ and the method of finite propagation speed cannot be applied.
\end{remark}

For any $f\in S(\mathbb{R})$ and $t>0$, define $f_{t}(x):=f(tx)$. Then we have $f_{t}\in S(\mathbb{R})$, and $\hat{f_{t}}(s)=t^{-1}\hat{f}(t^{-1}s)$ for any $s\in \mathbb{R}$. Then we have:
\begin{theorem}
\label{g3}
    Suppose $f\in S(\R)$ satisfies the following condition: There exist $M\geq0$ and $w>0$, such that for every $k\in \mathbb{N}$ where $0\leq k \leq n+n^2$ and $r\geq 0$, one has ${\int_{|s|\geq r}|\hat{f}^{(k)}(s)|ds}\leq M\exp(-wr)$. Then there exist $C\geq 0$ and $N\in \mathbb{N}^{+}$ such that for any $0<t\leq 1$ and $x,y\in M$, we have:
    $$||K_{f_{t}(D)}(x,y)||\leq \frac{C}{t^{N}}\exp\left(-w\frac{d(x,y)}{100t}\right).$$
\end{theorem}
\begin{proof}
    Because of Theorem \ref{g1}, there exist $C_{1}, C_{2}\geq 0$ and $N_{1}, N_{2}\in \mathbb{N}$ such that:
    \begin{align*}
     ||K_{f_{t}(D)}(x,y)||&\leq \frac{C_{1}}{d(x,y)^{N_{1}}}\sum_{k=0}^{n+n^2}{\int_{|s|\geq \frac{d(x,y)}{100}}|\hat{f_{t}}^{(k)}(s)|ds}\\
     &\leq\frac{C_{1}}{d(x,y)^{N_{1}}}\sum_{k=0}^{n+n^2}\left(t^{-(k+1)}\int_{|s|\geq \frac{d(x,y)}{100t}}|\hat{f^{(k)}}(s)|ds\right)\\
        &\leq\frac{C_{2}}{d(x,y)^{N_{1}}t^{N_{2}}}\exp\left(-w\frac{d(x,y)}{100t}\right)
       \end{align*}
    for any $0<t\leq 1$ and $x,y\in M$.

    Similarly by Theorem \ref{g2},  there exist $C_{3}, C_{4}\geq 0$ and $N_{3}\in \mathbb{N}$ such that:
    \begin{align*}
        ||K_{f_{t}(D)}(x,y)||&\leq C_{3}\sum_{k=0}^{n+n^2}\int_{\mathbb{R}}|\hat{f_{t}}^{(k)}(s)|ds\\
        &\leq C_{3}\sum_{k=0}^{n+n^2}\left(t^{-(k+1)}\int_{\mathbb{R}}|\hat{f}^{(k)}(s)|ds\right)\\
        &\leq C_{4}t^{-N_{3}}.
         \end{align*}
        When $d(x,y)\geq t$, we have:
         \begin{align*}
             ||K_{f_{t}(D)}(x,y)|&\leq\frac{C_{2}}{d(x,y)^{N_{1}}t^{N_{2}}}\exp\left(-w\frac{d(x,y)}{100t}\right)\\
             &\leq \frac{C_{2}}{t^{N_{2}+N_{1}}}\exp\left(-w\frac{d(x,y)}{100t}\right).
         \end{align*}
         When $d(x,y)\leq t$, we have:
         \begin{align*}
             ||K_{f_{t}(D)}(x,y)||&\leq C_{4}t^{-N_{3}}\\
             &\leq C_{4}\exp(\frac{w}{100})t^{-N_{3}}\exp\left(-w\frac{d(x,y)}{100t}\right).
         \end{align*}
    Then we can choose  suitable constant $C$ and $N$, such that:\\
    $$\frac{C_{2}}{t^{N_{2}+N_{1}}}\exp\left(-w\frac{d(x,y)}{100t}\right)\leq  \frac{C}{t^{N}}\exp\left(-w\frac{d(x,y)}{100t}\right),$$ and\\
    $$ C_{4}\exp(\frac{w}{100})t^{-N_{3}}\exp\left(-w\frac{d(x,y)}{100t}\right)\leq \frac{C}{t^{N}}\exp\left(-w\frac{d(x,y)}{100t}\right).$$ 
    Then the proof is complete.
\end{proof}
Let $D$ be the equivariant Dirac operator on a section of $E$ over $M$. The higher index $\Ind_{t}(D)$ is represented by:
\begin{align*}
    \Ind_{t}(D)&=W(t)-\begin{pmatrix}
0&0\\
0&\Id
\end{pmatrix}\\
&=\begin{pmatrix}
e^{-t^{2}D^{-}D^{+}}&e^{-\frac{t^{2}}{2}D^{-}D^{+}}(\frac{1-e^{-t^{2}D^{-}D^{+}}}{D^{-}D^{+}})^{\frac{1}{2}}D^{-}\\
e^{-\frac{t^{2}}{2}D^{+}D^{-}}(\frac{1-e^{-t^{2}D^{+}D^{-}}}{D^{+}D^{-}})^{\frac{1}{2}}D^{+}&-e^{-t^{2}D^{+}D^{-}}
\end{pmatrix}
\end{align*}
and we split it into $\gamma f_{t}(D)+g_{t}(D)$ where $f_{t}(x)=e^{-t^{2}x^{2}}$, $g_{t}(x)=e^{-\frac{t^{2}}{2}x^{2}}(\frac{1-e^{-t^{2}x^{2}}}{x^2})^{\frac{1}{2}}x$ and $\gamma$ is the $\mathbb{Z}_{2}$-grading operator.

We claim that: 
\begin{proposition}
    For any $t>0$, $f_{t}(x)=e^{-t^{2}x^{2}}$, $g_{t}(x)=e^{-\frac{t^{2}}{2}x^{2}}(\frac{1-e^{-t^{2}x^{2}}}{x^2})^{\frac{1}{2}}x$ satisfy the conditions of Theorem \ref{g3}.
\end{proposition}
We give a proof of this proposition in Appendix A (see Lemma \ref{a1} and \ref{a2}). Because of this, we have:
\begin{theorem}
\label{g3}
    Let $f_{t}(x)=e^{-t^{2}x^{2}}$, $g_{t}(x)=e^{-\frac{t^{2}}{2}x^{2}}(\frac{1-e^{-t^{2}X^{2}}}{x^2})^{\frac{1}{2}}x$ and $\gamma$ be the $\mathbb{Z}_{2}$-grading operator. Then there exist $C\geq 0$ and $N\in \mathbb{N}$ such that for any $0<t\leq 1$ and $x,y\in M$, we have:
    $$||K_{\gamma f_{t}(D)+g_{t}(D)}(x,y)||\leq \frac{C}{t^{N}}\exp(-w\frac{d(x,y)}{100t}).$$
\end{theorem}
\begin{remark}
    For any $u>0$, if we include derivatives in the estimates, through a similar discussion as above, we can find a number $0<t_{u}<1$, such that for any $0<t\leq t_{u}$, there is:
$$\sup_{x,y\in M}||\exp(u\cdot d(x,y))\nabla_{x}^{m}\nabla_{y}^{n}K_{\gamma f_{t}(D)+g_{t}(D)}(x,y)||<C_{m,n} \quad \forall q,m,n\in \mathbb{N}.$$
This implies that
$$\Ind_{t}(D)\in S^{u}_{G}(M,E).$$
\end{remark}
Therefore, for any $f\in C^{2q}_{v}(M)^{G}$ such that $2qv-u<-r_{M}$, we have:
{\small $$\tau(f)(K_{w'(t)})=\int_{M^{2q+1}}c_{0}(x_{0})f(x_{0},...,x_{2q})\Tr(K_{w'(t)}(x_{0},x_{1})...K_{w'(t)}(x_{2q},x_{0}))d\mu(x_{0})...d\mu(x_{2q}),
$$}
where $w'(t)=\Ind_{t}(D)=\gamma f_{t}(D)+g_{t}(D)$.

Next, we shall use the heat equation approach to find the topological expression of the localized analytic indices of elliptic operators as in \cite{ref10}, that is, find the topological formula of $\lim_{t\rightarrow 0}\tau(f)(\Ind_{t}(D))$. But in our case, $M$ is non-compact, so we cannot use the method in \cite{ref10} to get the local index formula directly and we need the following theorem:
\begin{theorem}
\label{rc}
    Let $c_{i}(x)\in C_{c}^{\infty}(M)$, such that $0\leq c_{i}(x)\leq 1$ and $c_{i}|_{B_{1}(\supp(c_{i-1}))}=1$, where $B_{1}(\supp(c_{i-1}))=\{x\in M| d(x,\supp(c_{i-1}))\leq 1\}$. Then we have:
    \begin{align*}
         &\tau(f)(\Ind_{t}(D))\\
         &=\int_{M^{2q+1}}c_{0}(x_{0})c_{1}(x_{1})...c_{2q}(x_{2q})f(x_{0},...,x_{2q})\Tr(K_{w'(t)}(x_{0},x_{1})...K_{w'(t)}(x_{2q},x_{0}))dx_{0}...dx_{2q}\\
         &+R(t)
    \end{align*}
    where $\lim_{t\rightarrow 0}R(t)=0$.
\end{theorem}
    Because of this theorem, we have 
   \begin{align*}
       &\lim_{t\rightarrow 0}\tau(f)(\Ind_{t}(D))\\
       &=\lim_{t\rightarrow 0}\int_{M^{2q+1}}c_{0}(x_{0})...c_{2q}(x_{2q})f(x_{0},...,x_{2q})\Tr(K_{w'(t)}(x_{0},x_{1})...K_{w'(t)}(x_{2q},x_{0}))dx_{0}...dx_{2q}.
   \end{align*}
   
    It reduces the problem to the compact case, then we can use the method in \cite{ref10} to obtain the local index formula.
    \begin{proof}[Proof of Theorem \ref{rc}]
    Let $I=\{i_{1},i_{2},...,i_{|I|}\}\subset\{1,2,...,2q\}$  where $|I|$ is the length of $I$. Let $J$ be the complement of $I$ in $\{1,2,...,2q\}$. Similarly write $J=\{j_{1},j_{2},...,j_{|J|}\}$ where $|J|$ is the length of $J$.  We denote:$$\prod_{I}(x_{0},..,x_{2q})=c_{0}(x_{0})c_{j_{1}}(x_{j_{1}})...c_{j_{|J|}}(x_{j_{|J|}})(1-c_{i_{1}}(x_{i_{1}})...(1-c_{i_{|I|}}(x_{i_{|I|}})).$$
    Therefore, we can rewrite $R(t)$ as follows:
    $$R(t)=\sum_{|I|=1}^{2q}\int_{M^{2q+1}}\prod_{I}(x_{0},..,x_{2q})f(x_{0},...,x_{2q})\Tr(K_{w'(t)}(x_{0},x_{1})...K_{w'(t)}(x_{2q},x_{0}))dx_{0}...dx_{2q}.$$
    Next, we show that: for each $I$,
    $$\lim_{t\rightarrow 0}\int_{M^{2q+1}}\prod_{I}(x_{0},..,x_{2q})f(x_{0},...,x_{2q})\Tr(K_{w'(t)}(x_{0},x_{1})...K_{w'(t)}(x_{2q},x_{0}))dx_{0}...dx_{2q}=0.$$
    Because $f$ has at most exponential growth,
    then there exist $C_{1},C_{2}\geq 0$ and $z_{0}\in M$ such that $$|f(x_{0},...,x_{2q})|\leq C_{1}e^{C_{2}(d(x_{0},z_{0})+d(x_{0},x_{1})+d(x_{1},x_{2})+...+d(x_{2q},x_{0}))}$$ for any $x_{i}\in M$.

    By Theorem \ref{g3}, there exist $C_{3},N, w>0$ such that:
    \begin{align*}
        &|\int_{M^{2q+1}}\prod_{I}(x_{0},..,x_{2q})f(x_{0},...,x_{2q})\Tr(K_{w'(t)}(x_{0},x_{1})...K_{w'(t)}(x_{2q},x_{0}))d\mu(x_{0})...d\mu(x_{2q})|\\
        &\leq \int_{M^{2q+1}}\prod_{I}(x_{0},..,x_{2q}) C_{1}e^{C_{2}(d(x_{0},z_{0})+d(x_{0},x_{1})+d(x_{1},x_{2})+...+d(x_{2q},x_{0}))}\\
        &\times \frac{C_{3}}{t^{N}}e^{-\frac{w}{t}(d(x_{0},x_{1})+d(x_{1},x_{2})+...+d(x_{2q},x_{0}))}d\mu(x_{0})...d\mu(x_{2q}).
    \end{align*}

    Using integration by part, we only need to estimate the following two cases of integrations. 
    
    $(1)$ $\int_{M}e^{C_{2}(d(x_{i},x_{i+1}))}e^{-\frac{w}{t}(d(x_{i},x_{i+1})}dx_{i+1}.$

    Since  $M$ is a $C^{\infty}$- manifold with an action of almost connected Lie group $G$ such that the action is proper and co-compact, there exist $M_{1}, C_{4}>0$, such that for any $x\in M$ and $r\leq 0$, we have $\vol (B_{r}(x))\leq M_{1}e^{C_{4}r}$  where $\vol(B_{r}(x))$ is the volume of $B_{r}(x)$. 

    We denote $U_{x}(n\leq r\leq m)=\{y\in M| n\leq d(x,y)\leq m\}$. Therefore, we can decompose $\int_{M}e^{C_{2}(d(x_{i},x_{i+1}))}e^{-\frac{w}{t}(d(x_{i},x_{i+1})}dx_{i+1}$ as $$\int_{B_{1}(x_{i})}e^{C_{2}(d(x_{i},x_{i+1}))}e^{-\frac{w}{t}(d(x_{i},x_{i+1})}dx_{i+1}+\sum_{n=1}^{+\infty}\int_{U_{x_{i}}(n\leq r\leq n+1)}e^{C_{2}(d(x_{i},x_{i+1}))}e^{-\frac{w}{t}(d(x_{i},x_{i+1})}dx_{i+1}.$$

    Clearly, there exists $M_{2}>0$ independent of $x_{i}$, such that $$\int_{B_{1}(x_{i})}e^{C_{2}(d(x_{i},x_{i+1}))}e^{-\frac{w}{t}(d(x_{i},x_{i+1})}dx_{i+1}\leq M_{2}.$$

    For $\int_{U_{x_{i}}(n\leq r\leq n+1)}e^{C_{2}(d(x_{i},x_{i+1}))}e^{-\frac{w}{t}(d(x_{i},x_{i+1})}dx_{i+1}$, we have:
    \begin{align*}
        &\int_{U_{x_{i}}(n\leq r\leq n+1)}e^{C_{2}(d(x_{i},x_{i+1}))}e^{-\frac{w}{t}(d(x_{i},x_{i+1})}dx_{i+1}\\
        &\leq e^{C_{2}(n+1)}e^{-\frac{w}{t}(n)}\vol(U_{x_{i}}(n\leq r\leq n+1))\\
        &\leq e^{C_{2}(n+1)}e^{-\frac{w}{t}(n)}\vol(B_{n+1}({x_{i}}))\\
        &\leq e^{C_{2}(n+1)}e^{-\frac{w}{t}(n)}M_{1}e^{C_{4}(n+1)}\\
        &\leq M_{1}e^{-(n+1)(\frac{w}{t}-(C_{2}+C_{4}))}.
    \end{align*}
    Choose $t$ small enough, we have $M_{1}e^{-(n+1)(\frac{w}{t}-(C_{2}+C_{4}))}\leq M_{1}e^{-(n+1)(\frac{w}{2t})}$.
    Thus, there exists $M_{3}$ such that:
    \begin{align*}
        &\int_{M}e^{C_{2}(d(x_{i},x_{i+1}))}e^{-\frac{w}{t}(d(x_{i},x_{i+1})}dx_{i+1}\\
        &=\int_{B_{1}(x_{i})}e^{C_{2}(d(x_{i},x_{i+1}))}e^{-\frac{w}{t}(d(x_{i},x_{i+1})}dx_{i+1}+\sum_{n=1}^{+\infty}\int_{U_{x}(n\leq r\leq n+1)}e^{C_{2}(d(x_{i},x_{i+1}))}e^{-\frac{w}{t}(d(x_{i},x_{i+1}))}dx_{i+1}\\
        &\leq M_{2}+M_{3}e^{-\frac{w}{t}}.
    \end{align*}
    $(2)$ $\int_{M}c_{i}(x_{i})e^{C_{2}(d(x_{i},x_{i+1}))}e^{-\frac{w}{t}(d(x_{i},x_{i+1})}(1-c_{i+1}(x_{i+1}))dx_{i+1}$

    Similarly, we have:
    \begin{align*}
        &\int_{M}c_{i}(x_{i})e^{C_{2}(d(x_{i},x_{i+1}))}e^{-\frac{w}{t}(d(x_{i},x_{i+1})}(1-c_{i+1}(x_{i+1}))dx_{i+1}\\
        &=\int_{B_{1}(x_{i})}c_{i}(x_{i})e^{C_{2}(d(x_{i},x_{i+1}))}e^{-\frac{w}{t}(d(x_{i},x_{i+1})}(1-c_{i+1}(x_{i+1}))dx_{i+1}\\
        &+\sum_{n=1}^{+\infty}\int_{U_{x_{i}}(n\leq r\leq n+1)}c_{i}(x_{i})e^{C_{2}(d(x_{i},x_{i+1}))}e^{-\frac{w}{t}(d(x_{i},x_{i+1}))}(1-c_{i+1}(x_{i+1}))dx_{i+1}\\
        &\leq 0+M_{3}e^{-\frac{w}{t}}=M_{3}e^{-\frac{w}{t}}.
    \end{align*}

    Since each $\prod_{I}(x_{0},..,x_{2q})$ contains $c_{i}(x_{i})(1-c_{i+1}(x_{i+1}))$ for some $i, 0\leq i\leq 2q$, there exists $M_{4}>0$ such that:
    \begin{align*}
        &\int_{M^{2q+1}}\prod_{I}(x_{0},..,x_{2q}) C_{1}e^{C_{2}(d(x_{0},z_{0})+d(x_{0},x_{1})+d(x_{1},x_{2})+...+d(x_{2q},x_{0}))}\\
        &\times \frac{C_{3}}{t^{N}}e^{-\frac{w}{t}(d(x_{0},x_{1})+d(x_{1},x_{2})+...+d(x_{2q},x_{0}))}dx_{0}...dx_{2q}\\
        &\leq \frac{M_{4}}{t^{N}}e^{-\frac{w}{t}}.
    \end{align*}

    Thus, we have:
    \begin{align*}
        &|\int_{M^{2q+1}}\prod_{I}(x_{0},..,x_{2q})f(x_{0},...,x_{2q})\Tr(K_{w'(t)}(x_{0},x_{1})...K_{w'(t)}(x_{2q},x_{0}))dx_{0}...dx_{2q}|\\
        &\leq \int_{M^{2q+1}}\prod_{I}(x_{0},..,x_{2q}) C_{1}e^{C_{2}(d(x_{0},z_{0})+d(x_{0},x_{1})+d(x_{1},x_{2})+...+d(x_{2q},x_{0}))}\\
        &\times \frac{C_{3}}{t^{N}}e^{-\frac{w}{t}(d(x_{0},x_{1})+d(x_{1},x_{2})+...+d(x_{2q},x_{0}))}dx_{0}...dx_{2q}\\
        &\leq \frac{M_{4}}{t^{N}}e^{-\frac{w}{t}}.
\end{align*}
    
    This implies that for each $I$,
    $$\lim_{t\rightarrow 0}\int_{M^{2q+1}}\prod_{I}(x_{0},..,x_{2q})f(x_{0},...,x_{2q})\Tr(K_{w'(t)}(x_{0},x_{1})...K_{w'(t)}(x_{2q},x_{0}))dx_{0}...dx_{2q}=0$$
    which completes the proof.
    
\end{proof}

Because of this theorem, we have 
   \begin{align*}
       &\lim_{t\rightarrow 0}\tau(f)(\Ind_{t}(D))\\
       &=\lim_{t\rightarrow 0}\int_{M^{2q+1}}c_{0}(x_{0})...c_{2q}(x_{2q})f(x_{0},...,x_{2q})\Tr(K_{w'(t)}(x_{0},x_{1})...K_{w'(t)}(x_{2q},x_{0}))dx_{0}...dx_{2q}.
   \end{align*}


To simplify notation, replace $c_{i}(x)$ by $c_{i}(x)^{2}$ for each $0\leq i\leq 2q$. Then the last equality is equal to:
\small{\begin{align*}
    &\int_{M^{2q+1}}c_{0}^{2}(x_{0})c_{1}^{2}(x_{1})...c_{2q}^{2}(x_{2q})f(x_{0},...,x_{2q})\Tr(K_{w'(t)}(x_{0},x_{1})...K_{w'(t)}(x_{2q},x_{0}))dx_{0}...dx_{2q}\\
    &=\int_{M^{2q+1}}f(x_{0},...,x_{2q})\Tr(c_{0}(x_{0})K_{w'(t)}(x_{0},x_{1})c_{1}(x_{1})...c_{2q}(x_{2q})K_{w'(t)}(x_{2q},x_{0})c_{0}(x_{0}))dx_{0}...dx_{2q}.
\end{align*} }

We shall conveniently use the same notation as in \cite{ref10}. Thus, we denote $\mathbb{D}=D\circ\gamma$ where $D$ is the equivariant Dirac operator over $E$ and $\gamma$ is the $\mathbb{Z}_{2}$-grading operator. Thus, $\mathbb{D}$ becomes skew adjonit such that $\mathbb{D}^{*}=-\mathbb{D}$ and $\mathbb{D}^{2}=-D^{2}$. Because of this, we can rewrite $W'(t)$ as:
\begin{align*}
    W'(t)&=\begin{pmatrix}
e^{-t^{2}D^{-}D^{+}}&e^{-\frac{t^{2}}{2}D^{-}D^{+}}(\frac{1-e^{-t^{2}D^{-}D^{+}}}{D^{-}D^{+}})^{\frac{1}{2}}D^{-}\\
e^{-\frac{t^{2}}{2}D^{+}D^{-}}(\frac{1-e^{-t^{2}D^{+}D^{-}}}{D^{+}D^{-}})^{\frac{1}{2}}D^{+}&-e^{-t^{2}D^{+}D^{-}}
\end{pmatrix}\\
&=(e^{t^{2}\mathbb{D}^{2}}+e^{\frac{1}{2}t^{2}\mathbb{D}^{2}}w(-t^{2}\mathbb{D}^{2})t\mathbb{D})\gamma.
\end{align*}
We also denote: 
\small{\begin{align*}
    &\tau(f)(c_{0}W'(t)c_{1},c_{1}W'(t)c_{2},...,c_{2q}W'(t)c_{0})\\
    &:=\int_{M^{2q+1}}f(x_{0},...,x_{2q})\Tr(c_{0}(x_{0})K_{w'(t)}(x_{0},x_{1})c_{1}(x_{1})...c_{2q}(x_{2q})K_{w'(t)}(x_{2q},x_{0})c_{0}(x_{0}))dx_{0}...dx_{2q}.
\end{align*} }

Then, we choose $f\in C^{2q}_{v,anti}(M)^{G}$. Since Theorem \ref{rc} allows us to modify $M$ outside a sufficiently large relatively compact domain, without loss of generality, let $f=f_{0}\otimes f_{1}\otimes...\otimes f_{2q}$ where $f_{i}\in C^{\infty}_{c}(M)$.  Then we have:
\begin{align*}
    &\tau(f)(c_{0}W'(t)c_{1},c_{1}W'(t)c_{2},...,c_{2q}W'(t)c_{0})\\
    &=\Tr(c_{2q}W'(t)c_{0}f_{0}[c_{0}W'(t)c_{1}, f_{1}]...[c_{2q-1}W'(t)c_{0}, f_{2q}]).
\end{align*}

Therefore, the quantity to be computed is:\small{
$$\lim_{t\rightarrow 0}\tau(f)(\Ind_{t}(D))=\lim_{t\rightarrow 0}\Str(c_{2q}(e^{t^{2}\mathbb{D}^{2}}-e^{\frac{1}{2}t^{2}\mathbb{D}^{2}}w(-t^{2}\mathbb{D}^{2})t\mathbb{D})c_{0}f_{0}[c_{0}W'(t)c_{1}, f_{1}]...[c_{2q-1}W'(t)c_, f_{2q}])
$$}
where $\Str=\Tr\circ\gamma$. Now, we use Getzler’s symbolic calculus to find this limit as in \cite{ref10}. 
\begin{remark}
    If we choose a representative of the higher index $\Ind_t(D)$ that has finite propagation, i.e., the Schwartz kernel is properly supported, then Theorem \ref{rc} is not needed. 
    Otherwise, the theorem is necessary. 
    Even if one uses the expansion 
\[    f=\sum_{\mu=0}^{\infty}f_{\mu}^{1}\otimes f_{\mu}^{2}\otimes...\otimes f_{\mu}^{n} \quad \text{where} \quad f_{\mu}^{i}\in C^{\infty}_{c}(M)
\] 
as in the compact case \cite{ref10}, in our noncompact setting, one needs to show the absolute convergence of \[\langle f,\Ind_{t}(D)\rangle=\sum_{\mu=0}^{\infty}\langle f_{\mu},\Ind_{t}(D)\rangle
\] 
with respect to $t$ and that will require Theorem \ref{rc}.
\end{remark}

We recall some basic definitions originally introduced in \cite{ref10}.
\begin{definition}[\cite{ref10}]
    Let $\mathscr{A}(\mathbb{D})$ be the algebra generated (in a strictly algebraic sense) by $\mathscr{A}=C^{\infty}(M)$, $\gamma$, $(\lambda+\mathbb{D}^{2})^{-1}$, $\lambda\in \mathbb{C}\backslash[0,+\infty)$ and operators of the form $u(-\mathbb{D}^{2})$ with $u$ a Schwartz function on $[0,+\infty)$ which admits a holomorphic extension to a complex neighborhood of $[0,+\infty)$.
\end{definition}
The notion of asymptotic orders for elements of  $\mathscr{A}(\mathbb{D})$ is recalled as follows.
\begin{definition}[\cite{ref10}]
 If $A$ belongs to the subalgebra $\mathscr{A}_{\diff}(\mathbb{D})$ generated by $\mathscr{A}$, $\gamma$ and $\mathbb{D}$, we let $A(t)$ be the operator obtained by replacing $\mathbb{D}$ with $t\mathbb{D}$, $t > 0$, in the expression of $A$. Define the asymptotic order of $A$ as the total Getzler order of $A(t)$, assigning to $t$ the order $-1$. In particular, $\mathbb{D}^{2}$ has asymptotic order 0. We have the same asymptotic order $0$ for any function of $-\mathbb{D}^{2}$. Finally, we extend the notion of asymptotic order to the whole $\mathscr{A}(\mathbb{D})$ in the usual fashion. Let $\mathscr{A}_{k}^{r}(\mathbb{D})$ means the subspace of $\mathscr{A}(\mathbb{D})$ consisting of operators of Getzler order $r$ and asymptotic order $k$. 

\end{definition}
 The following result is the “Fundamental Lemma” of Getzler’s symbolic calculus. For readers' convenience, we list it here.
 \begin{lemma}[\cite{ref10}]
     $(1)$ Let $A\in \mathscr{A}_{0}^{r}(\mathbb{D})$. Then, if $A\in \op(\mathscr{S}^{r}(E))$,
     $$\sigma_{t^{-1}}(A(t))=a+O(t), \quad O(t)\in \mathscr{S}^{r-1}(E)$$
     where $\sigma_{t^{-1}}$ is the rescaled symbol \cite{ref12} . Here $a$ is called the asymptotic symbol of $A\in \mathscr{A}_{0}(\mathbb{D})$, denoted by $\sigma_{0}(A)$.

     $(2)$ If $A,B\in \mathscr{A}_{0}^{r}(\mathbb{D})$, then one has $\sigma_{0}(AB)=\sigma_{0}(A)\ast\sigma_{0}(B)$. Here $\ast$ denotes the multiplication  of Getzler symbols, defined by the rule:
    $$(a\ast b)(x,\xi)=e^{-\frac{1}{4}R(\frac{\partial}{\partial\xi},\frac{\partial}{\partial\eta})}a(x,\xi)\wedge b(x,\eta)|_{\xi=\eta}$$
    where $R\in \wedge^{2}(M)\otimes \wedge^{2}(M)$ is the curvature tensor in \cite{ref12}.
     \end{lemma}
    The following results from \cite{ref10} will be needed in the calculations. For the convenience of readers, we list them here.
    \begin{align*}
        &(1) \sigma_{t^{-1}}(t\mathbb{D})=it^{-1}\xi;\\
        &(2) \sigma_{t^{-1}}(t^{2}\mathbb{D}^{2})=-|\xi|^{2}+Q+\frac{1}{4}t^{2}\tau(R);\\
        &(3)\sigma_{t^{-1}}([t\mathbb{D},f])(x,\xi)=df,\quad f\in \mathscr{A};\\
        &(4) \sigma_{t^{-1}}([t^{2}\mathbb{D}^{2},f])(x,\xi)=2it\langle df,\xi\rangle+t^{2}\Delta f;\\
        &(5)\sigma_{t^{-1}}([t^{2}\mathbb{D}^{2},f]t\mathbb{D})(x,\xi)=-2t\langle df,\xi\rangle\xi+O(t),\quad O(t)\in \mathscr{S}^{2}(E);\\
        &(6)\sigma_{t^{-1}}((\lambda+t^{2}\mathbb{D}^{2})^{-1})(x,\xi)=((\lambda-H(x,\xi))^{-1}+O(t),\quad O(t)\in \mathscr{S}^{-3}(E);\\
        &(7)\sigma_{t^{-1}}(u(-t^{2}\mathbb{D}^{2}))(x,\xi)=u(H)(x,\xi)+O(t),\quad O(t)\in \mathscr{S}^{-\infty}(E),\quad u(-t^{2}\mathbb{D}^{2})\in  \mathscr{A}(\mathbb{D});
    \end{align*}
    where $Q=\nabla_{W}^{2}$ is the equivariant curvature of $W$ ($E=W\otimes S$) induced by the equivariant curvature of $E$,$\tau(R)$ is the scalar curvature of $M$ and $H=H(x,\xi)=|\xi|^{2}-\frac{1}{2}R(\xi,\frac{\partial}{\partial\xi})-\frac{1}{16}R\wedge R(\frac{\partial}{\partial\xi},\frac{\partial}{\partial\xi})-Q$.

 Denote $A=c_{2q}u(-\mathbb{D}^{2})\mathbb{D}c_{0}$, $A^{j}=c_{j-1}u(-\mathbb{D}^{2})[\mathbb{D},f_{j}]c_{j}$ and $B^{j}=c_{j-1}[u(-\mathbb{D}^{2}),f_{j}]\mathbb{D}c_{j}$ for $j=1,2,...,2q$. 
 We can write:
 \begin{align*}
     \Pi:=&(c_{2q}e^{\mathbb{D}^{2}}c_{0}-A)(c_{0}[e^{\mathbb{D}^{2}},f_{1}]c_{1}+A^{1}+B^{1})(c_{1}[e^{\mathbb{D}^{2}},f_{2}]c_{2}-A^{2}-B^{2})...\\
     &...(c_{2q-2}[e^{\mathbb{D}^{2}},f_{2q-1}]c_{2q-1}+A^{2q-1}+B^{2q-1})(c_{2q-1}[e^{\mathbb{D}^{2}},f_{2}]c_{2q}-A^{2q}-B^{2q}).
 \end{align*}
The following terms of $\Pi$ will contribute to the asymptotic symbol:
\begin{align*}
    &(1)T:=(-1)^{q}c_{2q}e^{\mathbb{D}^{2}}c_{0}f_{0}A^{1}...A^{2q};\\
    &(2)T_{j}:=(-1)^{q}c_{2q}e^{\mathbb{D}^{2}}c_{0}f_{0}A^{1}...A^{j-1}B^{j}A^{j+1}...A^{2q}, j=1,...,2q;\\
    &(3)Z_{j}:=(-1)^{q+j}Af_{0}A_{1}...A_{j-1}c_{j-1}[e^{\mathbb{D}^{2}},f_{j}]c_{j}A^{j+1}..A^{2q},j=1,...,2q.
\end{align*}

Since we do the symbol calculation on the non-compact manifold, the Getzler’s fundamental trace formula we used is:
\begin{proposition}
\label{gz}
    Let $P$ be a smoothing operator with compact support over $E$, Then for all $t>0$,
    $$\Str(P)=\frac{1}{({2\pi})^{n}}\int_{T^{*}M}\Str(\sigma_{t}(P)(x,\xi))d\xi dx$$
\end{proposition}
\begin{remark}
    Since we consider the smoothing operator with compact support, the proof of Proposition \ref{gz} is similar to the classical one in \cite{ref12}. For reader's convenience, we give the proof in Appendix $C$.
\end{remark}
We are now ready to begin the essential calculations. We will show that
\begin{theorem}
\label{mthm}
    \begin{align*}
    &\lim_{t\rightarrow 0}\Str(c_{2q}(e^{t^{2}\mathbb{D}^{2}}-e^{\frac{1}{2}t^{2}\mathbb{D}^{2}}w(-t^{2}\mathbb{D}^{2})t\mathbb{D})c_{0}f_{0}[c_{0}W'(t)c_{1}, f_{1}]...[c_{2q-1}W'(t)c_, f_{2q}])\\
    &=\frac{(-1)^{n/2-q}}{(2\pi i)^{2q-n/2}}\frac{q!}{(2q)!}\int_{M}c_{0}\hat{A}_{G}(M)\wedge \ch_{G}^{0}(\pi!([\pi^{*}(E^{+}),\pi^{*}(E^{-1}),\sigma(D)]))\wedge f_{0}\wedge df_{1}...\wedge df_{2q}.
\end{align*}
\end{theorem}
\begin{proof}
   $(1)$ Computation of the contribution of $T$.\\
   We have:
   $$\sigma_{0}(T)=(-1)^{q}c_{2q}\ast\sigma_{0}(u(-\mathbb{D}^{2}))\ast c_{0}\ast f_{0}\ast c_{0}\ast \sigma_{0}(u(-\mathbb{D}^{2}))\ast df_{1}\ast c_{1}\ast...\ast c_{2q-1}\ast\sigma_{0}(u(-\mathbb{D}^{2}))\ast df_{2q}\ast c_{2q}.
   $$

   Since $\sigma_{0}([u(-\mathbb{D}^{2}),f])=0$ and $\sigma_{0}([u(-\mathbb{D}^{2}),[\mathbb{D},f]])=0$ for any $f\in \mathscr{A}$, one can permute $u(-\mathbb{D}^{2})$ with $c_{i}$ and $[D,f_{i}]$ without changing the asymptotic symbol of $T$.

   Thus, we have:
  \small{ \begin{align*}
       \sigma_{0}(T)&=(-1)^{q}c_{2q}\ast\sigma_{0}(e^{\mathbb{D}^{2}})\ast c_{0}\ast f_{0}\ast c_{0}\ast \sigma_{0}(u(-\mathbb{D}^{2}))\ast df_{1}\ast c_{1}\ast...\ast c_{2q-1}\ast\sigma_{0}(u(-\mathbb{D}^{2}))\ast df_{2q}\ast c_{2q}\\
       &=(-1)^{q}c_{2q}\sigma_{0}(e^{\mathbb{D}^{2}}u(-\mathbb{D}^{2})^{2q})\wedge c_{0}f_{0}\wedge c_{0}(df_{1})c_{1}\wedge...\wedge c_{2q-1}(df_{2q})c_{2q}.
   \end{align*} }
   Since $c_{i}(x)=1$ for any $x\in \supp(c_{0})$ where $i\geq1$,
   we have:
   $$\sigma_{0}(T)=(-1)^{q}c_{0}\sigma_{0}(e^{\mathbb{D}^{2}}u(-\mathbb{D}^{2})^{2q})\wedge f_{0}\wedge df_{1}\wedge...\wedge df_{2q}.
   $$
   Then by the  similar proof in \cite{ref10} and Proposition \ref{gz}, we have:
   \begin{align*}
       \lim_{t\rightarrow 0}\Str(T(t))&=(2\pi)^{-n}\int_{T^{*}M}\tr_{s}(\sigma_{0}(T))dxd\xi\\
       &=\beta_{q}(-1)^{q}(2\pi i)^{n/2}\int_{M}c_{0}\det\left(\frac{R/2}{\sinh{R/2}}\right)^{1/2}\wedge \Str(e^{Q})\wedge f_{0}\wedge df_{1}...\wedge df_{2q},
   \end{align*}
   where
   $$\beta_{q}=\int_{1}^{2}...\int_{1}^{2}(1+s_{1}+...+s_{q})^{-q}ds_{1}...ds_{q}.
   $$
  $(2)$  Computation of the contribution of  $\sum_{j=1}^{j=2q}T_{j}$. 
  
  As in \cite{ref10}, we have:
  $$T_{j}=(-1)^{q}c_{2q}e^{\mathbb{D}^{2}}c_{0}f_{0}A^{1}...A^{j-1}c_{j-1}\left(\frac{1}{2\pi i}\int_{C}u(\lambda)[(\lambda+\mathbb{D}^{2})^{-1},f_{j}]d\lambda\right)\mathbb{D}c_{j}A^{j+1}...A^{2q}$$
  where $C$ is a suitable contour in $\mathbb{C}$, oriented counterclockwise. Since:
  $$[(\lambda+\mathbb{D}^{2})^{-1},f_{j}]=-(\lambda+\mathbb{D}^{2})^{-1}[\mathbb{D}^{2},f_{j}](\lambda+\mathbb{D}^{2})^{-1},
  $$
  we have:
  \begin{align*}
      T_{j}&=(-1)^{q+1}c_{2q}e^{\mathbb{D}^{2}}c_{0}f_{0}A^{1}...A^{j-1}c_{j-1}\left(\frac{1}{2\pi i}\int_{C}u(\lambda)(\lambda+\mathbb{D}^{2})^{-1}[\mathbb{D}^{2},f_{j}]\mathbb{D}(\lambda+\mathbb{D}^{2})^{-1}c_{j}d\lambda\right)\\
      &\times A^{j+1}...A^{2q}.
  \end{align*}
  Since $\sigma_{0}([(\lambda+\mathbb{D}^{2})^{-1},f])=0$ for any $f\in \mathscr{A}$, we can exchange $(\lambda+\mathbb{D}^{2})^{-1}$ and $c_{j}$ without changing the asymptotic symbol of $T_{j}$.

  Because of this, let
  \begin{align*}
      T'_{j}&=(-1)^{q+1}c_{2q}e^{\mathbb{D}^{2}}c_{0}f_{0}A^{1}...A^{j-1}c_{j-1}\frac{1}{2\pi i}\int_{C}u(\lambda)(\lambda+\mathbb{D}^{2})^{-1}[\mathbb{D}^{2},f_{j}]\\
&\times\mathbb{D}c_{j}(\lambda+\mathbb{D}^{2})^{-1}d\lambda A^{j+1}...A^{2q}.
\end{align*}
We have: 
$$\lim_{t\rightarrow 0}\Str (T_{j}(t))=(2\pi )^{-n}\int_{T^{*}M}\tr_{s}(\sigma_{0}(T_{j}))dxd\xi=(2\pi )^{-n}\int_{T^{*}M}\tr_{s}(\sigma_{0}(T'_{j}))dxd\xi=\lim_{t\rightarrow 0}\Str (T'_{j}(t)).
$$
Then we take advantage of the fact that only the supertrace is needed as in \cite{ref10} and replace $T'_{j}$ by
\begin{align*}
    \tilde{T_{j}}&=\frac{(-1)^{q+1}}{2\pi i}\int_{C}(-1)^{2q+j}c_{j}(\lambda+\mathbb{D}^{2})^{-1}A^{j+1}...A^{2q}c_{2q}e^{\mathbb{D}^{2}}c_{0}f_{0}A^{1}\\
&...u(\lambda)(\lambda+\mathbb{D}^{2})^{-1}[\mathbb{D}^{2},f_{j}]\mathbb{D}d\lambda 
\end{align*}
thus bringing the two resolvents on the same side of $[\mathbb{D}^{2},f_{j}]$. When we compute $\sigma_{0}(\tilde{T_{j}})$ we can pass the first $\sigma_{0}((\lambda+\mathbb{D}^{2})^{-1})$ over the intermediate terms till it reaches the second $\sigma_{0}((\lambda+\mathbb{D}^{2})^{-1})$. Thus, we get:
\begin{align*}
    \sigma_{0}(\tilde{T_{j}})&=(-1)^{q+1+2q-j}c_{j}\wedge c_{j}(df_{j+1})c_{j+1}\wedge...\wedge c_{2q-1}(df_{2q})c_{2q}\wedge c_{0}f_{0}\wedge c_{0}(df_{1})c_{1}\wedge...\wedge c_{j-1}(df_{j})c_{j}\\
    &\ast \sigma_{0}(e^{\mathbb{D}^{2}}u^{2q-1}(-\mathbb{D}^{2})u'(-\mathbb{D}^{2}))\ast \sigma_{0}([\mathbb{D}^{2},f_{j}]\mathbb{D}).
\end{align*}
Then similarly as in \cite{ref10}, we have:
   \begin{align*}
       \lim_{t\rightarrow 0}\Str(T_{j}(t))&=\int_{T^{*}M}\tr_{s}(\sigma_{0}(\tilde{T_{j}}))dx d\xi\\
       &=-\frac{1}{2q}\delta_{q}(-1)^{q}(2\pi i)^{-n/2}\int_{M}c_{0}\det\left(\frac{R/2}{\sinh{R/2}}\right)^{1/2}\wedge \Str(e^{Q})\wedge f_{0}\wedge df_{1}...\wedge df_{2q}
   \end{align*}
   where $\delta_{q}=\int_{1}^{2}...\int_{1}^{2}(s_{1}+...+s_{q})(1+s_{1}+...+s_{q})^{-(q+1)}ds_{1}...ds_{q}$.

 $(3)$  Computation of the contribution of  $\sum_{j=1}^{j=2q}Z_{j}$.

 The $Z_{j}$ calculations are similar 
to $T_{j}$. Therefore, we have:
\begin{align*}
       \lim_{t\rightarrow 0}\Str(T_{j}(t))=\frac{q!}{(2q+1)!}(-1)^{q}(2\pi i)^{-n/2}\int_{M}c_{0}\det\left(\frac{R/2}{\sinh{R/2}}\right)^{1/2}\wedge \Str(e^{Q})\wedge f_{0}\wedge df_{1}...\wedge df_{2q}.
   \end{align*}

Finally, we have:
\begin{align*}
    &\lim_{t\rightarrow 0}\Str(c_{2q}(e^{t^{2}\mathbb{D}^{2}}-e^{\frac{1}{2}t^{2}\mathbb{D}^{2}}w(-t^{2}\mathbb{D}^{2})t\mathbb{D})c_{0}f_{0}[c_{0}W'(t)c_{1}, f_{1}]...[c_{2q-1}W'(t)c_, f_{2q}])\\
    &=\lim_{t\rightarrow 0}\left(\Str(T(t))+\Str(\sum T_{i}(t))+\Str(\sum Z_{j}(t))\right)\\
    &=(-1)^{n/2-q}(2\pi i)^{n/2-2q}\frac{q!}{(2q)!}\int_{M}c_{0}\det\left(\frac{R/2}{\sinh{R/2}}\right)^{1/2}\wedge \Str(e^{Q})\wedge f_{0}\wedge df_{1}...\wedge df_{2q}.\\
    &=(-1)^{n/2-q}(2\pi i)^{n/2-2q}\frac{q!}{(2q)!}\int_{M}c_{0}\Hat{{A}}_{G}(M)\wedge\ch_{G}^{0}(\pi!([\pi^{*}(E^{+}),\pi^{*}(E^{-}),\sigma(D)]))\wedge f_{0}\wedge df_{1}...\wedge df_{2q}.
\end{align*}

The last equation holds since $\det^{1/2}\left(\frac{R/2}{\sinh{R/2}}\right)$ is just $G$-invariant $\Hat{A}$-genus $\Hat{{A}}_{G}(M)$ in our notation, and $\Str(e^{Q})$ is equal to $\ch_{G}^{0}(\pi!([\pi^{*}(E^{+}),\pi^{*}(E^{-}),\sigma(D)]))$ since $\pi!([\pi^{*}(E^{+}),\pi^{*}(E^{-}),\sigma(D)])=[W^{+}]-[W^{-}]$. 
\end{proof}
Summarizing the preceding result, we complete the proof of Theorem \ref{mthm2}.
\section{Higher index formula and Pairing between $K$-theory and cohomology}
In this section, we aim to use Theorem \ref{mthm2} to obtain a higher index formula and define a pairing between $H^{even}_{DR}(G/H)^{G}$ and $K_{0}(C^{*}(G/H), E)^{G})$ where $H^{even}_{DR}(G/H)^{G}$ is the even part of $G$-invariant de Rham cohomology and $C^{*}(G/H), E)^{G}$ is the equivariant Roe algebra.
\subsection{Higher index formula }
 For a Lie group $G$, the space of smooth homogeneous group $k$-cochains given by
 \begin{align*}
     C^{k}_{\diff}(G):=\{&c:G^{k+1}\rightarrow \mathbb{C} \text{ smooth}|\\
     &c(gg_{0},gg_{1},...,gg_{k})=c(g_{0},g_{1},...,g_{k}), \forall~g,g_{0},...,g_{k}\in G\}
 \end{align*}
 equipped with the differential $d:C^{k}_{\diff}(G)\rightarrow C^{k+1}_{\diff}(G)$ 
 $$d(c)(g_{0},...,g_{k+1}):=\sum_{i=0}^{k+1}(-1)^{i}c(g_{0},...,g_{i-1},g_{i+1},...,g_{k+1})$$
 gives rise to the group cohomology, denoted by $H_{\diff}^{*}(G)$. Under the notation above, let 
 \begin{align*}
     C^{k}_{\diff,\lambda}(G):=\{&c:G^{k+1}\rightarrow \mathbb{C} \text{ smooth}|\\   &c(gg_{0},gg_{1},...,gg_{k})=c(g_{0},g_{1},...,g_{k}),\\
     &c(g_{0},...,g_{k})=(-1)^{k}c(g_{k},g_{0},...,g_{k-1}), \text{ for all }g,g_{0},...,g_{k}\in G\}
 \end{align*}
 and
  \begin{align*}
     C^{k}_{\diff,anti}(G):=\{&c:G^{k+1}\rightarrow \mathbb{C} \text{ smooth }|\\
     &c(gg_{0},gg_{1},...,gg_{k})=c(g_{0},g_{1},...,g_{k}),\\
     &c(g_{0},...,g_{k})=\sign(\sigma)c(g_{\sigma(0)},g_{\sigma(1)},...,g_{\sigma(k)}),\\
     &\text{ for all }g,g_{0},...,g_{k}\in G\text{ and }\sigma\in S_{k+1}\}.
\end{align*}
It can be checked that $(C^{*}_{\diff,\lambda}(G),d)$ and $(C^{*}_{\diff,anti}(G),d)$ are subcomplexs of  $(C^{*}_{\diff}(G),d)$. We denote by $H_{\diff,\lambda}^{*}(G)$ and $H_{\diff,anti}^{*}(G)$  the corresponding cohomology.

Moreover, fix $z_{0}\in M$, for any $g\in G$, and let $L(g):=d(z_{0},gz_{0}))$. There exists a natural subcomplex of the $(C^{*}_{\diff}(G),d)$:
\begin{definition}
Let $C_{\diff,pol}^{q}(G)$ be the 
vector space of all $G$-equivariant smoothing functions from $G^{q+1}$ to $\mathbb{C}$ with at most polynomial growth, that is, for each $c\in C_{\diff, pol}^{q}(G)$, there exists $m\in \mathbb{N}$ such that $$\sup(|c(g_{0},...,g_{q})(1+L(g_{0})+L(g_{1})+...+L(g_{q}))^{-m}|)< +\infty.$$
Thus, $(C^{*}_{\diff,pol}(G),d)$ is a subcomplex of $(C^{*}_{\diff}(G),d)$.
The corresponding cohomology is denoted by $H_{\diff,pol}^{*}(G)$. 
\end{definition}

Similarly, we have:
 \begin{definition}
Let $C_{\diff,pol}^{q}(M)^{G}$ be the 
vector space of all $G$-equivariant smoothing functions from $M^{q+1}$ to $\mathbb{C}$ with at most polynomial growth, that is, for each $\psi\in C_{\diff, pol}^{q}(M)^{G}$, there exists $m\in \mathbb{N}$ such that $$\sup(|\psi(x^{0},...,x^{p})(1+d(x_{0},x_{1})^{2}+d(x_{1},x_{2})^{2}+...+d(x_{p},x_{0})^{2})^{-m}|)< +\infty.$$
A coboundary homomorphism $d:C_{\diff,pol}^{q}(M)^{G}\rightarrow{C_{\diff,pol}^{q}(M)^{G}}$ is defined by the formula 
$$d(\psi)(x^{0},...,x^{p+1})=\sum_{i=0}^{q+1}(-1)^{i}\psi(x^{0},...,x^{i-1},x^{i+1},...,x^{q+1})$$
where $\psi\in C_{\diff,pol}^{q}(M)^{G}$. The corresponding cohomology is denoted by $H_{pol}^{*}(M)^{G}$. 
\end{definition}

 We denote by $(C_{pol,\lambda}^{*}(M)^{G},d)$ and $(C_{pol,anti}^{*}(M)^{G},d)$ subcomplices of $(C_{pol}^{*}(M)^{G},d)$, and $H_{pol,\lambda}^{*}(M)^{G}$, $H_{pol,anti}^{*}(M)^{G}$ the corresponding cohomology respectively. If we remove the restrictions on growth conditions, we denote the corresponding cohomology by $H^{*}(M)^{G}$.

Then, assuming that $G$ has finitely many connected components and satisfies the RD condition with respect to the length function $L(g):=d(z_{0},gz_{0})$, we recall the definition of the pairing between $Z_{\diff, pol}^{*}(G)$ and  $K_{0}(C^{*}(M, E)^{G})$ in \cite{pp2}(5B). 

Let $b\in C^{m}_{\diff}(G)$, for $a_{0},...,a_{m} \in C_{c}^{\infty}(G)$, define
{\small $$\tau_{c}^{G}(a_{0},...,{a_{m}}):=\int_{G^{m}}b(e,g_{1},g_{1}g_{2},...,g_{1}...g_{m})a_{0}((g_{1}...g_{m})^{-1})a_{1}(g_{1})...a_{m}(g_{m})dg_{1}...dg_{m}$$}
and for $k_{0},...,k_{m}\in A_{G}^{c}(M,E)$, define
\begin{align*}
    \tau_{c}^{M}(k_{0},...,{k_{m}}):=\int_{G^{m}}\int_{M^{m+1}}&c(x_{0})...c(x_{m})\tr(k(x_{0},g_{1}x_{1})...k(x_{m},(g_{1}...g_{m})^{-1}x_{0}))\\
    &\cdot b(e,g_{1},g_{1}g_{2},...,g_{1}...g_{m})dg_{1}...dg_{m}d\mu(x_{0})...d\mu(x_{m}).
\end{align*}

When $c\in Z_{\diff,pol}^{m}(G)$, the homomorphism $$\langle\tau_{c}^{M},.\rangle:K_{0}(A^{c}_{G}(M,E))\rightarrow\mathbb{C}$$
extends to a homomorphism
$$\langle\tau_{c}^{M},.\rangle:K_{0}(A^{\infty}_{G}(M,E)\cong K_{0}(C^{*}(M,E)^{G})\rightarrow\mathbb{C}$$
where $C^{*}(M, E)^{G}$ is the equivariant Roe algebra defined as the closure of bounded
G-equivariant operators on $L^{2}(M, E)$ that are locally compact and of finite propagation.
Let
$$A^{\infty}_{G}(M,E):=(H_{L}^{\infty}(G)\hat{\otimes}\Psi^{-\infty}(N,E|_{N}))^{K\times K}$$
where
$$H^{\infty}_{L}(G):=\{f\in L^{2}(G):~g\mapsto (1+L(g))^{k}f(g)\in L^{2}(G)~for~all~k\}$$
and $\Psi^{-\infty}(N,E|_{N})$ denotes the algebra of kernels of smoothing operators on $E|_{N}$.

Then, we explain the relation between the $S^{u}_{G}(M,E)$ and $A^{\infty}_{G}(M,E)$.
In fact, for suitable $u>0$, we have $S^{u}_{G}(M,E)\subset A^{\infty}_{G}(M,E)$. 
Recall that in \cite{pp2}(1A), Piazza and Posthuma defined 
\begin{align*}
A^{\exp}_{G}(M;E):=\{k\in C^{\infty}(M\times M,E\boxtimes E^{*})^{G}:&\sup_{x,y\in M}||\exp(qd(x,y)\nabla_{x}^{m}\nabla_{y}^{n}k(x,y)||<C_{q}\\
& \text{ for all }q,m,n\in \mathbb{N}\}
\end{align*}
and because $M\cong G\times_{K}N$, $A^{\exp}_{G}(M;E)$ can be decomposed as 
$$(A^{\exp}(G)\hat{\otimes}\Psi^{-\infty}(N,E|_{N}))^{K\times K},$$
where 
$$A^{\exp}(G):=\{f\in C^{\infty}(G):\sup_{g\in G}|\exp(qL(g))Df(g)|<+\infty, \text{for all }q,D\in U(\mathfrak{g})\}.$$

Similarly, \begin{align*}
S^{u}_{G}(M;E):=\{k\in C^{\infty}(M\times M,E\boxtimes E^{*})^{G}|&\sup_{x,y\in M}||\exp{(u\cdot d(x,y))}\nabla_{x}^{m}\nabla_{y}^{n}k(x,y)||<C_{m,n}\\
& \text{for all }m,n\in \mathbb{N}\}.
\end{align*} can be decomposed as
$$(A^{\exp,u}(G)\hat{\otimes}\Psi^{-\infty}(N,E|_{N}))^{K\times K},$$
where 
$$A^{\exp,u}(G):=\{f\in C^{\infty}(G):\sup_{g\in G}|\exp(uL(g))Df(g)|<+\infty, \text{for all }D\in U(\mathfrak{g})\}.$$

Since the Lie group $G$ has at most exponential growth, when $u$ is big enough, we have: $A^{\exp,u}(G)\subset H^{\infty}_{L}(G)$. Therefore,  $S^{u}_{G}(M,E)\subset A^{\infty}_{G}(M,E)$ holds for suitable $u>0$. 

Then, as we have shown in the last section, for any $u>0$, there exists $t_{u}>0$, and for any $0<t\leq t_{u}$, we have $\Ind_{t}(D)\in S^{u}_{G}(M, E)$. We claim that:
\begin{proposition}
\label{prop7.2}
    There exists $t_{u}>0$, for any $0<t\leq t_{u}$, $\Ind_{t}(D)\in A^{\infty}_{G}(M,E)$.
\end{proposition}

Therefore, for any $c\in Z^{m}_{\diff, pol}(G)$, $\langle\tau^{M}_{c},\Ind_{t}(D)\rangle$ is well defined for suitable $t$. Next, we explain the relation between 
$\langle\tau^{M}_{c},\Ind_{t}(D)\rangle$ and $\tau(f_{c})(\Ind_{t}(D))$ which is useful to obtain the higher index formula.

Indeed, we have:
\begin{proposition}
\label{prop10.4}
    Let $a\in C^{m}_{\diff}(G)$ and $k_{0},...,k_{m}\in C^{\infty}(M\times M, E\boxtimes E^{*})^{G}$. If both $\tau^{M}_{a}(k_{0},...,{k_{m}})$ and $\tau(f_{a})(k_{0},...,{k_{m}})$ are well defined where $$f_{a}(x_{0},...,x_{m}):=\int_{G^{m+1}}c(g_{0}^{-1}x_{0})...c(g_{m}^{-1}x_{m})a(g_{0},...,g_{m})dg_{0}...dg_{m},$$
    then
    $$\tau^{M}_{a}(k_{0},...,{k_{m}})=\tau(f_{a})(k_{0},...,{k_{m}}).$$

\end{proposition}
\begin{proof}
To prove the proposition, we need to show that:
\begin{align*}
    \tau_{a}^{M}(k_{0},...,{k_{m}})=\int_{M^{m+1}}c(x_{0})f_{a}(x_{0},...,x_{m})\tr(k(x_{0},x_{1})...k(x_{m},x_{0})d\mu(x_{0})...d\mu(x_{m}).
    \end{align*}

    By definition, we have:
 {\small    \begin{multline}
 \label{5}
    \tau_{a}^{M}(k_{0},...,{k_{m}}):=\int_{G^{m}}\int_{M^{m+1}}c(x_{0})...c(x_{m})\tr(k_{0}(x_{0},g_{1}x_{1})...k_{m}(x_{m},(g_{1}...g_{m})^{-1}x_{0}))\\
    \cdot a(e,g_{1},g_{1}g_{2},...,g_{1}...g_{m})dg_{1}...dg_{m}d\mu(x_{0})...d\mu(x_{m}).
    \end{multline}
}
    
Since $\int_{G}c(g_{0}^{-1}x_{0})dg_{0}=1$, for any $x_{0}\in M$ and $a$ is $G$-invariant, $(\ref{5})$ can be written as:
{\small
\begin{multline}
\label{6}
\int_{G^{m+1}}\int_{M^{m+1}}c(g_{0}^{-1}x_{0})c(x_{0})...c(x_{m})\tr(k(x_{0},g_{1}x_{1})...k(x_{m},(g_{1}...g_{m})^{-1}x_{0}))\\
    \cdot a(g_{0}^{-1},g_{0}^{-1}g_{1},g_{0}^{-1}g_{1}g_{2},...,g_{0}^{-1}g_{1}...g_{m})dg_{0}dg_{1}...dg_{m}d\mu(x_{0})...d\mu(x_{m}).
\end{multline}
}
     Let $u_{0}=g_{0}^{-1}, u_{1}=g_{0}^{-1}g_{1},...,u_{m}=g_{0}^{-1}g_{1}...g_{m}$. Then, 
{\small   \begin{multline*}
(\ref{6})=\int_{G^{m+1}}\int_{M^{m+1}}c(u_{0}x_{0})c(x_{0})...c(x_{m})\tr(k_{0}(x_{0},u_{0}^{-1}u_{1}x_{1})...k(_{m}x_{m},u_{m}^{-1}u_{0}x_{0}))\\
    \cdot a(u_{0},u_{1},u_{2},...,u_{m})d(u_{0}^{-1})d(u_{0}^{-1}u_{1})...d(u_{m-1}^{-1}u_{m})d\mu(x_{0})...d\mu(x_{m})
     \end{multline*}
    \begin{align*}
=\int_{G^{m+1}}\int_{M^{m+1}}&c(x_{0})c(u_{0}^{-1}x_{0})...c(u_{m}^{-1}x_{m})\tr(k_{0}(x_{0},u_{1}x_{1})...k_{m}(u_{m}x_{m},x_{0}))\\
    &\cdot a(u_{0},u_{1},u_{2},...,u_{m})du_{0}du_{1}...du_{m}d\mu(x_{0})...d\mu(x_{m})
    \end{align*}
    $$=\int_{M^{m+1}}c(x_{0})f_{a}(x_{0},...,x_{m})\tr(k_{0}(x_{0},x_{1})...k_{m}(x_{m},x_{0}))d\mu(x_{0})...d\mu(x_{m}).$$ }
    The first and second equalities hold since $du_{0}^{-1}=du_{0}$, $d(u_{i}^{-1}u_{i+1})=du_{i+1}$, $d\mu(u_{j}x_{j})=d\mu(x_{j})$  and $k_{1},...,k_{m}$ are $G$-invariant.
    \end{proof}

    On the other hand, we claim that:
    \begin{proposition}
         When $a\in C^{q}_{\diff,anti}(G)$ has almost polynomial growth, we have $f_{a}\in C_{pol,anti}^{q}(M)^{G}$.
    \end{proposition}
    \begin{proof}
        It is clear that when $a\in C^{q}_{\diff,anti}(G)$, $f_{a}\in C_{anti}^{q}(M)^{G}$. Thus, to prove the claim, we only need to show that
        $a$ having almost polynomial growth implies that $f_{a}$ has almost polynomial growth.

        Recall that $a\in C^{q}_{\diff}(G)$  having almost polynomial growth means that there exists $m\in \mathbb{N}$ such that $$|a(g_{0},...,g_{q})|\leq(1+L(g_{0})+L(g_{1})+...+L(g_{q}))^{m},$$
        where $L(g_{i})=d(z_{0}.g_{i}z_{0})$ for some  $z_{0}\in M$.

        Thus, for any $x_{0},...,x_{q}\in M$, we have:
        \begin{align}
            |f_{a}(x_{0},...,x_{q})|&=|\int_{G^{q+1}}c(g_{0}^{-1}x_{0})...c(g_{q}^{-1}x_{m})a(g_{0},...,g_{q})dg_{0}...dg_{q}|\notag\\
          &\leq \int_{G^{m+1}}c(g_{0}^{-1}x_{0})...c(g_{m}^{-1}x_{m})(1+L(g_{0})+L(g_{1})+...+L(g_{q}))^{m}dg_{0}...dg_{m}\notag\\
&=\int_{G^{m+1}}c(g_{0}^{-1}x_{0})...c(g_{m}^{-1}x_{m})(1+d(z_{0},g_{0}z_{0})+...+d(z_{0},g_{q}z_{0}))^{m}dg_{0}...dg_{m}\label{a}
        \end{align}

        Moreover, for any $x\in M$ and $m\in \mathbb{N}^{+}$, let $$\psi_{i,m}(x):=\int_{G}c(g_{i}^{-1}x)d(z_{0},g_{i}z_{0})^{m}dg_{i}.$$
        Since (\ref{a}) can be rewritten as :
        $$\sum_{m_{0}+...m_{q}=m}\int_{G^{m+1}}c(g_{0}^{-1}x_{0})...c(g_{m}^{-1}x_{m})d(z_{0},g_{0}z_{0})^{m_{0}}...d(z_{0},g_{q}z_{0})^{m_{q}}dg_{0}...dg_{m}$$
        where $m_{i}$ runs over $\{0,1,2,...q\}$,
        we have:
        $$ |f_{a}(x_{0},...,x_{q})|\leq\sum_{m_{0}+...m_{q}=m}\psi_{0,m_{0}}(x_{0})...\psi_{q,m_{q}}(x_{q}).$$
        Since $d(,)$ is left-invariant, for any $g\in G$ and $x\in M$,
        we have:
        \begin{align*}
        d(z_{0},gz_{0})&=d(g^{-1}z_{0},z_{0})\\
        &\leq d(g^{-1}x,g^{-1}z_{0})+d(g^{-1}x,z_{0})\\
        &=d(x,z_{0})+d(g^{-1}x,z_{0}).
        \end{align*}
        On the other hand, for any $p\in \mathbb{N}^{+}$, there exists $C>0$ such that, $$c(g^{-1}x)d(g^{-1}x,z_{0})^{p}\leq C^{p}.$$
        This is because $\supp(c)$ is compact and for any $y\in M$, we have:
        $$c(y)d(y,z_{0})^{p}\leq (\sup_{z\in \supp(c)}d(z_{0},z))^{p}.$$

        Thus, we have:
        \begin{align*}
            \psi_{i,m}(x)&=\int_{G}c(g_{i}^{-1}x)d(z_{0},g_{i}z_{0})^{m}dg_{i}\\
            &\leq\int_{G}c(g_{i}^{-1}x)(C+d(x,z_{0}))^{m}dg_{i}\\
            &=(C+d(x,z_{0}))^{m}.
        \end{align*}
        Because of this, 
        \begin{align*}
            |f_{a}(x_{0},...,x_{q})|&\leq\sum_{m_{0}+...m_{q}=m}\psi_{0,m_{0}}(x_{0})...\psi_{q,m_{q}}(x_{q})\\
            &\leq \sum_{m_{0}+...m_{q}=m}(C+d(x_{0},z_{0}))^{m_{0}}... (C+d(x_{q},z_{0}))^{m_{q}}
        \end{align*}
        which shows that $f_{a}$ has almost polynomial growth.
    \end{proof}

    Because of the proposition above, for any $v>0$, we have $C^{q}_{pol}(M)^{G}\subset C^{q}_{v}(M)^{G}$. Thus, for any $a\in Z_{\diff,anti}^{2q}(G)$ and suitable $0<t\leq 1$, $\tau(f_{a})(\Ind_{t}(D))$ is well defined and following from Proposition \ref{prop7.2}, we have:
    $$\langle\tau_{a}^{M},\Ind_{t}(D)\rangle=\frac{(2q)!}{q!}\tau_{a}^{M}(\Ind_{t}(D))=\frac{(2q)!}{q!}\tau(f_{a})(\Ind_{t}(D)).$$

    Since $\langle\tau_{a}^{M},\Ind_{t}(D)\rangle$ is independent of $t$, we have $$\langle\tau_{a}^{M},\Ind_{t}(D)\rangle=\lim_{t\rightarrow 0}\langle\tau_{a}^{M},\Ind_{t}(D)\rangle=\lim_{t\rightarrow 0}\frac{(2q)!}{q!}\tau(f_{a})(\Ind_{t}(D)).$$
    Thus, Theorem \ref{mthm2} can be used to obtain a higher index formula:
    \begin{theorem}
    \label{thm7.6}
        Let $a\in Z_{\diff,anti}^{2q}(G)$  be a cocycle of polynomial growth and $\Ind_{t}(D)\in K_{0}(A^{\infty}_{G}(M,E))$. Then the following $C^{*}$-higher index formula holds:
        $$\langle\tau_{a}^{M},\Ind_{t}(D)\rangle=\frac{(-1)^{n/2-q}}{(2\pi i)^{2q-n/2}}\int_{M}c_{0}\Hat{{A}}_{G}(M)\wedge\ch_{G}^{0}(\pi!([\pi^{*}(E^{+}),\pi^{*}(E^{-}),\sigma(D)]))\wedge \omega_{a}.$$
    \end{theorem}
    
   \subsection{Pairing between $K$-theory and cohomology}
    Now we consider the case $M\cong G/H$ and define the pairing between $H^{even}_{DR}(G/H)^{G}$ and $K_{0}(A^{\infty}_{G}(G/H,E))$.

First, we recall the definition of the van Est map linking $H_{\diff}^{*}(G)$ and $H^{*}(M)^{G}$.

As in \cite{pp1} and \cite{pp2}, the van Est map from the smooth group cohomology $H_{\diff}^{*}(G)$ to the invariant cohomology $H^{*}_{DR}(M)^{G}$ is constructed as follows: given a smooth group cochain $a\in C^{k}_{\diff}(G)$, define the differential form:
$$\omega_{a}:=(d_{1}...d_{k}f_{a})|_{\Delta}
$$
where $d_{i}$ is the differential with respect to the $i$-th variable and the function $f_{a}\in C^{\infty}(M)$ is given by:
$$f_{a}(x_{0},...,x_{k}):=\int_{G^{k+1}}c(g_{0}^{-1}x_{0})...c(g_{k}^{-1}x_{k})a(g_{0},...,g_{k})dg_{0}...dg_{k}.$$
Thus, the van Est homomorphism can be defined as follows:
\begin{definition}
    The map $a\rightarrow \omega_{a}$ gives rise to a homomorphism from $H_{\diff}^{*}(G)$ to $H_{DR}^{*}(M)^{G}$ :
    \begin{align*}
        \Psi_{M}:H_{\diff}^{*}(&G)\rightarrow H^{*}_{DR}(M)^{G}\\
        &[a]\mapsto [\omega_{a}].
    \end{align*}
\end{definition}
\begin{remark}
     Proposition $2.5$ in \cite{pp1} shows that $\Psi_{M}$ is a homomorphism from $H_{\diff}^{*}(G)$ to $H^{*}(M)^{G}$. Similarly, $a\rightarrow \omega_{a}$ induces a homomorphism from $H_{\lambda}^{*}(G)$(or $H_{anti}^{*}(G)$) to $H^{*}(M)^{G}$. We  denote these homomorphisms by $\Psi_{M,\lambda}$ and $\Psi_{M,anti}$ respectively. When $M=G/H$ where $H$ is a maximal compact subgroup of $G$. Remark~$2.7$ in \cite{pp1} shows that $\Psi_{M}$ is an isomorphism.
\end{remark}

Following the above notation, we have:
\begin{proposition}
\label{prop1}
    For any $k,l\in \mathbb{N}$, 
    \begin{align*}
        &(1)f_{\epsilon(a)}=\epsilon(f_{a})\\
        &(2)\omega_{g}=\omega_{\epsilon(g)}
        \end{align*}
        hold for any $a\in C^{k}_{\diff}(G)$ and $g\in C^{l}_{\lambda}(M)^{G}$.
\end{proposition}
\begin{proof}
    For any $a\in C^{k}_{\diff}(G)$, we have:
    \begin{align*}
        f_{\epsilon(a)}(x_{0},...,x_{k})&=\int_{G^{k+1}}c(g_{0}^{-1}x_{0})...c(g_{k}^{-1}x_{k})\epsilon(a)(g_{0},...,g_{k})dg_{0}...dg_{k}\\
        &=\frac{1}{(k+1)!}\sum_{\sigma\in S_{k}}\int_{G^{k+1}}c(g_{0}^{-1}x_{0})...c(g_{k}^{-1}x_{k})a(g_{\sigma({0})},...,g_{\sigma({k})})dg_{0}...dg_{k}\\
        &=\frac{1}{(k+1)!}\sum_{\sigma\in S_{k}}\int_{G^{k+1}}c(g_{0}^{-1}x_{\sigma(0)})...c(g_{k}^{-1}x_{\sigma(k)})a(g_{0},...,g_{k})dg_{0}...dg_{k}\\
        &=\epsilon(f_{a})(x_{0},...,x_{k})
    \end{align*}
    which proves the first claim.

    Next, let $t:C^{l}(M)^{G}\rightarrow C^{l}(M)^{G}$ be the cyclic operator which means that for any $p\in C^{l}(M)^{G}$, we have:
    $$t(p)(x_{0},...x_{l})=(-1)^{l}p(x_{l},x_{0},...,x_{l-1}).$$
    Thus, let $\lambda:C^{l}(M)^{G}\rightarrow C^{l}(M)^{G}$ be defined by
    $$\lambda(p)=\frac{1}{l+1}\sum_{i=0}^{l}t^{l}(p)$$
     for any $p\in C^{l}(M)^{G}$. 

     Then, for convenience, let $p=p_{0}\otimes...\otimes p_{l}\in C^{l}(M)^{G}$, we have:
    \begin{align*}
        \omega_{\epsilon(p)}&=\frac{1}{(l+1)!}\sum_{\sigma\in S_{l+1}}\sign(\sigma)p_{\sigma(0)}dp_{\sigma(1)}\wedge...\wedge dp_{\sigma(l)}\\
        &=\frac{1}{l+1}\sum_{i=0}^{l}(-1)^{i}p_{i}dp_{0}\wedge...dp_{i-1}\wedge dp_{i+1}\wedge...\wedge dp_{l}\\
        &=\omega_{\lambda(p)}.
    \end{align*}
    Since $p=\lambda(p)\in C^{l}_{\lambda}(M)^{G}$, we have:
    $$\omega_{p}=\omega_{\lambda(p)}=\omega_{\epsilon(p)}.$$
    Therefore, we proved the second claim.
\end{proof}

On the other hand, there exists an isomorphism from  $H^{*}_{DR}(G/H)^{G}$ to $H_{\diff}^{*}(G)$ when $H$ is a maximal compact subgroup of $G$.:
\begin{definition}
    The map $\alpha\rightarrow J(\alpha)$ induces an isomorphism from   $H^{*}_{DR}(G/H)^{G}$ to $H_{\diff}^{*}(G)$:
    \begin{align*}
        J_{*}:H^{*}_{DR}(&G/H)^{G}\rightarrow H_{\diff}^{*}(G)\\
        &[\alpha]\mapsto [J(\alpha)].
    \end{align*}
\end{definition}
\begin{remark}
    As in \cite{pp1}, $J_{*}$ is the inverse of $\Psi_{G/H}$ and when $G/H$ is of non-positive sectional curvature, $J(\alpha)$ has almost polynomial growth.
\end{remark}

On the other hand, we have:
\begin{definition}
    The map $a\rightarrow f_{a}$ induces a homomorphism from $H_{\diff}^{*}(G)$ to $H^{*}_{DR}(G/H)^{G}$.
\end{definition}
Thus, we can construct a homomorphism from $H^{*}_{DR}(G/H)^{G}$ to $H_{pol,anti}^{*}(G/H)^{G}$ as follows:
\begin{definition}
    The map $\alpha\rightarrow f_{\epsilon\circ J(\alpha)}$ induced a homomorphism from $H^{*}_{DR}(G/H)^{G}$ to $H_{pol,anti}^{*}(G/H)^{G}$ when $G/H$ is of non-positive sectional curvature.
\end{definition}
Thus, using the homomorphism defined by Piazza and Posthuma in \cite{pp1} and \cite{pp2}, we can define a pairing between $H^{even}_{DR}(G/H)$ and $K_{0}(A^{\infty}_{G}(G/H, E))$:
\begin{definition}
    For any $[\alpha]\in H^{even}_{DR}(G/H)$ and $[p]-[q]\in K_{0}(A^{\infty}_{G}(G/H,E))$, we define:
    $$\langle[\alpha],[p]-[q]\rangle_{DR,G/H}:=\langle\tau^{G/H}_{\epsilon\circ J(\alpha)},[p]-[q]\rangle.$$
    Here $G$ has finitely many connected components, $H$ is a maximal compact subgroup of $G$ and $G/H$ is of non-positive sectional curvature.
\end{definition}
Similarly, there exists a pairing between $H^{even}_{DR}(G/H)$ and $K_{0}(C^{*}_{r}(G))$:
\begin{definition}
    For any $[\alpha]\in H^{even}_{DR}(G/H)$ and $[p]-[q]\in K_{0}(C^{*}_{r}(G))$, we define:
    $$\langle[\alpha],[p]-[q]\rangle_{DR,G}:=\langle\tau^{G}_{\epsilon\circ J(\alpha)},[p]-[q]\rangle.$$
    Here $G$ has finitely many connected components and satisfies the RD condition, $H$ is a maximal compact subgroup of $G$ and $G/H$ is of non-positive sectional curvature.
\end{definition}

\section{Equivalence of topological index and analytic index}
Let $D$ be a self-adjoint, $G$-equivariant, elliptic differential operator on a $\mathbb{Z}_{2}$-graded, $G$-equivariant vector bundle $E$ over $G/H$. In \cite{wang3}, the higher index $\Ind_{G}(D)\in K_{0}(C_{r}^{*}(G))$ is defined by using the analytic assembly map from the Baum-Connes conjecture. Since the $K$-homology class of $D$  can be identified with $[\pi^{*}(E^{+}),\pi^{*}(E^{-}),\sigma(D)]\in K_{G}^{0}(T^{*}(G/H))$ under Poincar\'e duality, we obtain a map:
\begin{align*}
    &\Ind_{A}:K_{G}^{0}(T^{*}(G/H))\rightarrow  K_{0}(C_{r}^{*}(G))\\
    &[\pi^{*}(E^{+}),\pi^{*}(E^{-}),\sigma(D)]\mapsto\Ind_{G}(D).
\end{align*}

On the other hand, we define the topological index map as follows:
\begin{align*}
    &\Ind_{T}:K_{G}^{0}(T^{*}(G/H))\rightarrow H^{even}_{DR}(G/H)^{G}\\
    &[\pi^{*}(E^{+}),\pi^{*}(E^{-}),\sigma(D)]\mapsto \pi!((\Hat{A}^{2}_{G}(G/H)))\wedge \ch_{G}^{0}([\pi^{*}(E^{+}),\pi^{*}(E^{-}),\sigma(D)])).
\end{align*}

We denote $H_{DR,even}(G/H)^{G}:=\Hom(H^{even}_{DR}(G/H)^{G},\mathbb{C})$. Then there exists a map $\PD:H^{even}_{DR}(G/H)^{G}\rightarrow H_{DR,even}(G/H)^{G}$ such that: for any $[\alpha],[\beta]\in H^{even}_{DR}(G/H)^{G}$, $$\PD([\alpha])([\beta])=\sum_{k=0}^{\frac{n}{2}}\frac{({-1})^{k}}{(2\pi i)^{2k-n}}\int_{G/H}c_{0}\alpha_{2k}\wedge\beta_{n-2k}$$
where $\dim(G/H)=n$, $\alpha_{2k}$ and $\beta_{n-2k}$ are the $2k$-degree component of $\alpha$ and ${(n-2k)}$-degree component of $\beta$ respectively.

In this section, we aim to prove the agreement between the topological index and analytic index which is stated   as follows:
\begin{theorem}
\label{mth}
The following diagram commutes:
\begin{equation*}
    \xymatrix{   K_{G}^{0}(T^{*}(G/H))\ar[r]^{\Ind_{A}}\ar[d]_{\Ind_{T}} &  K_{0}(C_{r}^{*}(G)) \ar[d]^{\tilde{\ch} }\\
         H^{even}_{DR}(G/H)^{G}\ar[r]^{PD}&H_{DR,even}(G/H)^{G}
                      }
\end{equation*}
where $\tilde{\ch}:K_{0}(C^{*}_{r}(G))\rightarrow H_{DR,even}(G/H)^{G}$ is defined by
$$\tilde{\ch}([p]-[q])(\alpha):=\langle\alpha,[p]-[q]\rangle_{DR,G}$$
for every $\alpha\in  H^{even}_{DR}(G/H)^{G}$ and $[p]-[q]\in K_{0}(C^{*}_{r}(G))$. Here $G$ has finitely many connected components and satisfies the RD condition, $H$ is a maximal compact subgroup of $G$ and $G/H$ is of non-positive sectional curvature.
\end{theorem}

To prove Theorem \ref{mth},  we need to show:
 Let $k\in \mathbb{N}$, for any $[\alpha]\in H^{2k}_{DR}(G/H)^{G}$, 
    $$\langle[\alpha],\Ind_{t}(D)\rangle_{DR,G/H}=\frac{(-1)^{k-n/2}}{(2\pi i)^{2k-n/2}}\int_{G/H}c_{0}\alpha\wedge\hat{A}_{G}(G/H)\wedge\ch_{G}^{0}(\pi!([\pi^{*}(E^{+}),\pi^{*}(E^{-}),\sigma(D)]))$$
    where $D$ is the Dirac operator on $E$.

 We will split the proof of the claim above into two lemmas.
    
    First, we explain the relation between $H^{*}_{\diff}(G)$ and  $H_{anti}^{*}(G)$.
\begin{lemma}
\label{anti}
    The anti-symmetric operator $\epsilon_{*}$ induces an isomorphism from  $H^{*}_{\diff}(G)$ to $H_{anti}^{*}(G)$.
\end{lemma}
\begin{proof}
     First, we claim that the following diagram commutes:
     $$\begin{tikzcd}
H_{\lambda}^{*}(G)\arrow[r, "\epsilon_{*}"] \arrow[rd,"\Psi_{G/H,\lambda}"] & H_{anti}^{*}(G) \arrow[d, "\Psi_{G/H,anti}"]                              \\
                                           & H^{*}_{DR}(G/H)^{G}.  
\end{tikzcd}$$

For any $[a]\in H^{k}_{\diff,\lambda}(G)$, we have :$[\Psi_{G/H}\circ\epsilon(a)]=[\omega_{f_{\epsilon(a)}}]$.

Because of Proposition \ref{prop1}, we have:
$$[\Psi_{G/H}\circ\epsilon(a)]=[\omega_{f_{\epsilon(a)}}]=[\omega_{\epsilon(f_{a})}]=[\omega_{f_{a}}]=[\Psi_{G/H}(a)].$$
Thus, the above diagram commutes,
which shows the claim.

According to \cite{pp1}(5B), $C^{*}_{\diff,\lambda}(G)$ is a  quasi-isomorphic subcomplex of $C^{*}_{\diff}(G)$, and so the inclusion
     $i:C^{*}_{\diff,\lambda}(G)\rightarrow C^{*}_{\diff}(G)$ induces an isomorphism from $H_{\lambda}^{*}(G)$ to $ H_{\diff}^{*}(G)$. Since $\Psi_{G/H}:H^{*}_{\diff}(G)\rightarrow H^{*}_{DR}(G/H)^{G}$ is an isomorphism, 
     $\Psi_{G/H,\lambda}=\Psi_{G/H}\circ i_{*}$ is also an isomorphism. This implies that $\epsilon_{*}$ is injective because $\Psi_{G/H,\lambda}=\Psi_{G/H,anti}\circ\epsilon_{*}$.

     Since the following diagram commutes:
     $$\begin{tikzcd}
H_{anti}^{*}(G)\arrow[r, "i_{*}"] \arrow[rd,"\Id"] & H_{\lambda}^{*}(G) \arrow[d, "\epsilon_{*}"]                              \\
                                           & H^{*}_{anti}(G)  
\end{tikzcd}$$
 this means $\epsilon_{*}$ is also surjective. Thus, $\epsilon_{*}:H^{*}_{\lambda}(G)\rightarrow H^{*}_{anti}(G)$ is an isomorphism.

 Finally, since the following diagram commutes:
     $$\begin{tikzcd}
H_{\lambda}^{*}(G)\arrow[r, "i_{*}"] \arrow[rd,"\epsilon_{*}"] & H_{\diff}^{*}(G) \arrow[d, "\epsilon_{*}"]                              \\
                                           & H^{*}_{anti}(G),  
\end{tikzcd}$$
$\epsilon_{*}:H^{*}_{\diff}(G)\rightarrow H^{*}_{anti}(G)$ is an isomorphism, which completes the proof.
\end{proof}
Then, we prove the second lemma we need.
\begin{lemma}
\label{7.35}
    For any $[\alpha]\in H^{*}_{DR}(G/H)^{G}$, we have:
    $$[\alpha]=[\Psi_{G/H}\circ\epsilon\circ J(\alpha)].$$
\end{lemma}
\begin{proof}
    First, we show that the following diagram commutes:
    $$\begin{tikzcd}
H_{\diff}^{*}(G)\arrow[r, "\Psi_{G/H}"]  & H_{DR}^{*}(G/H)^{G}\arrow[d, "J_{*}"]                              \\
        H_{anti}^{*}(G)\arrow[u, "i_{*}"]
                            &H^{*}_{\diff}(G).\arrow[l,"\epsilon_{*}"]
\end{tikzcd}$$
The above diagram commutes since $\Id=J_{*}\circ\Psi_{G/H} $ and $\Id=i_{*}\circ\epsilon_{*}$.

Since all homomorphisms appearing in the diagram above are  isomorphisms, the following diagram also commutes:
 $$\begin{tikzcd}
H_{DR}^{*}(G/H)^{G}\arrow[r, "J_{*}"]  & H_{\diff}^{*}(G)^{}\arrow[d, "\epsilon_{*}"]                              \\
        H_{\diff}^{*}(G)\arrow[u, "\Psi_{G/H}"]
                            &H^{*}_{anti}(G).\arrow[l,"i_{*}"] 
\end{tikzcd}$$

This implies that 
    $$[\alpha]=[\Psi_{G/H}\circ\epsilon\circ J(\alpha)]\quad \forall [\alpha]\in H^{*}_{DR}(G/H)^{G}$$
    and the proof is completed.
\end{proof}

After these preparations, we will prove the following theorem:

\begin{theorem}
\label{thm8.4}
    Let $k\in \mathbb{N}$. For any $[\alpha]\in H^{2k}_{DR}(G/H)^{G}$, we have:
    $$\langle[\alpha],\Ind_{t}(D)\rangle_{DR,G/H}=\frac{(-1)^{n/2-k}}{(2\pi i)^{2k-n/2}}\int_{G/H}c_{0}\alpha\wedge\hat{A}_{G}(G/H)\wedge\ch_{G}^{0}(\pi!([E^{+},E^{-},\sigma(D)])).$$
\end{theorem}
\begin{proof}
 Indeed, we have:
    \begin{align*}
        &\langle[\alpha],\Ind_{t}(D)\rangle_{DR,G/H}=\langle\tau^{G/H}_{\epsilon\circ J(\alpha)},\Ind_{t}(D)\rangle\\
        =&\frac{(2k)!}{k!}\frac{(-1)^{n/2-k}}{(2\pi i)^{2k-n/2}}\frac{k!}{(2k)!}\int_{G/H}c_{0}\omega_{f_{\epsilon\circ J(\alpha)}}\wedge\hat{A}_{G}(G/H)\wedge\ch_{G}^{0}(\pi!([E^{+},E^{-},\sigma(D)]))\\
        =&\frac{(-1)^{n/2-k}}{(2\pi i)^{2k-n/2}}\int_{G/H}c_{0}\Psi_{G/H}\circ{\epsilon\circ J(\alpha)}\wedge\hat{A}_{G}(G/H)\wedge\ch_{G}^{0}(\pi!([E^{+},E^{-},\sigma(D)]))\\
        =&\frac{(-1)^{n/2-k}}{(2\pi i)^{2k-n/2}}\int_{G/H}c_{0}\alpha\wedge\hat{A}_{G}(G/H)\wedge\ch_{G}^{0}(\pi!([E^{+},E^{-},\sigma(D)])).
    \end{align*}
    The second equality holds because of Theorem \ref{mthm2} and by Lemma \ref{7.35} the last equality holds.
\end{proof}
  
    Therefore, we are ready to prove Theorem \ref{mth}.
\begin{proof}[Proof of Theorem \ref{mth}]
    Let $[\pi^{*}(E^{+}),\pi^{*}(E^{-}),\sigma(D)]\in K_{G}^{0}(T^{*}(G/H))$ and $[\alpha]\in H^{2k}_{DR}(G/H)^{G}$. Let $\mathcal{M}: K_{0}(C^{*}(M,E)^{G})\rightarrow K_{0}(C^{*}_{r}(G))$ be the Morita isomorphism. Then, $$\langle\tau^{G/H}_{\epsilon\circ J(\alpha)},\Ind_{t}(D)\rangle=\langle\tau^{G}_{\epsilon\circ J(\alpha)},\mathcal{M}\circ\Ind_{t}(D)\rangle.$$
    holds following from Proposition 5.7 in \cite{pp1}. Since $\Ind_{A}(D)=\mathcal{M}\circ\Ind_{t}(D)$ (independent of $t\in (0,1]$), we have:
    $$\langle[\alpha],\Ind_{A}(D)\rangle_{DR,G}=\langle\tau^{G}_{\epsilon\circ J(\alpha)},\Ind_{A}(D)\rangle=\langle\tau^{G/H}_{\epsilon\circ J(\alpha)},\Ind_{t}(D)\rangle=\langle[\alpha],\Ind_{t}(D)\rangle_{DR,G/H}.$$

    Then, applying Theorem \ref{thm8.4}, we have:
    \begin{align*}
        \langle[\alpha],\Ind_{A}(D)\rangle_{DR,G}&=\langle[\alpha],\Ind_{t}(D)\rangle_{DR,G/H}\\
        &=\frac{(-1)^{n/2-k}}{(2\pi i)^{2k-n/2}}\int_{G/H}c_{0}\alpha\wedge\hat{A}_{G}(G/H)\wedge \ch_{G}^{0}(\pi!([\pi^{*}(E^{+}),\pi^{*}(E^{-}),\sigma(D)]))\\
        &=\frac{({-1})^{k}}{(2\pi i)^{2k-n}}\int_{T^{*}(G/H)}c_{0}\alpha\wedge\hat{A}_{G}^{2}(G/H)\wedge\ch_{G}^{0}([\pi^{*}(E^{+}),\pi^{*}(E^{-}),\sigma(D)])\\
        &=\PD(\Ind_{T}(D))(\alpha).
    \end{align*}
    Therefore, we complete the proof.
\end{proof}
\begin{remark}
    If we identified $H^{*}(G/H)^{G}$ with $H^{*}(\mathfrak{g},H)$. Following Theorem~\ref{thm6.1} and Theorem~\ref{mth}, for any $\alpha\in H^{2k}(\mathfrak{g},H)$, we have:
    $$\langle[\alpha],\Ind_{A}(D)\rangle_{DR,G}=\frac{1}{(2\pi i)^{k}}\langle\alpha\wedge\hat{A}(\mathfrak{g},H)\wedge\ch_{CM}(\pi!(\sigma(D)),[V]\rangle$$
    where $[V]$is the fundamental class of $T^{*}_{[e]}(G/H)$.
   This is the higher generalization of \cite{Wang1} and an analytic proof of an example of \cite{ppt2}. 
\end{remark}
\appendix
\section{} 

\begin{proof}[Proof of Lemma \ref{c1}]
    
      Fix an $x_{0}\in M$ and let $\tilde{s}\in \Gamma(M, E)$ such that $\supp(\tilde{s})\subset B_{1}(x_{0})$. Since the group action is proper and cocompact, we only need to consider the special case of $B_{1}(x_{0})$. By the Sobolev embedding theorem and the elliptic estimate of $D$, there exist $C>0$ such that:
    $$||\tilde{s}||_{C^{0}(B_{1}(x_{0}))}\leq C\sum_{i=0}^{n}||D^{i}
    \tilde{s}||_{0,B_{1}(x_{0})}.$$
    For any $0<r\leq 1$, define $f(x)=\exp(-\frac{r}{r-d(x,x_{0})^{2}})$ when $d(x,x_{0})\leq r$ and $f(x)=0$ when $d(x,x_{0})>r$, . Thus, for any $s\in \Gamma(M,E)$ we have:
    $$|s(x_{0})|\leq e\cdot\sup_{x\in B_{r}(x_{0})}|f s(x)|\leq C\sum_{i=0}^{n}||D^{i}(f s)||_{0,B_{1}(x_{0})}.
    $$

    Let $\psi \in C^{\infty}_{c}(B_{r}(x_{0}))$ and $w_{k}(\psi)=\sup_{|\alpha|\leq k, x\in B_{1}(x_{0})}|\partial_{\alpha}\psi(x)|$. Thus, 
    \begin{align*}
        ||\psi s||_{0,B_{1}(x_{0})}&=\left(\int_{B_{1}(x_{0})}|\psi s(x)|^{2}dx\right)^{\frac{1}{2}}\\
        &\leq \sup_{x\in B_{1}(x_{0})}|\psi(x)|\cdot\left(\int_{B_{r}(x_{0})}| s(x)|^{2}dx\right)^{\frac{1}{2}}\\
        &=w_{0}(\psi)\cdot||s||_{0,B_{r}(x_{0})}.
    \end{align*}
    We may assume there exist $C_{m}>0$ such that:
    $$||D^{m}(\psi s)||_{0,B_{1}(x_{0})}\leq C_{m}\cdot w_{m}(\psi)\cdot\left(\sum_{i=0}^{m}||D^{i}s||_{0,B_{r}(x_{0})}\right)
    $$
    for any $s\in \Gamma(M,E)$, $0<r\leq1$ and $\psi\in C^{\infty}(M)$ such that $\supp(\psi)\subset B_{r}(x_{0})$.
    For $m+1$, we have:
    $$||D^{m+1}(\psi s)||_{0,B_{1}(x_{0})}\leq ||D^{m}[D,\psi]s||_{0,B_{1}(x_{0})}+||D^{m}\psi Ds||_{0,B_{1}(x_{0})}.
    $$
    Since $$||D^{m}\psi Ds||_{0,B_{1}(x_{0})}\leq C_{m}\cdot w_{m}(\psi)\cdot\left(\sum_{i=0}^{m+1}||D^{i}s||_{0,B_{r}(x_{0})}\right)$$
    and
   \begin{align*}
       ||D^{m}[D,\psi]s||_{0,B_{1}(x_{0})}&\leq \sum_{i=1}^{n}||D^{m}\cdot e_{i}\cdot(\partial_{i}\psi) s||_{0,B_{1}(x_{0})}\\
       &\leq B_{1}\sum_{i=1}^{n}||e_{i}\cdot(\partial_{i}\psi) s||_{m,B_{1}(x_{0})}\\
       &\leq B_{1}B_{2}\sum_{i=1}^{n}||(\partial_{i}\psi )s||_{m,B_{1}(x_{0})}\\
       &\leq B_{1}B_{2}B_{3}\cdot\left(\sum_{i=1}^{n}\sum_{j=0}^{m}||D^{j}(\partial_{i}\psi) s||_{0,B_{1}(x_{0})}\right)\\
       &\leq B_{1}B_{2}B_{3}B_{4}\cdot w_{m+1}(\psi)\cdot\left(\sum_{j=0}^{m}||D^{j}s||_{0,B_{r}(x_{0})}\right),
   \end{align*}
   the first inequality holds because $D$ is the equivariant Dirac operator. The second inequality holds because $D^{m}:L^{2}_{m}(B_{1}(x_{0}),E)\rightarrow L^{2}(B_{1}(x_{0}),E)$ is bounded. 
   
   The reason why the third inequality holds is that $e_{i}:L^{2}_{m}(B_{1}(x_{0}),E)\rightarrow L^{2}_{m}(B_{1}(x_{0}),E)$ is bounded and because of the elliptic estimate of $D$, the fourth inequality hold.
   Thus, there exists $C_{m+1}>0$ such that:$$||D^{m+1}(\psi s)||_{0,B_{1}(x_{0})}\leq C_{m+1}\cdot w_{m}(\psi)\cdot\left(\sum_{i=0}^{m+1}||D^{i}s||_{0,B_{r}(x_{0})}\right)
    $$
    for any $s\in \Gamma(M,E)$, $0<r\leq1$ and $\psi\in C^{\infty}(M)$ satisfying $\supp(\psi)\subset B_{r}(x_{0})$. Let $\psi=f=\exp(-\frac{r}{r-d(x,x_{0})^{2}})$, there exists $C_{1}>0$ such that $w_{k}(\psi)\leq \frac{C_{1}}{r^{k}}$ for any $0<r\leq 1$ and $1\leq k\leq n,k\in\mathbb{N}^{+}$. Then we complete the proof.
\end{proof}
\begin{proof} [Proof of Lemma \ref{c2}]
    Let $K(.,.)\in \Gamma(M\times{M},E\boxtimes E^{*})$ and $0<r\leq 1$. For any $y\in B_{r}(y_{0})$, there exists $u(y)\in E_{y}$ such that $|u(y)|=||K(x,y)||$ and $|k(x,y)s(y)|=||K(x,y)||^{2}$. Since $K(x,y)$ is smooth, we have $u\in C(B_{r}(y_{0}),E)$. There exists $\tilde{u}(y)\in E_{x_{0}}$ for any $y\in B_{r}(y_{0})$ such that $k(x,y)u(y)=||K(x,y)||^{2}\cdot\tilde{u}(y)$. Since $|\tilde{u}(y)|=1$, let $\tilde{u}(y)=\sum_{i=1}^{n}\tilde{u}_{i}(y)\cdot e_{i}$ where $\tilde{u}_{i}\in C(B_{r}(y_{0}))$ and $\{e_{i}\}$ is the orthonormal basis of $E_{x}$. Thus, we have:
    \begin{align*}
        \int_{B_{r}(y_{0})}||k(x,y)||^{2}dy&= \int_{B_{r}(y_{0})}||k(x,y)||^{2}\cdot |\tilde{u}(y)|^{2}dy\\
        &=\sum_{i=1}^{n}\int_{B_{r}(y_{0})}||k(x,y)||^{2}\cdot|\tilde{u}_{i}(y)|^{2}dy\\
        &\leq \sum_{i=1}^{n}\int_{B_{r}(y_{0})}||k(x,y)||^{2}\cdot|\tilde{u}_{i}(y)|dy.
    \end{align*}
    Since $\int_{B_{r}(y_{0})}||k(x,y)||^{2}dy\leq \sum_{i=1}^{n}\int_{B_{r}(y_{0})}||k(x,y)||^{2}\cdot|\tilde{u}_{i}(y)|dy$, there exists $1\leq i_{0}\leq n$ such that
    $$\int_{B_{r}(y_{0})}||k(x,y)||^{2}\cdot|\tilde{u}_{i_{0}}(y)|dy\geq\frac{1}{n}\int_{B_{r}(y_{0})}||k(x,y)||^{2}dy.
    $$
    Let $\hat{u}=\sgn(u_{i_{0}})u$, we have:
    \begin{align*}
        |\int_{B_{r}(y_{0})}k(x,y)\hat{u}(y)dy|&=\left(\sum_{i=1}^{n}\left(\int_{B_{r}(y_{0})}||k(x,y)||^{2}\cdot \sgn(u_{i_{0}})\tilde{u}(y)dy\right)^{2}\right)^{\frac{1}{2}}\\
        &\geq\int_{B_{r}(y_{0})}||k(x,y)||^{2}\cdot|\tilde{u}_{i_{0}}(y)|dy\\
        &\geq\frac{1}{n}\int_{B_{r}(y_{0})}||k(x,y)||^{2}dy.
    \end{align*}
    Therefore, we have:
    \begin{align*}
        \sup_{||s||=1, \supp(s)\subset B_{r}(y_{0})}|\int_{B_{r}(y_{0}}k(x,y)s(y)|&\geq \frac{|\int_{B_{r}(y_{0})}k(x,y)\hat{u}(y)dy|}{||\hat{u}||_{0,B_{r}(y_{0})}}\\
        &=\frac{|\int_{B_{r}(y_{0})}k(x,y)\hat{u}(y)dy|}{(\int_{B_{r}(y_{0})}||k(x,y)||^{2}dy)^{\frac{1}{2}}}\\
        &\geq \frac{1}{n}(\int_{B_{r}(y_{0})}||k(x,y)||^{2}dy)^{\frac{1}{2}}\\
        &=\frac{1}{n}||k(x,y)||_{0,B_{r}(y_{0})}.
    \end{align*}
\end{proof}
\begin{lemma}
\label{a1}
    Let $f$ be holomorphic in $\Omega=\{z\in \mathbb{C}||\im(z)|\leq 2w\}$ where $w>0$ and satisfy: for any $k\in \mathbb{N}^{+}$, there exists $M_{k}>0$ such that $|f(z)z^{k}|\leq M_{k}$ when $z\in \Omega$.
    Then for any $k\in \mathbb{N}^{+}$, there exists $p_{k}>0$ such that $|\hat{f}^{(k)}(s)|\leq p_{k}\exp(-w|s|)$ for any $s\in \mathbb{R}$.
\end{lemma}
\begin{proof}
    Since $$\hat{f}(s)=(2\pi)^{-1}\int_{-\infty}^{+\infty}e^{-isx}f(x)dx$$
    for any  $k\in \mathbb{N}^{+}$, we have:
    $$\hat{f}^{(k)}(s)=\widehat{x^{k}f}(s)=(2\pi)^{-1}\int_{-\infty}^{+\infty}e^{-isx}x^{k}f(x)dx.
    $$
    Let $g(z)=e^{-isz}f(z)z^{k}$, because $f$ is holomorphic in $\Omega=\{z\in \mathbb{C}||\im(z)|\leq 2w\}$,  $g$ is also holomorphic in $\Omega$. Since for any $k\in \mathbb{N}^{+}$, there exists $M_{k}>0$ such that $|f(z)z^{k}|\leq M_{k}$ when $z\in \Omega$, we have $\int_{-\infty}^{+\infty}|g(x+iy)|dx<+\infty$ for any $y\in \mathbb{R}, |y|\leq 2w$ and $\lim_{|z|\rightarrow+\infty}|g(z)|=0$.

    Because $g$ is  holomorphic in $\Omega$ and $\lim_{|z|\rightarrow+\infty}|g(z)|=0$, it follows from Cauchy's theorem that:
    $$\int_{-\infty}^{+\infty}e^{-isx}x^{k}f(x)dx=-\int_{-\infty}^{+\infty}e^{-is(x-iw)}(x-iw)^{k}f(x-iw)dx.
    $$
    Then we have $$|\hat{f}^{(k)}(s)|=\exp(-ws)|(2\pi)^{-1}\int_{-\infty}^{+\infty}e^{-isx}(x-iw)^{k}f(x-iw)dx|\leq c_{k}\exp(-ws)$$
    for some $c_{k}>0$.

    Similarly, $|\hat{f}^{(k)}(s)|\leq b_{k}\exp(ws)$ for some $b_{k}>0$. Therefore, there exist $p_{k}>0$ such that:
    $$|\hat{f}^{(k)}(s)|\leq p_{k}\exp(-w|s|)
    $$
    for any $s\in \mathbb{R}$.
\end{proof}
\begin{lemma}
\label{a2}
   The functions $f(x)=e^{-x^{2}}$, $g(x)=e^{-x^{2}}(\frac{1-e^{-x^{2}}}{x^2})^{\frac{1}{2}}x$ satisfy the conditions of Theorem \ref{g3}.
\end{lemma}
\begin{proof}
    It is clear that $f(x)=e^{-x^{2}}$ satisfies the conditions of Theorem~\ref{g3}. For $g(x)=e^{-\frac{1}{2}x^{2}}(\frac{1-e^{-x^{2}}}{x^{2}})^{\frac{1}{2}}x$, let $u(z)=(\frac{1-e^{-z^{2}}}{z^{2}})$ and $w=|z|^{\frac{1}{2}}\cdot e^{\frac{\arg(z)}{2}}$, and we have $g(z)=e^{-\frac{1}{2}z^{2}}w( u(z))\cdot z$.

    Let $z=a+bi$ where $a,b\in \mathbb{R}$. We have:
    $$u(z)=(\frac{1-e^{-z^{2}}}{z^{2}})=\frac{1}{|z|^{4}}\cdot(1-e^{-(a^{2}-b^{2})}\cdot e^{-2abi})\cdot(a^{2}-b^{2}-2abi).
    $$
    Since $w$ is an analytical branch of the multivalue function $z^{\frac{1}{2}}$ which is holomorphic in $\mathbb{C}\backslash(-\infty,0]$. Since $w(u(z))$ is holomorphic if $u(z)\notin (-\infty,0]$. We will study when $u(z)\leq 0$.
    If $u(z)\leq 0$, we have $\im(|z|^{4}u(z))=0$. 

    If $ab=0$, then $a=0$ or $b=0$. If $a=0$, $u(z)=\frac{e^{b^{2}-1}}{b^{2}}>0$ and if $b=0$ we have $u(z)=\frac{1-e^{-a^{2}}}{a^{2}}>0$.

    When $ab\neq 0$, $\im(|z|^{4}u(z))=0$ implies that 
    \begin{align*}
        1-e^{-(a^{2}-b^{2})}\cos{(2ab)}&=(a^{2}-b^{2})\frac{\sin{(2ab)}}{2ab}\cdot e^{-(a^{2}-b^{2})}\\
        e^{(a^{2}-b^{2})}&=\cos{(2ab)}+(a^{2}-b^{2})\frac{\sin{(2ab)}}{2ab}.
    \end{align*}

    Thus, we have $e^{(a^{2}-b^{2})}\leq 1+|a^{2}-b^{2}|$ which implies that $(a^{2}-b^{2})\leq 0$. Because of this, we claim that $u(z)\leq 0$ implies that $(a^{2}-b^{2})\leq 0$. Since $u(0)=1$, then $w(u(z))$ is  holomorphic in a neighborhood of $0$. Therefore, there exists $w>0$ such that $w(u(z))$ is holomorphic in $\Omega=\{z|~|\im z|\leq 2w\}$. Because of this, $g(z)$ is also holomorphic in $\Omega=\{z|~|\im z|\leq 2w\}$.
\end{proof}

\section{} 
\begin{proof} [Proof of Lemma~\ref{l4}]
Note that we have the identity:
\begin{equation}
\label{1}
    \exp(\frac{1}{2}J^{t}\omega J)=\sum_{p=0}^{\frac{n}{2}}\frac{1}{p!}(\sum_{i<j}w_{ij}J_{i}J_{j})^{p}.
\end{equation}
Because for any $p$, $$\frac{1}{p!}(\sum_{i<j}w_{ij}J_{i}J_{j})^{p}=\sum\omega_{i_{1}i_{2}}w_{i_{3}i_{4}}...\omega_{i_{2p-1}i_{2p}}\cdot J_{i_{1}}...J_{i_{2p}}$$
where $i_{1},...,i_{2p}$  runs over $1,2,3,...,2p$ and $i_{2q-1}<i_{2q}$ for any $1\leq q\leq p, q\in \mathbb{N}^{+}$.

Thus,
\begin{equation}
\label{2}
    \frac{1}{p!}(\sum_{m=1}^{n}c_{k}\omega_{km}J_{m})\cdot(\sum_{i<j}w_{ij}J_{i}J_{j})^{p}=\sum_{m=1}^{n}c_{k}\omega_{km}J_{m}\cdot \omega_{i_{1}i_{2}}w_{i_{3}i_{4}}...\omega_{i_{2p-1}i_{2p}}\cdot J_{i_{1}}...J_{i_{2p}}.
\end{equation}

We claim that only the term satisfying $k,m\notin\{i_{1},..., i_{2p}\}$ will remain in the sum on the right hand side of (\ref{2}).

It is clear that if $m\in \{i_{1},..., i_{2p}\}$ then $$\omega_{km}J_{m}\cdot \omega_{i_{1}i_{2}}w_{i_{3}i_{4}}...\omega_{i_{2p-1}i_{2p}}\cdot J_{i_{1}}...J_{i_{2p}}=0.$$

When $k\in \{i_{1},..., i_{2p}\}$, without loss of generality, let $k=i_{1}$, we need to consider:
$$\omega_{km}J_{m}\cdot \omega_{ki_{2}}\omega_{i_{3}i_{4}}...\omega_{i_{2p-1}i_{2p}}\cdot J_{k}J_{i_{2}}...J_{i_{2p}}$$
and since $m$ runs over $\{1,2,...,n\}$ there exists another term:
$$\omega_{ki_{2}}J_{i_{2}}\cdot \omega_{km}\omega_{i_{3}i_{4}}...\omega_{i_{2p-1}i_{2p}}\cdot J_{k}J_{m}...J_{i_{2p}},$$
so that the two cancel each other. Thus, we have proved our claim.

Let $I_{p}\subset\{1,2,...,n\}$ such that $k\notin I_{p}$ and $|I_{p}|=2p$. We denote:
$$Q_{I_{p}}^{k}=\sum c_{k}\omega_{km}J_{m}\cdot \omega_{i_{1}i_{2}}w_{i_{3}i_{4}}...\omega_{i_{2p-1}i_{2p}}\cdot J_{i_{1}}...J_{i_{2p}}, $$
where $m,i_{1},...,i_{2p}$  runs over $I$ and $i_{2q-1}<i_{2q}$ for any $1\leq q\leq p, q\in \mathbb{N}^{+}$.

After wedging $J_{k}$ on the left of $Q_{I_{p}}^{k}$, we have:
$$J_{k}\wedge Q_{I_{p}}^{k}=\sum c_{k}\omega_{km}J_{k}J_{m}\cdot \omega_{i_{1}i_{2}}w_{i_{3}i_{4}}...\omega_{i_{2p-1}i_{2p}}\cdot J_{i_{1}}...J_{i_{2p}}=c_{k}\Pf(\omega_{I_{p}\cup \{k\}})J^{I_{p}\cup \{k\}}.$$

Thus, we have $Q_{I_{p}}^{k}=c_{k}\epsilon(k,I_{p})\Pf(\omega_{I_{p}\cup \{k\}})J^{I_{p}}$.

Because the equation (\ref{2}) can be written as $\sum_{I_{p}}Q_{I_{p}}^{k}$ where $I_{p}$ runs over all subset of $\{1,2,...,n\}$ with length $2p+1$, we have:
\begin{align*}   (\sum_{k=1}^{n}c_{k}\omega(J)_{k})\exp(\frac{1}{2}J^{t}\omega{J})&=\sum_{p=0}^{\frac{n}{2}}\sum_{k=1}^{n}(\sum_{m=1}^{n}c_{k}\omega_{km}J_{m})\cdot\frac{1}{p!}(\sum_{i<j}w_{ij}J_{i}J_{j})^{p}\\
  &=\sum_{p=0}^{\frac{n}{2}}\sum_{k=1}^{n}\sum_{I_{p}}Q_{I_{p}}^{k}\\
&=\sum_{p=0}^{\frac{n}{2}}\sum_{k=1}^{n}\sum_{I_{p}}c_{k}\epsilon(\{k\},I_{p})\Pf(\omega_{I_{p}\cup\{k\}})J^{I_{p}}\\
&=\sum_{k=1}^{n}\sum_{I}c_{k}\epsilon(\{k\},I)\Pf(\omega_{I\cup\{k\}})J^{I}.
\end{align*}

In the last equation, we remove the restriction on the length of sets, we allow $I$ runs over all subsets of $\{1,2,3...n\}$. Therefore, we complete the proof.
\end{proof}

\section{}
\begin{proof}[Proof of Proposition~\ref{gz}]
Let $\exp:TM\rightarrow M$ be the exponential map and $\alpha\in C^{\infty}_{c}(M\times M)$ satisfy the assumptions: 
\begin{itemize}
    \item $\supp(\alpha)\subset\supp(P)$;
    \item $\exp^{-1}$ is a diffeomorphism in a neighborhood of $\supp(\alpha)$; 
    \item $\alpha=1$  in a neighborhood of the diagonal in $\supp(P)\times\supp(P)$.
\end{itemize}
    For any $x,y\in \supp(\alpha)$, denote the parallel translation from $x$ to $y$ along the geodesic between $x$ and $y$ by $\tau_{E}(x,y)$ and let $\alpha_{E}(x,y)=\alpha(x,y)\cdot \tau_{E}(x,y)$.
Let $\rho\in C^{\infty}_{0}(\mathbb{R}^{n})$ satisfy:
    \begin{itemize}
        \item $\rho(x)=1$, $|x|\leq\frac{1}{2}$ and $\rho(x)\geq 0$;
        \item $\int_{R^{n}}\rho(x)dx=1$.
    \end{itemize}
Thus, for any $x\in \supp(P)$, we have:
    \begin{align*}
        P(\epsilon^{-n}\rho(\epsilon^{-1}\exp^{-1}_{x}(y))\cdot \alpha_{E}(y,x))|_{x}&=\int_{M}k_{P}(x,z)\epsilon^{-n}\rho(\epsilon^{-1}\exp^{-1}_{x}(z))\cdot \alpha_{E}(z,x))dz\\
        &=\int_{T_{x}M}k_{P}(x,\exp_{x}(v))\epsilon^{-n}\rho(\epsilon^{-1}v)\cdot \alpha_{E}(\exp_{x}(v),x))\cdot\det(d\exp_{x})|_{v}dv\\
        &=\int_{T_{x}M}k_{P}(x,\exp_{x}(\epsilon v))\rho(v)\cdot \alpha_{E}(\exp_{x}(\epsilon v),x))\cdot\det(d\exp_{x})|_{\epsilon v}dv.
    \end{align*}
Thus, for any $x\in \supp(P)$, 
\begin{align*}
    &\lim_{\epsilon\rightarrow 0}\Str(P(\epsilon^{-n}\rho(\epsilon^{-1}\exp^{-1}_{x}(y))\cdot \alpha_{E}(y,x))|_{x})\\
    &=\lim_{\epsilon\rightarrow 0}\Str(\int_{T_{x}M}k_{P}(x,\exp_{x}(\epsilon v))\rho(v)\cdot \alpha_{E}(\exp_{x}(\epsilon v),x))\cdot\det(d\exp_{x})|_{\epsilon v}dv)\\
    &=\Str(K_{P}(x,x)).
\end{align*}
Since $\supp(P)$ is compact, we have:
\begin{align*}
    \Str(P)&=\int_{M}\Str(k_{P}(x,x))dx\\
    &=\int_{M}\lim_{\epsilon\rightarrow 0}\Str(P(\epsilon^{-n}\rho(\epsilon^{-1}\exp^{-1}_{x}(y))\cdot \alpha_{E}(y,x))|_{x})\\
    &=\lim_{\epsilon\rightarrow 0}\int_{M}\Str(P(\epsilon^{-n}\rho(\epsilon^{-1}\exp^{-1}_{x}(y))\cdot \alpha_{E}(y,x))|_{x}).
\end{align*}
Let $\hat{\rho}$ be the Fourier transform of $\rho$. Then:
\begin{align*}
    \Str(P)&=\lim_{\epsilon\rightarrow 0}\int_{M}\Str(P(\epsilon^{-n}\rho(\epsilon^{-1}\exp^{-1}_{x}(y))\cdot \alpha_{E}(y,x))|_{x})\\
    &=\lim_{\epsilon\rightarrow 0}(2\pi)^{-n}\int_{T^{*}(M)}\Str(P(e^{i\langle\exp_{x}^{-1}(y),\xi\rangle}\hat{\rho}(\epsilon\xi)\alpha_{E}(y,x))|_{x}d\xi dx\\
    &=\lim_{\epsilon\rightarrow 0}(2\pi)^{-n}\int_{T^{*}(M)}\hat{\rho}(\epsilon\xi)\Str(\sigma(P)(x,\xi))d\xi dx\\
    &=\int_{T^{*}M}\Str(\sigma(P)(x,\xi))d\xi dx.
\end{align*}
Since for any $t>0$, $\int_{T^{*}M}\Str(\sigma_{t}(P)(x,\xi))d\xi dx=\int_{T^{*}M}\Str(\sigma(P)(x,\xi))d\xi dx$, we have:
$$\Str(P)=\int_{T^{*}M}\Str(\sigma_{t}(P)(x,\xi))d\xi dx.$$
Therefore we complete the proof.
\end{proof}
\bibliography{main}{}
\bibliographystyle{amsplain}  

\end{document}